\documentclass[final]{siamart0216}
\usepackage{amsfonts}

\usepackage{amsfonts,amsmath,amssymb}
\usepackage{mathrsfs,mathtools,stmaryrd,wasysym}
\usepackage{enumerate}
\usepackage{xspace,MACROS/mydef} 
\usepackage{esint}
\usepackage{graphicx,psfrag}
\DeclareMathAlphabet{\mathpzc}{OT1}{pzc}{m}{it}

\DeclareMathOperator{\supp}{supp}

\newsiamremark{remark}{Remark}

\newcommand{\dist}{ {\textup{\textsf{d}}}_{z} }
\newcommand{\Sides}{\mathscr{S}}
\newcommand{\Ne}{\mathcal{N}}
\newcommand{\E}{\mathscr{E}}

\newcommand{\TheTitle}{A posteriori error estimates for the Stokes problem with singular sources}
\newcommand{\ShortTitle}{A posteriori error estimates with singular sources}
\newcommand{\TheAuthors}{A. Allendes, E.~Ot\'arola, A. J.~Salgado}

\headers{\ShortTitle}{\TheAuthors}

\title{{\TheTitle}\thanks{AA has been partially supported by CONICYT through FONDECYT project 1170579. EO has been partially supported by CONICYT through FONDECYT project 3160201. AJS has been partially supported by NSF grant DMS-1720213.}}

\author{
  Alejandro Allendes\thanks{Departamento de Matem\'atica, Universidad T\'ecnica Federico Santa Mar\'ia, Valpara\'iso, Chile.
    (\email{alejandro.allendes@usm.cl}, \url{http://aallendes.mat.utfsm.cl/}).}
    \and
  Enrique Ot\'arola\thanks{Departamento de Matem\'atica, Universidad T\'ecnica Federico Santa Mar\'ia, Valpara\'iso, Chile.
    (\email{enrique.otarola@usm.cl}, \url{http://eotarola.mat.utfsm.cl/}).}
  \and
  Abner J.~Salgado\thanks{Department of Mathematics, University of Tennessee, Knoxville, TN 37996, USA.
    (\email{asalgad1@utk.edu}, \url{http://www.math.utk.edu/\string~abnersg})}
}

\ifpdf
\hypersetup{
  pdftitle={\TheTitle},
  pdfauthor={\TheAuthors}
}
\fi

\date{Draft version of \today.}

\begin{document}

\maketitle

\begin{abstract}
We propose a posteriori error estimators for classical low--order inf--sup stable and stabilized finite element approximations of the Stokes problem with singular sources in two and three dimensional Lipschitz, but not necessarily convex, polytopal domains. The designed error estimators are proven to be reliable and locally efficient. 
On the basis of these estimators we design a simple adaptive strategy that yields optimal rates of convergence for the numerical examples that we perform.
\end{abstract}



\section{Introduction}
\label{sec:intro}

The purpose of this work is the design and analysis of a posteriori error estimates for low--order inf--sup stable and stabilized finite element approximations of the Stokes problem
\begin{equation}
\label{eq:StokesStrong}
  \begin{dcases}
      -\Delta \bu + \GRAD p = \bF \delta_z, & \text{in } \Omega, \\
      \DIV \bu = 0, & \text{in } \Omega, \\
      \bu = 0, & \text{on } \partial\Omega,
  \end{dcases}
\end{equation}
where, for $d \in \{2,3\}$, $\Omega$ denotes a bounded polytope of $\R^d$ with Lipschitz boundary, $\delta_z$ corresponds to the Dirac delta supported at the interior point $z \in \Omega$ and $\bF \in \mathbb{R}^d$.

As it is well known, system \eqref{eq:StokesStrong} is one of the simplest systems of equations that describes the motion of an incompressible fluid. Here $\bu$ represents the velocity of the fluid, $p$ the pressure, and  $\bF \delta_z$ is an externally applied force. Notice that, for simplicity, we have taken the viscosity to be equal to one. The first equation represents the conservation of momentum and the second one (incompressibility) the conservation of mass. 

While it is fair to say that the study of approximation techniques for \eqref{eq:StokesStrong} and related models in a standard setting is matured and well understood \cite{Guermond-Ern,MR851383}, recent applications and models have emerged where the motion of a fluid is described by \eqref{eq:StokesStrong} or a small variation of it, but due to the material properties (encoded by the viscosity) or, as is our interest here, the singularity of forces $\bF \delta_z$, the problem must be understood in a completely different setting and rigorous approximation techniques are nonexistent. For instance, \cite{Lacouture2015187} models the motion of active thin structures by using a generalization of \eqref{eq:StokesStrong}, where the right hand side is a linear combination of the terms we have there. The author of this work proposes a numerical scheme but its stability and convergence properties are not investigated. Another instance where a singular force like that of \eqref{eq:StokesStrong} may occur, see \cite{MR3679932} and \cite{BrettElliott,MR3449612}, is in a PDE constrained optimization problem where the state is governed by a standard Stokes problem, but the objective contains a point value of $\bu$. The idea in this problem is that one tries to optimize the flow profile so as to match the velocity at a certain point. If one were to write the optimality conditions for this problem the so-called adjoint variable will be governed by a slight modification of \eqref{eq:StokesStrong} where $z$ is the observation point. Finally, \cite{MR3582412} studies a class of asymptotically Newtonian fluids (Newtonian under large shear rates) under singular forcing. The authors show existence and uniqueness as well as some regularity results. In this respect, our work can be understood as an initial step towards the a posteriori error estimation of such fluids.

The examples presented above justify the need to develop robust numerical methods for the numerical approximation of solutions to \eqref{eq:StokesStrong}, and this is the purpose of this work. The key observation that will allow us to handle the singularity in this problem is that there is a Muckenhoupt weight $\omega$, related to the distance to $z$, such that $\delta_z \in H^{-1}(\omega,\Omega)$. In light of this, we propose to study numerical methods in Muckenhoupt weighted Sobolev spaces. However, this will require us to understand the discrete problem as a generalized saddle point problem \cite{MR972452}, \ie one for which the solution and test spaces do not coincide. In spite of this, we are able to develop a posteriori error estimators.
 
We finally comment that, since $\delta_z$ is very singular, it is not expected for the pair $(\bu,p)$ to have any regularity properties beyond those merely needed for the problem to be well--posed. For this reason a priori error estimates in their natural norms might not carry much value in this setting. It might be possible however, using duality techniques, to obtain error estimates in lower order norms. This will be deferred to a future study.

Our presentation is organized as follows. We set notation in Section \ref{sec:prelim}, where we also recall basic facts about weights and introduce the weighted spaces we shall work with. In Section  \ref{sec:Stokes}, we introduce a saddle point formulation of the Stokes problem \eqref{eq:StokesStrong} and review well--posedness results. Section \ref{sec:fem} presents basic ingredients of finite element methods. Section \ref{sec:inf_sup_stable_a_posteriori} is one of the highlights of our work. There we propose an a posteriori error estimator for inf--sup stable finite element approximations of the Stokes problem \eqref{eq:StokesStrong}; the devised error estimator is proven to be locally efficient and globally reliable. In Section \ref{sec:stabilized_fem} we extend the results obtained in Section \ref{sec:inf_sup_stable_a_posteriori} to the case when stabilized finite element approximations are considered. We conclude, in Section \ref{sec:numerics}, with a series of numerical experiments that illustrate our theory.

\section{Notation and preliminaries}
\label{sec:prelim}

Let us fix the notation and conventions in which we will operate. Throughout this work $d \in \{2,3\}$ and $\Omega\subset\mathbb{R}^d$ is an open and bounded polytopal domain with Lipschitz boundary $\partial\Omega$. Notice that we do not assume that $\Omega$ is convex. If $\mathcal{W}$ and $\mathcal{Z}$ are Banach function spaces, we write $\mathcal{W} \hookrightarrow \mathcal{Z}$ to denote that $\mathcal{W}$ is continuously embedded in $\mathcal{Z}$. We denote by $\mathcal{W}'$ and $\|\cdot\|_{\mathcal{W}}$ the dual and the norm of $\mathcal{W}$, respectively.

For $E \subset \bar\Omega$ of finite Hausdorff $i$-dimension, $i \in \{1,\ldots,d \}$,
we denote its measure by $|E|$. If $E$ is such a set and $f : E \to \R$ we denote its mean value by
\[
 \fint_E f  = \frac{1}{|E|}\int_{E} f .
\]

The relation $a \lesssim b$ indicates that $a \leq C b$, with a constant $C$ that depends neither on $a$, $b$ nor the discretization parameters. The value of $C$ might change at each occurrence.

\subsection{Weights and weighted spaces}
By a weight, we shall mean a locally integrable function $\omega$ on $\mathbb{R}^d$ such that $\omega(x)> 0$ a.e. $x \in \mathbb{R}^d$. A particular class of weights, that will be of importance in our work, is the so--called Muckenhoupt class $A_2$ \cite{MR0293384}: If $\omega$ is a weight, we say that $\omega \in A_2$ if
\begin{equation}
\label{A_2class}
[\omega]_{A_2} := \sup_{B} \left( \fint_{B} \omega \right) \left( \fint_{B} \omega^{-1} \right)  < \infty,
\end{equation}
where the supremum is taken over all balls $B$ in $\R^d$. In what follows, for $\omega \in A_2$, we call $[\omega]_{A_2}$ the \emph{Muckenhoupt characteristic} of $\omega$.

We refer the reader to \cite{MR1800316,MR2491902,NOS3,MR1774162} for the basic facts about Muckenhoupt classes and the ensuing weighted spaces. Here we only mention an example of an $A_2$ weight which will be essential in the analysis presented below. Let $z \in \Omega$ be an interior point of $\Omega$ and, for $\alpha \in \R$, define
\begin{equation}
\label{distance_A2}
\dist^\alpha(x) = |x-z|^{\alpha}.
\end{equation}
We then have that $\dist^\alpha \in A_2$ provided that $\alpha \in (-d,d)$.

It is also important to notice that, in the previous example, since $z \in \Omega$ there is a neighborhood of $\partial\Omega$ where the weight $\dist^{\alpha}$ has no degeneracies or singularities. In fact, it is continuous and strictly positive. Consequently, we have that the weight $\dist^{\alpha}$ belongs to the class $A_2(\Omega)$, introduced in \cite[Definition 2.5]{MR1601373}, and which we define as follows.

\begin{definition}[class $A_2(\Omega)$]
\label{def:ApOmega}
Let $\Omega \subset \R^d$ be a Lipschitz domain. We say that $\omega \in A_2$ belongs to $A_2(\Omega)$ if there is an open set $\calG \subset \Omega$, and positive constants $\varepsilon>0$ and $\omega_l>0$ such that:
\begin{enumerate}[1.]
  \item $\{ x \in \Omega: \mathrm{dist}(x,\partial\Omega)< \varepsilon\} \subset \calG$,
  
  \item $\omega|_{\bar\calG} \in C(\bar\calG)$, and
  
  \item $\omega_l \leq \omega(x)$ for all $x \in \bar\calG$. 
\end{enumerate}
\end{definition}

The fact that $\dist^{\alpha}$ belongs to the restricted class $A_2(\Omega)$ has been shown to be crucial in the analysis of  \cite{OS:17infsup} that guarantees the well--posedness of problem \eqref{eq:StokesStrong} in weighted Sobolev spaces. We will recall these facts in Section~\ref{sec:Stokes}.

For $\alpha \in (-d,d)$ and an open set $E \subseteq \Omega$, we define
\[
 L^2(\dist^{\pm\alpha},E):= \left\{ v \in L^1_{\mathrm{loc}}(E): \| v \|_{L^2(\dist^{\pm\alpha},E)}:= \left( \int_{E} \dist^{\pm\alpha} |v|^2 \right)^{\frac{1}{2}} < \infty \right \}
\]
and
\[
 H^1(\dist^{\pm\alpha},E):= \{ v \in L^2(\dist^{\pm\alpha},E): | \nabla v| \in L^2(\dist^{\pm\alpha},E)\}
\]
with norm 
\begin{equation}
\label{eq:norm}
 \| v \|_{H^1(\dist^{\pm\alpha},E)}:= \left(  \| v \|_{L^2(\dist^{\pm\alpha},E)}^2 +  \| \nabla v \|_{L^2(\dist^{\pm\alpha},E)}^2  \right)^{\frac{1}{2}}.
\end{equation}
We also define $H_0^1(\dist^{\pm\alpha},E)$ as the closure of $C_0^{\infty}(E)$ in $H^1(\dist^{\pm\alpha},E)$. In view of the fact that, for $\alpha \in (-d,d)$, the weight $\dist^{\pm\alpha}$ belongs to $A_2$, we conclude that the spaces $L^2(\dist^{\pm\alpha},E)$ and $H^1(\dist^{\pm\alpha},E)$ are Hilbert \cite[Proposition 2.1.2]{MR1774162} and that smooth functions are dense \cite[Corollary 2.1.6]{MR1774162}; see also \cite[Theorem 1]{MR2491902}. In addition, \cite[Theorem 1.3]{MR643158} guarantees the existence of a weighted Poincar\'e inequality which, in turn, implies that over $H^1_0(\dist^{\pm\alpha},\Omega)$ the seminorm $\| \nabla v \|_{L^2(\dist^{\pm\alpha},\Omega)}$ is an equivalent norm to the one defined in \eqref{eq:norm} for $E = \Omega$. We also introduce the weighted space of vector--valued functions and the norm
\[
 \bH_0^1(\dist^{\pm\alpha},E) = [ H_0^1(\dist^{\pm\alpha},E) ]^d, \quad \| \GRAD \bv \|_{\bL^2(\dist^{\pm\alpha},E)}:= \left( \sum_{i=1}^d \| \nabla v_i \|^2_{L^2(\dist^{\pm\alpha},E)} \right)^{\frac{1}{2}},
\]
where $\bv = (v_1,\dots,v_d)^\intercal$. 


For $\alpha \in (-d,d)$ we also introduce the product spaces 
\begin{equation}
\label{XandY}
\mathcal{X}(E) = \bH^{1}_0(\dist^{\alpha},E) \times L^2(\dist^{\alpha},E)/ \mathbb{R}, \quad
\mathcal{Y}(E) = \bH^{1}_0(\dist^{-\alpha},E) \times L^2(\dist^{-\alpha},E)/ \mathbb{R},
\end{equation}
which we endow with the norms
\begin{equation}
\label{normX}
\| (\bw,r) \|_{\mathcal{X}(E)} = \left( \| \GRAD \bw \|^2_{\bL^2(\dist^{\alpha},E)} + \| r \|^2_{L^2(\dist^{\alpha},E)/ \mathbb{R}} \right)^{\frac{1}{2}}
\end{equation}
and
\begin{equation}
\label{normY}
\| (\bv,q) \|_{\mathcal{Y}(E)} = \left( \| \GRAD \bv \|^2_{\bL^2(\dist^{-\alpha},E)} + \| q \|^2_{L^2(\dist^{-\alpha},E)/ \mathbb{R}} \right)^{\frac{1}{2}},
\end{equation}
respectively. When $E = \Omega$, and in order to simplify the presentation of the material, we write $\mathcal{X} = \mathcal{X}(\Omega)$ and $\mathcal{Y} = \mathcal{Y}(\Omega)$.

The following continuous embedding results will be instrumental in the analysis that follows. 

\begin{proposition}[continuous embeddings I]
\label{pro:ce_I}
Let $E$ be an open subset of $\Omega \subset \mathbb{R}^d$ and $\alpha \in (-d,d)$. Then, we have the following continuous embeddings 
 \begin{equation}
\label{eq:embedding_L1}
\bL^2(\dist^{-\alpha},E) \hookrightarrow \bL^1_{\mathrm{loc}}(E), \qquad \bL^2(\dist^{\alpha},E) \hookrightarrow \bL^1_{\mathrm{loc}}(E).
\end{equation}
\end{proposition}
\begin{proof}
Let $\bv \in \bL^2(\dist^{\alpha},E)$ and $B \subset E$ be a ball. A trivial application of the Cauchy--Schwartz inequality reveals that
\[
 \int_{B} |\bv| =  \int_{B} \dist^{\frac{\alpha}{2}} |\bv| \dist^{-\frac{\alpha}{2}} \leq  \left( \int_{B} \dist^{\alpha} |\bv|^2 \right)^{\frac{1}{2}} \left( \int_{B} \dist^{-\alpha}\right)^{\frac{1}{2}} \lesssim \left( \int_{B} \dist^{\alpha} |\bv|^2 \right)^{\frac{1}{2}},
\]
where, to obtain the last inequality, we have used that, since $\alpha \in (-d,d)$, $\dist^{-\alpha}$ is a weight, \ie a nonnegative and locally integrable function. This yields the continuous embedding $\bL^2(\dist^{\alpha},E) \hookrightarrow \bL^1_{\mathrm{loc}}(E)$. The proof of the continuous embedding $\bL^2(\dist^{-\alpha},E) \hookrightarrow \bL^1_{\mathrm{loc}}(E)$ is similar.
\end{proof}

\begin{proposition}[continuous embeddings II]
\label{pro:ce_II}
Let $E$ be an open subset of $\Omega \subset \mathbb{R}^d$. If $\alpha \in (0,d)$, then we have the following continuous embeddings 
\begin{equation}
\label{eq:embedding_H1}
\bH_0^1(\dist^{-\alpha},E) \hookrightarrow \bH_0^1(E) \hookrightarrow \bH_0^1(\dist^{\alpha},E).
\end{equation}
\end{proposition}
\begin{proof}
The proof follows from the fact that, since $\alpha \in (0,d)$, then $\max_{x \in E} \dist^{\alpha}(x)$ is uniformly bounded. In fact, for $\bv \in \mathbf{C}_0^{\infty}(E)$, we have that
\[
\int_{E} | \GRAD \bv|^2 \leq \max_{x \in E} \dist^{\alpha}(x) \int_{E} \dist^{-\alpha} | \GRAD \bv|^2,
\qquad
\int_{E} \dist^{\alpha} | \GRAD \bv|^2 \leq \max_{x \in E} \dist^{\alpha}(x) \int_{E} | \GRAD \bv|^2.
\]
The embeddings described in \eqref{eq:embedding_H1} now follow from a density argument. This concludes the proof.
\end{proof}

\section{The Stokes problem with Dirac sources}
\label{sec:Stokes}

Having set up the needed functional setting, we now begin with the systematic study of problem \eqref{eq:StokesStrong}. First, we provide a motivation for the use of weights.

\subsection{Motivation}

Let us assume that \eqref{eq:StokesStrong} is posed over the whole space $\R^d$. If that is the case, the results of \cite[Section IV.2]{MR2808162} provide the following asymptotic behavior of the solution $(\bu, p)$ to problem \eqref{eq:StokesStrong} near the point $z$:
\begin{equation}\label{asympt-x0}
| \nabla \bu (x) | \approx |x-z|^{1-d} \quad \textrm{and} \quad | p(x) | \approx |x-z|^{1-d}.
\end{equation}
On the basis of these asymptotic estimates, basic computations reveal that 
\[
\alpha \in (d-2,\infty) \implies
\int_{\Omega} \dist^{\alpha} | \nabla \bu |^2 < \infty, 
\quad
\int_{\Omega} \dist^{\alpha} | p |^2 < \infty.
\]
This heuristic suggests to seek solutions to problem \eqref{eq:StokesStrong} in weighted Sobolev spaces \cite{MR3582412,OS:17infsup}. In what follows we will make these considerations rigorous.
 
\subsection{Saddle point formulation}
 
The motivation of the previous paragraph suggests that we seek for solutions of \eqref{eq:StokesStrong} in the weighted spaces defined in Section~\ref{sec:prelim}.
To accomplish this task, we define the bilinear forms
\begin{equation}
\label{eq:defofforma}
\begin{aligned}
    a &: \bH^1_0(\dist^{\alpha},\Omega) \times \bH^1_0(\dist^{-\alpha},\Omega) \to \R, \\
    a(\bw,\bv) &:= \int_\Omega \nabla \bw : \nabla \bv = \sum_{i=1}^d \int_{\Omega}  \nabla \bw_i \cdot \nabla \bv_i
  \end{aligned}
\end{equation}
and
\begin{equation}
\label{eq:defofformb}
 \begin{aligned}
    b &: \bH^1_0(\dist^{\pm\alpha},\Omega)\times L^2(\dist^{\mp\alpha},\Omega) \to \R, \\
    b(\bv,q) &:=  - \int_{\Omega} q \DIV \bv.
  \end{aligned}
\end{equation}

The weak formulation of problem \eqref{eq:StokesStrong} that we shall consider is: Find $(\bu,p) \in \bH^1_0(\dist^{\alpha},\Omega) \times L^2(\dist^{\alpha},\Omega)/\R$ such that
\begin{equation}
\label{eq:StokesWeak}
  \begin{dcases}
  a(\bu,\bv) + b(\bv,p) =  \langle \bF \delta_z, \bv \rangle,  & \forall \bv \in \bH^1_0(\dist^{-\alpha},\Omega), \\
  b(\bu,q) = 0, &\forall q \in L^2(\dist^{-\alpha},\Omega)/\R,
  \end{dcases}
\end{equation}
where $\langle \cdot, \cdot \rangle$ denotes the duality pairing between $\bH^{1}_0(\dist^{-\alpha},\Omega)'$ and $\bH^{1}_0(\dist^{-\alpha},\Omega)$. We must immediately comment that, in order to guarantee that $\delta_{z} \in H_0^1(\dist^{-\alpha},\Omega)'$, and thus that $\langle \bF \delta_z, \bv \rangle$ is well--defined for $\bv \in \bH^1_0(\dist^{-\alpha},\Omega)$, the parameter $\alpha$ should be restricted to belong to the interval $(d-2,d)$ \cite[Lemma 7.1.3]{KMR}.

We now present an alternative weak formulation for problem \eqref{eq:StokesStrong}. To accomplish this task, we define the bilinear form $c: \mathcal{X} \times \mathcal{Y} \rightarrow \mathbb{R}$ by
\begin{equation}
\label{def:c}
c((\bw,r),(\bv,q)):= a(\bw,\bv) + b(\bv,r) - b(\bw,q)
\end{equation}
with norm
\begin{equation}
\label{eq:c_norm}
 \| c \| = \sup_{ (\boldsymbol{0},0) \neq (\bw,r) \in \mathcal{X} }  \sup_{ (\boldsymbol{0},0) \neq (\bv,q) \in \mathcal{Y} }\frac{c((\bw,r),(\bv,q))}{ \| (\bw,r) \|_{\mathcal{X}}  \|  (\bv,q)\|_{\mathcal{Y}}},
\end{equation}
where the product spaces $\mathcal{X}$ and $\mathcal{Y}$ were defined in \eqref{XandY}.

The aforementioned alternative weak formulation of problem \eqref{eq:StokesStrong} thus reads as follows: Find $(\bu,p) \in \mathcal{X}$
such that
\begin{equation}
\label{eq:StokesWeakAdd}
c((\bu,p),(\bv,q)) =  \langle \bF \delta_z, \bv \rangle \quad \forall (\bv,q) \in \mathcal{Y}.
\end{equation}

It has been recently proved in \cite{OS:17infsup} that, since $\dist^\alpha \in A_2(\Omega)$, problem \eqref{eq:StokesWeakAdd} admits a unique solution $(\bu,p) \in \mathcal{X} = \bH^1_0(\dist^{\alpha},\Omega) \times L^2(\dist^{\alpha},\Omega)/\R$ for $\alpha \in (d-2,d)$. In addition, the following a priori error estimate was also derived in \cite[Theorem 14]{OS:17infsup}:
\begin{equation}
\label{eq:aprioristokes}
  \| \GRAD \bu \|_{\bL^2(\dist^{\alpha},\Omega)} + \| p \|_{L^2(\dist^{\alpha},\Omega)/\R} \lesssim 
   | \bF |  \| \delta_z \|_{\bH_0^1(\dist^{\alpha},\Omega)'}.
\end{equation}
With such a well--posedness result at hand, we can thus invoke a result by Ne\v cas, see \cite[Th\'eor\`eme 6.3.1]{MR0227584}, \cite[Th\'eor\`emes 3.1 et 3.2]{MR0163054}, and \cite[Theorem 2.2 and Corollary 2.1]{NSV:09}, to conclude the existence of a constant $\beta>0$ such that
\begin{multline}
\label{eq:infsup}
  \inf_{ (\boldsymbol{0},0) \neq (\bw,r) \in \mathcal{X} }  \sup_{ (\boldsymbol{0},0) \neq (\bv,q) \in \mathcal{Y} } \frac{c((\bw,r),(\bv,q))}{ \| (\bw,r) \|_{\mathcal{X}}  \|  (\bv,q)\|_{\mathcal{Y}} }
  =  \\
  \inf_{ (\boldsymbol{0},0) \neq (\bv,q) \in \mathcal{Y}  } \sup_{ (\boldsymbol{0},0) \neq (\bw,r) \in \mathcal{X}  } \frac{c((\bw,r),(\bv,q))}{ \| (\bw,r) \|_{\mathcal{X}}  \| (\bv,q) \|_{\mathcal{Y}} } = \beta.  
\end{multline}

\section{Triangulation}
\label{sec:fem}

Having shown the well-posedness of \eqref{eq:StokesWeak}, we can now begin with its numerical approximation and the analysis of the ensuing methods. We first introduce some terminology and a few basic ingredients that will be common to all our methods.

We consider $\T = \{T\}$ to be a conforming partition of $\bar\Omega$ into closed simplices $T$ with size $h_T = \diam(T)$ and define $h_{\T} = \max_{T \in \T} h_T$.  We denote by $\Tr$ the collection of conforming and shape regular meshes that are refinements of an initial mesh $\T_0$ \cite{CiarletBook,Guermond-Ern}.

We denote by $\Sides$ the set of internal ($d-1$)-dimensional interelement boundaries $S$ of $\T$. For $S \in \Sides$, we indicate by $h_S$ the diameter of $S$. If $T\in\T$, we define $\Sides_{T}$ as the subset of $\Sides$ that contains the sides of $T$. For $S \in \Sides$, we set $\Ne_S = \{ T^+, T^-\}$, where $T^+, T^- \in \T$ are such that $S=T^+ \cap T^-$, in other words, $\Ne_S$ denotes the subset of $\T$ that contains the two elements of $\T$ that have $S$ as a side. For $T \in \T$, we define the following \emph{stars} or \emph{patches} associated with the element $T$
\begin{equation}
  \Ne_T := \left\{ T' \in \T : \Sides_T \cap \Sides_{T'} \neq \emptyset \right\}.
  \label{eq:NeTstar}
\end{equation}
and
\begin{equation}
  \label{eq:STstar}
 \mathcal{S} _T := \bigcup_{T' \in \T : T \cap T' \neq \emptyset} T'.
\end{equation}

Having defined our mesh we introduce two classes of finite element approximations, the ensuing finite element schemes, and provide an analysis for them.

\section{Inf--sup stable finite element spaces}
\label{sec:inf_sup_stable_a_posteriori}
In the literature, several finite element approximations have been proposed and analyzed to approximate the solution to the Stokes problem \eqref{eq:StokesWeak} when the forcing term of the momentum equation is not singular; see, for instance, \cite[Section 4]{Guermond-Ern}, \cite[Chapter II]{MR851383}, and references therein. If, given a mesh $\T \in \Tr$, we denote by $\mathbf{V}(\T)$ 
and $\mathcal{P}(\T)$
the finite element spaces that approximate the velocity field and the pressure, respectively, then the following elections are popular:
\begin{enumerate}[(a)]
\item The \emph{mini} element \cite{MR799997}, \cite[Section 4.2.4]{Guermond-Ern}: in this the case, 
\begin{align}
\label{eq:mini_V}
\mathbf{V}(\T) & = \left\{  \bv_{\T} \in \mathbf{C}(\bar \Omega):\ \forall T \in \T, \bv_{\T}|_{T} \in [\mathbb{P}_1(T) \oplus \mathbb{B}(T)]^d \right\} \cap \bH_0^1(\Omega),
\\
\label{eq:mini_P}
\mathcal{P}(\T) & = \left\{  q_{\T} \in  L^2(\Omega)/\R \cap C(\bar \Omega):\ \forall T \in \T, q_{\T}|_{T} \in \mathbb{P}_1(T) \right\}.
\end{align}
$\mathbb{B}(T)$ denotes the space spanned by local bubble functions.

\item The classical Taylor Hood element \cite{hood1974navier}, \cite{MR993474}, \cite[Section 4.2.5]{Guermond-Ern}: in this scenario, 
 \begin{align}
\label{eq:th_V}
\mathbf{V}(\T) & = \left\{  \bv_{\T} \in \mathbf{C}(\bar \Omega): \ \forall T \in \T, \bv_{\T}|_{T} \in \mathbb{P}_2(T)^d \right\} \cap \bH_0^1(\Omega),
\\
\label{eq:th_P}
\mathcal{P}(\T) & = \left\{ q_{\T} \in L^2(\Omega)/\R \cap C(\bar \Omega):\  \forall T \in \T, q_{\T}|_{T} \in \mathbb{P}_1(T) \right\}.
\end{align}
\end{enumerate}

The aforementioned pairs of finite element spaces $(\mathbf{V}(\T), \mathcal{P}(\T))$ satisfy the following compatibility condition \cite[Proposition 4.13]{Guermond-Ern}: there exists a positive constant $\gamma$ such that, for all $\T \in \Tr$,
\begin{equation}
\label{eq:infsup_discrete}
\inf_{ 0 \neq q_{\T} \in \mathcal{P}(\T)} \sup_{ \boldsymbol{0} \neq \bv_{\T} \in \mathbf{V}(\T) } \frac{b(\bv_{\T},q_{\T})}{ \| \nabla \bv_{\T} \|_{\bL^2(\Omega)}  \| q_{\T} \|_{L^2(\Omega)/\R}  } \geq \gamma.
\end{equation}
We refer the reader to \cite[Lemma 4.20 and Lemma 4.24]{Guermond-Ern} for a proof.

We observe now that, since the $\dist^{\alpha}$ is a weight, we thus have, for the elections given by \eqref{eq:mini_V}--\eqref{eq:th_P}, that 
\[
\mathbf{V}(\T) \subset \bH^1_0(\dist^{\alpha},\Omega), \qquad \mathcal{P}(\T) \subset L^2(\dist^{\alpha},\Omega)/\R.
\] 
Consequently, we can consider the following finite element approximation of problem \eqref{eq:StokesWeak}: Find $(\bu_{\T},p_{\T}) \in \mathbf{V}(\T) \times \mathcal{P}(\T)$ such that
\begin{equation}
\label{eq:StokesDiscrete}
  \begin{dcases}
    a(\bu_{\T},\bv_{\T}) + b(\bv_{\T},p_{\T}) = \bF \cdot \bv_{\T}(z),   & \forall \bv_{\T} \in \mathbf{V}(\T), \\
    b(\bu_{\T},q_{\T}) = 0, &\forall q_{\T} \in \mathcal{P}(\T).
  \end{dcases}
\end{equation}
Notice that, since $\bv_{\T} \in \mathbf{C}(\bar\Omega)$, $\langle \bF \delta_z, \bv_{\T} \rangle = \bF \cdot \bv_{\T}(z)$. In addition, since the bilinear form $a$ is coercive on $\bH^1_0(\Omega) \supset \mathbf{V}(\T)$ and the pairs $(\mathbf{V}(\T), \mathcal{P}(\T))$ satisfy \eqref{eq:infsup_discrete}, the system \eqref{eq:StokesDiscrete} has a unique solution for each mesh $\T$.

The main issue, however, is not existence of discrete solutions, but the fact that the stability estimates that might be obtained are not in norms that are compatible with those of \eqref{eq:StokesWeak}. In what follows we will propose a posteriori error estimators in these natural norms and show their reliability and efficiency.

\subsection{A quasi--interpolation operator}
\label{subsec:quasi}
As it is customary in a posteriori error analysis \cite{Verfurth}, in order to derive reliability properties for a proposed a posteriori error estimator it is useful to have at hand a suitable quasi--interpolation operator with optimal approximation properties. Since the interest is to approximate rough functions, namely those without point values, the classical Lagrange interpolation operator cannot be applied. Instead, we consider the quasi--interpolation operator $\Pi_{\T}:\bL^1(\Omega) \rightarrow \mathbf{V}(\T)$ analyzed in \cite{NOS3}. The construction of $\Pi_{\T}$ is inspired in the ideas developed by Cl\'ement \cite{MR0400739}, Scott and Zhang \cite{MR1011446}, and Dur\'an and Lombardi \cite{MR2164092}: it is built on local averages over stars and thus well--defined for functions in $\bL^1(\Omega)$; it also exhibits optimal approximation properties. In what follows, we shall make use of the following estimates of the local interpolation error. To present them, we first define, for $T \in \T$, 
\begin{equation}
\label{eq:DT}
D_T: = \max_{x \in T} |x-z|.
\end{equation}

\begin{proposition}[stability and interpolation estimates]
Let $\alpha \in (-d,d)$, and $T \in \T$. Then, for every $\bv \in \bH_0^1(\dist^{\pm\alpha},\mathcal{S}_T)$, we have the local stability bound
\begin{equation}
\label{eq:local_stability}
 \| \GRAD \Pi_{\T} \bv\|_{\bL^2(\dist^{\pm\alpha},T)} \lesssim \| \GRAD \bv \|_{\bL^2(\dist^{\pm\alpha},\mathcal{S}_T)}
\end{equation}
and the interpolation error estimate
\begin{equation}
\label{eq:interpolation_estimate_1}
\|  \bv - \Pi_{\T} \bv \|_{\bL^2(\dist^{\pm\alpha},T)} \lesssim h_T \| \GRAD \bv \|_{\bL^2(\dist^{\pm\alpha},\mathcal{S}_T)},
\end{equation}
In addition, if $\alpha \in (d-2,d)$, then we have that
\begin{equation}
\label{eq:interpolation_estimate_2}
 \|  \bv - \Pi_{\T} \bv \|_{L^2(T)} \lesssim h_T D_T^{\frac{\alpha}{2}} \| \GRAD \bv \|_{\bL^2(\dist^{-\alpha},\mathcal{S}_T)},
\end{equation}
The hidden constants, in the previous inequalities, are independent of $\bv$, the cell $T$, and the mesh $\T$.
\end{proposition}
\begin{proof}
First, notice that, since $\alpha \in (-d,d)$, we have that $\dist^{\pm\alpha} \in A_2$, which implies that, in view of Proposition~\ref{pro:ce_I}, $\Pi_{\T}$ is well--defined for functions in $\bH_0^1(\dist^{\pm\alpha},\mathcal{S}_T)$. In addition, the theory of \cite{NOS3} can be applied, and thus the local stability bound \eqref{eq:local_stability} follows from \cite[Lemma 5.1]{NOS3} by setting $\omega = \dist^{\pm\alpha}$. The estimate \eqref{eq:interpolation_estimate_1} follows directly from \cite[Theorems 5.2 and 5.3]{NOS3} after setting $\omega=\dist^{\pm\alpha}$.

It remains then to prove \eqref{eq:interpolation_estimate_2}. First, since $\alpha \in (d-2,d) \subset (0,d)$ we have, according to Proposition~\ref{pro:ce_II}, that $\bH_0^1(\dist^{-\alpha},\mathcal{S}_T) \hookrightarrow \bH_0^1(\mathcal{S}_T)$. Thus, an application of \cite[Theorem 5.2 and Theorem 5.3]{NOS3} with $\omega=1$, gives
\[
  \|  \bv - \Pi_{\T} \bv \|_{L^2(T)} \lesssim h_T \| \GRAD \bv \|_{\bL^2(\mathcal{S}_T)} \lesssim h_T D_T^{\frac{\alpha}{2}} \| \GRAD \bv \|_{\bL^2(\dist^{-\alpha},\mathcal{S}_T)},
\]
where, in the last step, we used that $\alpha > 0 $ and that, for all $x \in \mathcal{S}_T$, $\dist^\alpha(x) \leq D_T^\alpha$.
 
This concludes the proof.
\end{proof}

\begin{proposition}[trace interpolation estimate]
Let $\alpha \in (d-2,d)$, $T \in \T$, $S \subset \Sides_T$, and $\bv \in \bH_0^1(\dist^{-\alpha},\mathcal{S}_T)$. Then we have the following interpolation error estimate for the trace
\begin{equation}
\label{eq:interpolation_estimate_trace}
 \|  \bv - \Pi_{\T} \bv \|_{L^2(S)} \lesssim h_T^{\frac{1}{2}} D_T^{\frac{\alpha}{2}} \| \GRAD \bv \|_{\bL^2(\dist^{-\alpha},\mathcal{S}_T)},
\end{equation}
where the hidden constant is independent of $\bv$.
\end{proposition}
\begin{proof}
As a first step, we recall the scaled trace inequality of \cite[Corollary 6.1]{NSV:09}:
\[
 \| v \|_{L^2(S)} \lesssim h_S^{-\frac{1}{2}} \| v \|_{L^2(T)} + h_S^{\frac{1}{2}} \| \GRAD v\|_{L^2(T)} \quad \forall v \in H^1(T),
\]
where $S \in \Sides_{T}$. In view of the continuous embedding $\bH_0^1(\dist^{-\alpha},\mathcal{S}_T) \hookrightarrow \bH_0^1(\mathcal{S}_T)$ that was shown in Proposition \ref{pro:ce_II}, we can apply the previous bound to the function $\bv -\Pi_\T \bv \in \bH^1(\dist^{-\alpha},\mathcal{S}_T)$. This, combined with the interpolation error estimate \eqref{eq:interpolation_estimate_2}, reveal that
\begin{equation}
\label{eq:aux_trace}
  \|  \bv - \Pi_{\T} \bv \|_{L^2(S)} \lesssim h_S^{-\frac{1}{2}} h_T D_T^{\frac{\alpha}{2}} \| \nabla \bv \|_{\bL^2(\dist^{-\alpha},\mathcal{S}_T)} + h_S^{\frac{1}{2}} \| \GRAD (\bv - \Pi_\T \bv)  \|_{\bL^2(T)}.
\end{equation}
To control the second term on the right--hand side of the previous expression, we invoke the stability of the quasi--interpolation operator $\Pi_{\T}$ in $H^1$ \cite[Lemma 5.1]{NOS3} and, once again, the fact that $\alpha > 0$ to obtain that
\begin{equation}
\label{eq:astuta_estimate}
\| \GRAD (\bv - \Pi_{\T} \bv)  \|_{\bL^2(T)} \lesssim \| \GRAD \bv  \|_{\bL^2(\mathcal{S}_T)} \lesssim D_T^{\frac{\alpha}{2}} \| \GRAD \bv  \|_{\bL^2(\dist^{-\alpha},\mathcal{S}_T)}.
\end{equation}
Replacing the previous estimate into \eqref{eq:aux_trace} combined with the fact that $h_T \approx |T|/|S| \approx h_S$ yield \eqref{eq:interpolation_estimate_trace} and concludes the proof.
\end{proof}

With the operator $\Pi_{\T}$ at hand, and following \cite[Section 4.10]{Verfurth}, we define the following restriction operator
\begin{equation}
\label{eq:operator_Q}
  \mathcal{Q}_{\T}: \mathcal{Y} \longrightarrow \mathbf{V}(\T) \times \mathcal{P}(\T), 
  \qquad
  (\bv,q) \longmapsto (\Pi_{\T} \bv, 0),
\end{equation}
where $\Pi_\T \bv = (\Pi_\T v_1,\ldots,\Pi_\T v_d)^\intercal$.

\subsection{A posteriori error estimators}
\label{sec:a_posteriori}

We are now ready to introduce an a posteriori error estimator for the finite element approximation \eqref{eq:StokesDiscrete}, on the basis of the discrete pairs $(\mathbf{V}(\T),\mathcal{P}(\T))$ given as in \eqref{eq:mini_V}--\eqref{eq:mini_P} or \eqref{eq:th_V}--\eqref{eq:th_P}, of the Stokes problem \eqref{eq:StokesWeak}. To accomplish this task, we first recall the definition of the local distance $D_T$ given as in \eqref{eq:DT}. We thus define, for $\alpha \in (d-2,d)$ and $T \in \T$, the \emph{element error indicators}
\begin{multline}
\E_{\alpha}(\bu_{\T},p_{\T};T):= \left( h_T^2D_T^{\alpha}   \|  \Delta \bu_{\T} - \nabla p_{\T} \|_{\bL^2(T)}^2 +  \|  \DIV \bu_{\T} \|_{L^2(\dist^\alpha, T)}^2 \right.
\\
\left. + h_T D_T^{\alpha} \| \llbracket (\nabla \bu_{\T}  - p_{\T} \mathbf{I}) \cdot \boldsymbol{\nu}\rrbracket \|_{\bL^2(\partial T \setminus \partial \Omega)}^{2} + h_{T}^{\alpha + 2 -d} | \bF |^2 \chi(\{z \in T \}) \right)^{\frac{1}{2}},
\label{eq:local_indicator}
\end{multline}
where $(\bu_{\T},p_{\T})$ denotes the solution to the discrete problem \eqref{eq:StokesDiscrete}, $\mathbf{I} \in \R^{d \times d}$ denotes the identity matrix, and the function $\chi(\{z \in T\})$ equals one if $z \in T$ and zero otherwise. Here we must recall that we consider our elements $T$ to be closed sets. For a discrete tensor valued function $\bW_{\T}$, we denote by $\llbracket \bW_{\T} \cdot \boldsymbol{\nu}\rrbracket$ the jump or interelement residual, which is defined, on the internal side $S \in \Sides$ shared by the distinct elements $T^+$, $T^{-} \in \mathcal{N}_S$, by
\begin{equation}
\label{eq:jump}
 \llbracket \bW_{\T} \cdot \boldsymbol{\nu} \rrbracket =  \bW_{\T}|_{T^+}\cdot \boldsymbol{\nu}^+ +  \bW_{\T}|_{T^-} \cdot \boldsymbol{\nu}^-.
\end{equation}
Here $\boldsymbol{\nu}^+, \boldsymbol{\nu}^-$ are unit normals on $S$ pointing towards $T^+$, $T^{-}$, respectively. The \emph{error estimator} is thus defined as
\begin{equation}
\E_{\alpha}(\bu_{\T},p_{\T};\T): = \left( \sum_{T \in \T} \E^2_{\alpha}(\bu_{\T},p_{\T};T) \right)^{\frac{1}{2}}.
\label{eq:global_estimator}
\end{equation}

\subsection{Error and residual}

An important ingredient in the analysis that we will provide below is the so--called \emph{residual}. Let $(\bu,p) \in \mathcal{X}$ and $(\bu_{\T},p_{\T}) \in \mathbf{V}(\T) \times \mathcal{P}(\T)$ denote the unique solutions to problems \eqref{eq:StokesWeak} and \eqref{eq:StokesDiscrete}, respectively. In order to obtain information about the error $(\be_\bu, e_p) = (\bu - \bu_{\T},p - p_{\T}) \in \mathcal{X}$,
we define the residual $\mathcal{R} = \mathcal{R}(\bu_{\T},p_{\T}) \in \mathcal{Y}'$ as follows:
\begin{equation}
\label{residual}
  \langle \mathcal{R}, (\bv,q) \rangle_{\mathcal{Y}' \times \mathcal{Y}} = \langle \bF \delta_z, \bv \rangle - c((\bu_{\T},p_{\T}), (\bv,q)),
\end{equation}
where $\langle \cdot, \cdot \rangle$ denotes the duality pairing between $\bH^{1}_0(\dist^{-\alpha},\Omega)'$ and $\bH^{1}_0(\dist^{-\alpha},\Omega)$ and the bilinear form $c$ is defined in \eqref{def:c}. Notice that the residual $\mathcal{R}$ depends only on the data and the approximate solution $(\bu_{\T},p_{\T})$ and is related to the error function by the relation
\begin{equation}
\label{eq:relation_residual_error}
\langle \mathcal{R}, (\bv,q) \rangle_{\mathcal{Y}',\mathcal{Y}} = c( (\be_\bu, e_p),(\bv,q) ) \quad \forall (\bv,q) \in \mathcal{Y}. 
\end{equation}

The following result shows that the $\mathcal{Y}'$--norm of $\mathcal{R}$ is equivalent to the error.

\begin{lemma}[abstract a posteriori error bounds]
\label{le:abstract}
If $\alpha \in (d-2,d)$, then
\begin{equation}
\label{eq:a_posteriori_dual}
 \beta \| (\be_{\bu}, e_{p}) \|_{\mathcal{X}} \leq \|  \mathcal{R} \|_{\mathcal{Y'}} \leq \| c \|\| (\be_{\bu}, e_{p}) \|_{\mathcal{X}},
\end{equation}
where $\beta$ and $\| c \|$ are the inf--sup and continuity constants of the bilinear form $c$, which are defined in \eqref{eq:c_norm} and \eqref{eq:infsup}, respectively, and verify $0 < \beta \leq \| c \|$. 
\end{lemma}
\begin{proof}
An application of the inf--sup condition \eqref{eq:infsup}, combined with the definition of the residual $\mathcal{R}$ and the relation \eqref{eq:relation_residual_error}, imply that
\begin{align*}
\beta \| (\be_{\bu}, e_{p}) \|_{\mathcal{X}} &\leq  \sup_{ (\boldsymbol{0},0) \neq (\bv,q) \in \mathcal{Y} } \frac{c((\be_{\bu}, e_{p}),(\bv,q))}{\| (\bv,q) \|_{\mathcal{Y}}}  \\
&= \sup_{(\boldsymbol{0},0) \neq (\bv,q) \in \mathcal{Y} } \frac{\langle \mathcal{R},(\bv,q) \rangle_{\mathcal{Y}',\mathcal{Y}} }{\| (\bv,q) \|_{\mathcal{Y}}} =  \|  \mathcal{R} \|_{\mathcal{Y'}}.
\end{align*}
On the other hand, 
\begin{equation}
 \|  \mathcal{R} \|_{\mathcal{Y'}} = \sup_{ (\boldsymbol{0},0) \neq (\bv,q) \in \mathcal{Y} } \frac{c((\be_{\bu}, e_{p}),(\bv,q))}{\| (\bv,q) \|_{\mathcal{Y}}} \leq \| c \| \| (\be_{\bu}, e_{p}) \|_{\mathcal{X}}.
\end{equation}
Estimate \eqref{eq:a_posteriori_dual} follows by collecting these two bounds.
\end{proof}

\subsubsection{Reliability}

In what follows we obtain a global reliability property for the a posteriori error estimator \eqref{eq:global_estimator}.

\begin{theorem}[reliability]\label{th:reliability}
Let $(\bu,p) \in \mathcal{X}$ be the unique solution to problem \eqref{eq:StokesWeak} and $(\bu_{\T},p_{\T})\in \mathbf{V}(\T) \times \mathcal{P}(\T)$ its finite element approximation given as the solution to \eqref{eq:StokesDiscrete}. If $\alpha \in (d-2,d)$, then
\begin{equation}
\label{eq:global_upper_bound}
\| \GRAD \be_\bu \|_{\bL^2(\dist^{\alpha},\Omega)} + \| e_p \|_{L^2(\dist^{\alpha},\Omega)} \lesssim \E_{\alpha}(\bu_{\T},p_{\T}; \T),
\end{equation}
where the hidden constant is independent of the continuous and discrete solutions, the size of the elements in the mesh $\T$ and $\#\T$.
\end{theorem}
\begin{proof}
In view of the first bound in \eqref{eq:a_posteriori_dual}, we conclude that, to bound the $\mathcal{X}$--norm of the error, it suffices to control the dual norm $\| \mathcal{R} \|_{\mathcal{Y}'}$. To accomplish this task, we proceed as follows. Let $(\bv,q) \in \mathcal{Y}$ be arbitrary. Applying a standard integration by parts argument to \eqref{residual} yields 
\begin{multline}
\langle  \mathcal{R},  (\bv,q) \rangle_{\mathcal{Y}',\mathcal{Y}} = \langle \bF  \delta_z, \bv \rangle - c((\bu_{\T},p_{\T}), (\bv,q)) =   \langle \bF \delta_z, \bv \rangle 
\\
+ \sum_{T \in \T} \int_{T} (\Delta \bu_{\T} - \nabla p_{\T}) \cdot  \bv - \sum_{S \in \Sides} \int_{S}  \llbracket (\nabla \bu_{\T}  - p_{\T} \mathbf{I}) \cdot \boldsymbol{\nu}\rrbracket \cdot \bv -  \sum_{T \in \T} \int_{T} q \DIV \bu_{\T}.
\label{eq:reliability_1} 
\end{multline}

Next we observe that, since $\mathbf{V}(\T) \times \mathcal{P}(\T) \subset \mathcal{Y}$, we can invoke Galerkin orthogonality to
conclude that, for all $(\bv_{\T},q_{\T}) \in \mathbf{V}(\T) \times \mathcal{P}(\T)$,
\begin{equation}
\label{eq:GO}
0= c( (\bu - \bu_{\T}, p - p_{\T}), (\bv_{\T},q_{\T}) )  =  \langle \bF \delta_z, \bv_{\T} \rangle - c( (\bu_{\T}, p_{\T}), (\bv_{\T},q_{\T})).
\end{equation}
We now invoke the restriction operator $\mathcal{Q}_{\T}$, defined in \eqref{eq:operator_Q}, and set $(\bv_{\T},0) = \mathcal{Q}_{\T} (\bv,q)$ in \eqref{eq:GO}. By replacing the obtained relation into \eqref{eq:reliability_1} we arrive at
\begin{multline}\label{eq:residuo}
\langle  \mathcal{R},  (\bv ,q) \rangle_{\mathcal{Y}',\mathcal{Y}} =  \langle \bF\delta_z, \bv - \bv_{\T} \rangle + \sum_{T \in \T} \int_{T} (\Delta \bu_{\T} - \nabla p_{\T}) \cdot  ( \bv - \bv_{\T} )
\\
- \sum_{S \in \Sides} \int_{S}  \llbracket (\nabla \bu_{\T}  - p_{\T} \mathbf{I}) \cdot \boldsymbol{\nu}\rrbracket \cdot (\bv - \bv_{\T}) -  \sum_{T \in \T} \int_{T} q \DIV \bu_{\T} =: \textrm{I} + \textrm{II} - \textrm{III} - \textrm{IV}.
\end{multline}
In what follows we proceed to control each term separately. 

To bound $\mathrm{II}$, we invoke the interpolation error estimate \eqref{eq:interpolation_estimate_2} and conclude that
\begin{equation}
\label{eq:II}
\textrm{II} 
\lesssim \sum_{T \in \T} h_T D_T^{\frac{\alpha}{2}}  \| \Delta \bu_{\T} - \nabla p_{\T}  \|_{\bL^2(T)} \| \GRAD \bv \|_{\bL^2(\dist^{-\alpha},\mathcal{S}_T)}.
\end{equation}

We now proceed to control the term $\mathrm{III}$. To accomplish this task, we apply the estimate \eqref{eq:interpolation_estimate_trace} and arrive at 
\begin{equation}
\label{eq:III}
\textrm{III} \lesssim \sum_{S \in \Sides} h_T^{\frac{1}{2}}D_T^{\frac{\alpha}{2}}\| \llbracket (\nabla \bu_{\T}  - p_{\T} \mathbf{I})\cdot \boldsymbol{\nu}\rrbracket  \|_{\bL^2(S)} \| \GRAD \bv \|_{\bL^2(\dist^{-\alpha},\mathcal{S}_T)}.
\end{equation}

The control of the term $\textrm{IV}$ follows from a simple application of the Cauchy--Schwartz
inequality. In fact, we have that
\begin{equation}
\label{eq:IV}
%
\textrm{II} \lesssim \sum_{T \in \T} \| \DIV \bu_{\T} \|_{L^2(\dist^\alpha, T)}\| q \|_{L^2(\dist^{-\alpha},T)}.
\end{equation}

Since $\bv -\bv_{\T} \in \bH_0^1(\dist^{-\alpha},\Omega)$, we control the term $\mathrm{I}$ by using the estimate of \cite[Theorem 4.7]{AGM} followed by the interpolation error estimate \eqref{eq:interpolation_estimate_1} and the local stability bound \eqref{eq:local_stability}. These arguments allow us to conclude that
\begin{equation}
\label{eq:I}
\begin{aligned}
\langle \bF \delta_z, \bv  -\bv_{\T}\rangle &\lesssim |\bF|h_T^{\frac{\alpha}{2}-\frac{d}{2}} \| \bv - \bv_{\T}\|_{\bL^2(\dist^{-\alpha},T)} \\ &+ |\bF|h_T^{\frac{\alpha}{2}+1 - \frac{d}{2}} \| \GRAD(\bv - \bv_{\T})\|_{\bL^2(\dist^{-\alpha},T)} \\
& \lesssim |\bF| h_T^{\frac{\alpha}{2}+1-\frac{d}{2}} \| \GRAD \bv \|_{\bL^2(\dist^{-\alpha},\mathcal{S}_T)}.
\end{aligned}
\end{equation}

Finally, by gathering the estimates for the terms $\mathrm{I}$, $\mathrm{II}$, $\mathrm{III}$, and $\mathrm{IV}$, obtained in \eqref{eq:II}--\eqref{eq:I}, and resorting to the finite overlapping property of stars we arrive at the global upper bound \eqref{eq:global_upper_bound} and conclude the proof.
\end{proof}

\subsubsection{Local efficiency}
To derive efficiency properties of the local error indicator $\E_{\alpha}(\bu_{\T},p_{\T};T)$, defined in \eqref{eq:local_indicator}, we utilize the standard residual estimation techniques developed in references \cite{MR993474,Verfurth} but on the basis of suitable bubble functions, whose construction we owe to \cite[Section 5.2]{AGM} and proceed to describe in what follows. 

Given $T \in \T$, we first introduce a bubble function $\varphi_T$ that satisfies the following properties: 
$0 \leq \varphi_T \leq 1$,
\begin{equation}
\label{eq:bubble_T}
\varphi_T(z) = 0, \qquad |T| \lesssim \int_T \varphi_T, \qquad \| \GRAD \varphi_T \|_{L^{\infty}(R_T)} \lesssim h_T^{-1},
\end{equation}
and there exists a simplex $T^{*} \subset T$ such that $R_{T}:= \supp(\varphi_T) \subset T^{*}$. Notice that, since $\varphi_T$ satisfies \eqref{eq:bubble_T}, we have that
\begin{equation}
\label{eq:aux_bubble_T}
 \| \theta \|_{L^2(R_T)} \lesssim \| \varphi_T^{\frac{1}{2}} \theta \|_{L^2(R_T)} \quad \forall \theta \in \mathbb{P}_{2}(R_T),
\end{equation}
where $R_T = \supp(\varphi_T)$.

Second, given $S \in \Sides$, we introduce a bubble function $\varphi_S$ that satisfies the following properties: $0 \leq \varphi_S \leq 1$,
\begin{equation}
\label{eq:bubble_S}
\varphi_S(z) = 0, \qquad |S| \lesssim \int_S \varphi_S, \qquad \| \GRAD \varphi_S \|_{ L^{\infty}(R_{S} ) } \lesssim h_T^{-1/2}|S|^{1/2},
\end{equation}
and $R_S:= \supp(\varphi_S)$ is such that, if $\mathcal{N}_{S} = \{ T, T' \}$, there are simplices $T_{*} \subset T$ and $T_{*}' \subset T'$ such that $R_S \subset T_{*} \cup T_{*}' \subset \mathcal{N}_{S}$.

The following estimates that involve the bubble functions $\varphi_T$ and $\varphi_S$ are instrumental in the efficiency analysis that we will perform.

\begin{proposition}[estimates for bubble functions]
Let $T \in \T$ and $\varphi_T$ be the bubble function that satisfies \eqref{eq:bubble_T}. If $\alpha \in (0,d)$, then
\begin{equation}
\label{eq:chi_bubble_T}
h_T \| \GRAD (\theta \varphi_T) \|_{L^2(\dist^{-\alpha},T)} \lesssim D_T^{-\frac{\alpha}{2}} \| \theta \|_{L^2(T)} \quad \forall \theta \in \mathbb{P}_{2}(T).
\end{equation}
Let $S \in \Sides$ and $\varphi_S$ be the bubble function that satisfies \eqref{eq:bubble_S}. If $\alpha \in (0,d)$, then
\begin{equation}
\label{eq:chi_bubble_S}
h_T^{\frac{1}{2}} \| \GRAD (\theta \varphi_S) \|_{L^2(\dist^{-\alpha},\mathcal{N}_S)} \lesssim D_T^{-\frac{\alpha}{2}} \| \theta \|_{L^2(S)} \quad \forall \theta \in \mathbb{P}_3(S),
\end{equation}
where $\theta$ is extended to $\mathcal{N}_S$ as a constant along the direction of one side of each element of $\T$ contained in $\mathcal{N}_S$.
\end{proposition}
\begin{proof}
See \cite[Lemma 5.2]{AGM}.
\end{proof}

The following result provides a local estimate for the residual $\mathcal{R}$.

\begin{lemma}[local dual norm]
Let $G$ be a subdomain of $\Omega$. If $\alpha \in (d-2,d)$, then
\begin{equation}
\label{eq:residual_lower}
\|  \mathcal{R} \|_{\mathcal{Y'}(G)} \lesssim \| (\be_{\bu}, e_{p}) \|_{\mathcal{X}(G)},
\end{equation}
where the hidden constant is independent of $(\be_{\bu},e_p)$.
\end{lemma}
\begin{proof}
Let $(\bv,q) \in \mathcal{Y}(G)$.
The extension of $\bv$ and $q$ by zero to $\Omega \setminus G$ yield functions $(\tilde \bv, \tilde q) \in \mathcal{Y}$.
We thus have that
\[
 \langle \mathcal{R}, (\tilde \bv, \tilde q) \rangle_{\mathcal{Y}',\mathcal{Y}} = c( (\be_{\bu},e_p),(\tilde \bv, \tilde q) ) \lesssim \|(\be_{\bu},e_p) \|_{\mathcal{X}(G)} \| (\bv,q) \|_{\mathcal{Y}(G)}.
\]
Consequently \eqref{eq:residual_lower} follows. This concludes the proof.
\end{proof}

With all these ingredients at hand, we are ready to derive the local efficiency properties of the local error indicator $\E_{\alpha}(\bu_{\T},p_{\T};T)$. 

\begin{theorem}[local efficiency]\label{Th:efficiency}
Let  $(\bu,p) \in \mathcal{X}$ be the unique solution to problem \eqref{eq:StokesWeak} and $(\bu_{\T},p_{\T}) \in \mathbf{V}(\T) \times \mathcal{P}(\T)$ its finite element approximation given as the solution to \eqref{eq:StokesDiscrete}. If $\alpha \in (d-2,d)$, then
\begin{equation}
\label{eq:local_lower_bound}
\E^2_{\alpha}(\bu_{\T},p_{\T}; T) \lesssim \| \GRAD \be_{\bu} \|^2_{\bL^2(\dist^{\alpha},\mathcal{N}_{T})} + \|  e_{p} \|^2_{L^2(\dist^{\alpha},\mathcal{N}_{T})},
\end{equation}
where the hidden constant is independent of the continuous and discrete solutions, the size of the elements in the mesh $\T$ and $\#\T$.
 \end{theorem}
\begin{proof}
We estimate each contribution in \eqref{eq:local_indicator} separately. 

We begin the proof by bounding, for $T \in \T$, the term 
$h_T^2D_T^{\alpha} \|  \Delta \bu_{\T} - \nabla p_{\T} \|_{\bL^2(T)}^2$.  Define $\bphi_T:= \varphi_T (\Delta \bu_{\T} - \nabla p_{\T})$ and invoke \eqref{eq:aux_bubble_T} to conclude that
\begin{equation}
\label{eq:eff_first_step}
 \|  \Delta \bu_{\T} - \nabla p_{\T} \|^2_{\bL^2(T)} \lesssim \int_{R_T} | \Delta \bu_{\T} - \nabla p_{\T} |^2 \varphi_T \leq \int_{T} ( \Delta \bu_{\T} - \nabla p_{\T} ) \cdot \bphi_T.
\end{equation}
We now consider the relation \eqref{eq:reliability_1} with $(\bv,q) = (\bphi_T,0)$ and observe that $\bphi_T(z) = \varphi_T(z) (\Delta \bu_{\T} - \nabla p_{\T})(z) = 0$. This allows us to conclude that
\begin{equation}
\begin{aligned}
\label{eq:blah}
 \int_{T} ( \Delta \bu_{\T} - \nabla p_{\T} ) \cdot \bphi_T  & = \langle \mathcal{R},(\bphi_T,0) \rangle_{\mathcal{Y}',\mathcal{Y}} = c((\be_{\bu},e_p ), (\bphi_T,0))
\\
& \lesssim \left( \| \GRAD \be_{\bu} \|^2_{\bL^2(\dist^{\alpha},T)} + \| e_{p} \|^2_{L^2(\dist^{\alpha},T)} \right)^{\frac{1}{2}}  \| \GRAD \bphi_T  \|_{\bL^2(\dist^{-\alpha},T)}.
\end{aligned}
\end{equation}
We now recall that $\bphi_T = \varphi_T (\Delta \bu_{\T} - \nabla p_{\T})$ and utilize \eqref{eq:chi_bubble_T} to conclude that
\[
 \| \GRAD \bphi_T  \|_{\bL^2(\dist^{-\alpha},T)} \lesssim h_T^{-1}D_T^{-\frac{\alpha}{2}} \| \Delta \bu_{\T} - \nabla p_{\T}  \|_{\bL^2(T)}.
\]
Replacing this estimate into \eqref{eq:blah}, and the obtained one in \eqref{eq:eff_first_step}, allow us to write
\begin{equation}
\label{eq:efficiency_first_term}
 h_T^2 D_T^{\alpha}\|  \Delta \bu_{\T} - \nabla p_{\T} \|^2_{\bL^2(T)} \lesssim \| \GRAD \be_{\bu} \|^2_{\bL^2(\dist^{\alpha},T)} + \| e_{p} \|^2_{L^2(\dist^{\alpha},T)}.
\end{equation}

Let $T \in \T$ and $S$ be a side of $T$. In what follows we control the jump term $h_T D_T^{\alpha} \| \llbracket (\nabla \bu_{\T}  - p_{\T} \mathbf{I})\cdot \boldsymbol{\nu}\rrbracket \|^2_{\bL^2(\partial T \setminus \partial \Omega)}$ in \eqref{eq:local_indicator}. To accomplish this task, we proceed by using similar arguments to the ones that lead to \eqref{eq:efficiency_first_term} but now utilizing the bubble function $\varphi_S$. In fact, the use of properties \eqref{eq:bubble_S} yields
\[
 \| \llbracket (\nabla \bu_{\T}  - p_{\T} \mathbf{I})\cdot \boldsymbol{\nu}\rrbracket \|^2_{\bL^2(S)} \lesssim \int_S  |\llbracket (\nabla \bu_{\T}  - p_{\T} \mathbf{I})\cdot \boldsymbol{\nu}\rrbracket|^2 \varphi_S = \int_S  \llbracket (\nabla \bu_{\T}  - p_{\T} \mathbf{I}) \cdot \boldsymbol{\nu}\rrbracket \cdot \bphi_S,
\]
where $\bphi_S:= \varphi_S \llbracket (\nabla \bu_{\T}  - p_{\T} \mathbf{I})\cdot \boldsymbol{\nu} \rrbracket$. Now, set $(\bv,q) = (\bphi_S,0)$ in \eqref{eq:reliability_1}, and use that $\phi_S(z) = 0$ and that $R_S = \supp(\phi_S) \subset T_* \cup T_*' \subset \mathcal{N}_S$, to conclude that
\begin{multline*}
\int_{S}  \llbracket (\nabla \bu_{\T}  - p_{\T} \mathbf{I})\cdot \boldsymbol{\nu}\rrbracket  \cdot \bphi_S  = \sum_{T \in \mathcal{N}_S} \int_{T} (\Delta \bu_{\T} - \nabla p_{\T}) \cdot \bphi_S - \langle \mathcal{R},(\bphi_S,0)\rangle_{\mathcal{Y}',\mathcal{Y}}
\\
= \sum_{T \in \mathcal{N}_S} \int_{T} (\Delta \bu_{\T} - \nabla p_{\T}) \cdot \bphi_S - c((\be_\bu ,e_p ), (\bphi_S,0))
\\
\leq \sum_{T \in \mathcal{N}_S} \| \Delta \bu_{\T} - \nabla p_{\T} \|_{\bL^2(T)} \| \bphi_S \|_{\bL^2(T)}
\\
+ \sum_{T \in \mathcal{N}_S}
\left( \| \GRAD \be_{\bu} \|^2_{\bL^2(\dist^{\alpha},T)} + \| e_{p} \|^2_{L^2(\dist^{\alpha},T)} \right)^{\frac{1}{2}}  \| \GRAD \bphi_S  \|_{\bL^2(\dist^{-\alpha},T)}.
\end{multline*}
The control of the first term on the right--hand side of the previous expression follows from the fact that $ \| \bphi_S \|_{\bL^2(T)} \approx |T|^{\frac{1}{2}} |S|^{-\frac{1}{2}}\| \bphi_S \|_{\bL^2(S)}$ while the bound of the second term follows from \eqref{eq:chi_bubble_S}. These arguments reveal that
\begin{multline}
 \int_{S}  \llbracket (\nabla \bu_{\T}  - p_{\T} \mathbf{I}) \cdot \boldsymbol{\nu}\rrbracket \cdot \bphi_S \lesssim \sum_{T \in \mathcal{N}_S} \| \Delta \bu_{\T} - \nabla p_{\T} \|_{\bL^2(T)} |T|^{\frac{1}{2}} |S|^{-\frac{1}{2}}\| \bphi_S \|_{\bL^2(S)} 
 \\
 + \sum_{T \in \mathcal{N}_S} \left( \| \GRAD \be_{\bu} \|^2_{\bL^2(\dist^{\alpha},T)} + \| e_{p} \|^2_{L^2(\dist^{\alpha},T)} \right)^{\frac{1}{2}} D_T^{-\frac{\alpha}{2}} h_T^{-\frac{1}{2}} \| \bphi_S \|_{\bL^2(S)},
\end{multline}
which, in view of \eqref{eq:efficiency_first_term}, $|T|/|S| \approx h_T$, and $\bphi_S = \varphi_S \llbracket (\nabla \bu_{\T}  - p_{\T} \mathbf{I}) \cdot \boldsymbol{\nu}\rrbracket$ imply that
\begin{equation}
\label{eq:efficiency_second_term}
 h_T D_{T}^{\alpha}\| \llbracket (\nabla \bu_{\T}  - p_{\T} \mathbf{I})\cdot \boldsymbol{\nu}\rrbracket \|^2_{\bL^2(S)} \lesssim \sum_{T' \subset \mathcal{N}_S } \left( \| \GRAD \be_{\bu} \|^2_{\bL^2(\dist^{\alpha},T')} + \| e_{p} \|^2_{L^2(\dist^{\alpha},T')} \right).
\end{equation}

The control of the term $\| \DIV \bu_{\T} \|_{L^2(\dist^{\alpha},T)}^2$ follows easily from the mass conservation equation, that reads $\DIV \bu = 0$. In fact, for $T \in \T$, we have that
\begin{equation}
\label{eq:efficiency_third_term}
 \| \DIV \bu_{\T} \|_{L^2(\dist^{\alpha},T)}^2 = \| \DIV \be_{\bu} \|_{L^2(\dist^{\alpha},T)}^2 \lesssim \| \GRAD \be_{\bu}\|^2_{L^2(\dist^{\alpha},T)}.
\end{equation}

Finally, we control the term $h_{T}^{\alpha + 2 -d} | \bF |^2 \chi(\{z \in T\})$. Let $T \in \T$, and notice first that, if $T \cap \{ z \} = \emptyset$, then the estimate \eqref{eq:local_lower_bound} follows from \eqref{eq:efficiency_first_term}, \eqref{eq:efficiency_second_term}, and \eqref{eq:efficiency_third_term}. If, on the other hand, $T \cap \{ z \} = \{ z \}$, then the element indicator $\E_{\alpha}$ contains the term $h_{T}^{\alpha + 2 -d} | \bF |^2$. To control this term we follow the arguments developed in the proof of \cite[Theorem 5.3]{AGM} that yield the existence of a smooth function $\eta$ such that
\begin{equation}
 \eta(z) = 1,\quad \| \eta \|_{L^{\infty}(\Omega)} = 1, \quad \| \nabla \eta \|_{L^{\infty}(\Omega)} = h_T^{-1},
 \quad \supp(\eta) \subset \mathcal{N}_{T}.
\end{equation}
With the function $\eta$ at hand, we define $\bv_{\eta}:= \bF \eta \in \bH_0^1(\dist^{-\alpha},\Omega)$ and notice that
\begin{multline}
| \bF |^2 = \langle \bF \delta_z, \bv_{\eta} \rangle  = c ( (\bu, p), (\bv_{\eta},0) ) = c ( (\be_{\bu}, e_p), (\bv_{\eta},0) ) + c ( (\bu_{\T}, p_{\T}), (\bv_{\eta},0) ) 
\\
\lesssim \left( \| \GRAD \be_{\bu} \|^2_{\bL^2(\dist^{\alpha},\mathcal{N}_T)} + \| e_{p} \|^2_{L^2(\dist^{\alpha},\mathcal{N}_T)} \right)^{\frac{1}{2}} \| \GRAD \bv_{\eta} \|_{\bL^2(\dist^{-\alpha},\mathcal{N}_T)} 
\\
+ \sum_{T' \in \T: T' \subset \mathcal{N}_T} \| \Delta \bu_{\T} - \nabla p_{\T} \|_{\bL^2(T')} 
\| \bv_{\eta} \|_{\bL^2(T')} 
\\
+ 
\sum_{T' \in \T: T' \subset \mathcal{N}_T} \sum_{S\in \Sides_{T'}: S \not\subset \partial \mathcal{N}_T} \| \llbracket (\nabla \bu_{\T}  - p_{\T} \mathbf{I})\cdot \boldsymbol{\nu}\rrbracket \|_{L^2(S)}  
\| \bv_{\eta} \|_{L^2(S)}.
\end{multline}

We now use the estimates 
\begin{equation*}
\| \eta\|_{L^2(S)} \lesssim h_T^{\frac{d-1}{2}},
\quad
\| \eta \|_{L^2(\mathcal{N}_T)} \lesssim  h_T^{\frac{d}{2}},
\quad
\| \GRAD \eta \|_{L^2(\dist^{-\alpha},\mathcal{N}_T)} \lesssim h_T^{\frac{d-2}{2}-\frac{\alpha}{2}},
\end{equation*}
to conclude that
\begin{multline}
 | \bF |^2 \lesssim h_T^{\frac{d-2}{2}-\frac{\alpha}{2}} |\bF|  \left( \| \GRAD \be_{\bu} \|^2_{\bL^2(\dist^{\alpha},\mathcal{N}_T)} + \| e_{p} \|^2_{L^2(\dist^{\alpha},\mathcal{N}_T)} \right)^{\frac{1}{2}}
 \\
 + h_T^{\frac{d-2}{2}-\frac{\alpha}{2}} |\bF|  \Bigg ( \sum_{T' \in \T: T' \subset \mathcal{N}_T} h_{T'} D_{T'}^{\frac{\alpha}{2}}  \| \Delta \bu_{\T} - \nabla p_{\T} \|_{\bL^2(T')}
 \\
 + \sum_{T' \in \T: T' \subset \mathcal{N}_{T'}} \sum_{S\in \Sides_{T'}: S \not\subset \partial \mathcal{N}_T} D_{T'}^{\frac{\alpha}{2}} h_{T'}^{\frac{1}{2}} \| \llbracket (\nabla \bu_{\T}  - p_{\T} \mathbf{I})\cdot \boldsymbol{\nu}\rrbracket \|_{L^2(S)} \Bigg),
\end{multline}
where we have also used that, since $z \in T$, $h_T \approx D_T$. Use the estimates \eqref{eq:efficiency_first_term} and \eqref{eq:efficiency_second_term} and conclude.
\end{proof}

\section{Low order stabilized schemes}
\label{sec:stabilized_fem}

In the previous section we have provided an a posteriori error analysis for the discrete scheme \eqref{eq:StokesDiscrete} that is based on the finite element pairs \eqref{eq:mini_V}--\eqref{eq:mini_P} and \eqref{eq:th_V}--\eqref{eq:th_P}. We recall that both of these pairs are compatible, \ie they satisfy the discrete inf--sup condition \eqref{eq:infsup_discrete}, and that this feature does come at a cost. Namely, this condition requires to increase the polynomial degree of the discrete spaces beyond what is required for conformity: it is not possible to approximate the velocity field with piecewise linears, while the pressure space is approximated by piecewise constants or linears; see
\cite[Section 4.2.3]{Guermond-Ern}. If lowest order possible is desired, it is thus necessary to modify the discrete problem to circumvent the need of satisfying condition \eqref{eq:infsup_discrete} \cite{hughes1986new}: this gives rise to the so--called \emph{stabilized methods}. In the literature several stabilized techniques can be found: the residual--free--bubbles method, variational multiscale formulations, enriched Petrov--Galerkin methods, pressure projection methods, local projection techniques and Galerkin/least--squares formulations. For an extensive review of different stabilized finite element methods we refer the reader to \cite[Part IV, Section 3]{MR2454024}, \cite[Chapter 7]{bochev2009least}, and \cite[Chapter 4]{john2016finite}.

Let us now describe the low--order stabilized schemes that we shall consider. First we introduce the following finite element spaces
\begin{align}
\label{eq:th_V1}
\mathbf{V}_{\mathrm{stab}}(\T) & = \left\{  \bv_{\T} \in \mathbf{C}(\bar \Omega) : \ \forall T \in \T, \bv_{\T}|_{T} \in \mathbb{P}_1(T)^d \right\} \cap \bH_0^1(\Omega),
\\
\label{eq:th_Pl}
\mathcal{P}_{\ell,\mathrm{stab}}(\T) & = \left\{ q_{\T} \in L^2(\Omega)/\R : \ \forall T \in \T, q_{\T}|_{T} \in \mathbb{P}_{\ell}(T) \right\},
\end{align}
where $\ell \in\{0,1\}$. The approximation to problem \eqref{eq:StokesWeak} seeks then a pair $(\bu_{\T},p_{\T})$ in $\mathbf{V}_{\mathrm{stab}}(\T) \times \mathcal{P}_{\ell,\mathrm{stab}}(\T)$ 
such that
\begin{equation}
\label{eq:StokesDiscrete_STAB}
  \begin{dcases}
    a(\bu_{\T},\bv_{\T}) + b(\bv_{\T},p_{\T}) + s(\bu_{\T},\bv_{\T}) = \bF \cdot \bv_{\T}(z),   &\forall \bv_{\T} \in \mathbf{V}_{\mathrm{stab}}(\T), \\
    -b(\bu_{\T},q_{\T}) + m(p_{\T},q_{\T})= 0, &\forall q_{\T} \in \mathcal{P}_{\ell,\mathrm{stab}}(\T).
  \end{dcases}
\end{equation}
Where the bilinear forms $s:\mathbf{V}_{\mathrm{stab}}(\T) \times \mathbf{V}_{\mathrm{stab}}(\T) \to \R$ and $m:\mathcal{P}_{\ell,\mathrm{stab}}(\T)\times \mathcal{P}_{\ell,\mathrm{stab}}(\T) \to \R$ are chosen as in \cite[Part IV, Section 3.1]{MR2454024}, and are meant to stabilize the scheme:
\begin{equation}\label{eq:s_m}
\begin{aligned}
  s(\bu_{\T},\bv_{\T}) &:= \sum_{T\in\T}\tau_{\mathrm{div}}\int_{T}\DIV \bu_{\T} \DIV \bv_{\T},\\
  m(p_{\T},q_{\T}) &:= \sum_{T\in\T} \tau_{T} \int_{T}\nabla p_{\T}\cdot\nabla q_{\T}+ \sum_{S \in \Sides} \tau_{S}h_{S}\int_{S} \llbracket p_{\T}\rrbracket\llbracket q_{\T}\rrbracket,
\end{aligned}
\end{equation}
where $\tau_{\mathrm{div}}\geq 0$, $\tau_{T} \geq 0$ and $\tau_{S}>0$ denote the so--called stabilization parameters, and $\llbracket q_\T \rrbracket$ has a similar meaning as in the tensor valued case described in \eqref{eq:jump}. It follows from \cite[Lemma 3.4, Section 3.1]{MR2454024} (when $\tau_T >0$) and \cite[Section 2.1]{MR1740398} (when $\tau_T = 0$ and $\ell = 0$) that problem \eqref{eq:StokesDiscrete_STAB} is well--posed.

We immediately notice that, due to the presence of the stabilization terms $s$ and $m$ in the discrete problem \eqref{eq:StokesDiscrete_STAB}, the Galerkin orthogonality property \eqref{eq:GO} is no longer valid. Instead, we have the relation
\[
 \langle  \mathcal{R},  (\bv_\T,q_\T) \rangle_{\mathcal{Y}',\mathcal{Y}} = s(\bu_{\T},\bv_{\T}) + m(p_{\T},q_{\T}) \quad \forall (\bv_{\T},q_{\T}) \in \mathbf{V}_{\mathrm{stab}}(\T) \times \mathcal{P}_{\ell,\mathrm{stab}}(\T),
\]
where 
$\mathcal{R}$ is defined in \eqref{residual}. The previous relation can be rewritten, for $(\bv_{\T},q_{\T}) \in \mathbf{V}_{\mathrm{stab}}(\T) \times \mathcal{P}_{\ell,\mathrm{stab}}(\T)$, as
\begin{equation}\label{eq:GO_new}
0 =  \langle \bF \delta_z, \bv_{\T} \rangle - c( (\bu_{\T}, p_{\T}), (\bv_{\T},q_{\T})) 
- s(\bu_{\T},\bv_{\T}) - m(p_{\T},q_{\T}).
\end{equation}

For the discrete scheme \eqref{eq:StokesDiscrete_STAB}, we define the local error indicators
\begin{multline*}
\mathcal{E}_{\alpha,\mathrm{stab}}(\bu_{\T},p_{\T};T):= \left( h_T^2D_T^{\alpha}   \|  \Delta \bu_{\T} - \nabla p_{\T} \|_{\bL^2(T)}^2 +  (1+\tau_{\mathrm{div}}^{2}) \|  \DIV \bu_{\T} \|_{L^2(\dist^{\alpha},T)}^{2} \right.
\\
\left. + h_T D_T^{\alpha} \| \llbracket (\nabla \bu_{\T}  - p_{\T} \mathbf{I}) \cdot \boldsymbol{\nu} \rrbracket \|_{\bL^2(\partial T \setminus \partial \Omega)}^2 + h_{T}^{\alpha + 2 -d} | \bF |^2 \chi(\{z \in T\}) \right)^{\frac{1}{2}},
\end{multline*}
and the global error estimator 
\begin{equation}\label{eq:est_stab}
\mathcal{E}_{\alpha,\mathrm{stab}}(\bu_{\T},p_{\T}; \T):=\left(\sum_{T\in\T}\mathcal{E}^2_{\alpha,\mathrm{stab}}(\bu_{\T},p_{\T}; T)\right)^{\frac{1}{2}}.
\end{equation}

It is now our intention to show the reliability and efficiency of this estimator.

\begin{theorem}[reliability and local efficiency]
Let the pair $(\bu,p) \in \bH_0^1(\dist^{\alpha},\Omega) \times L^2(\dist^{\alpha},\Omega)/\R$ be the solution to problem \eqref{eq:StokesWeak} and $(\bu_{\T},p_{\T}) \in \mathbf{V}_{\mathrm{stab}}(\T) \times \mathcal{P}_{\ell,\mathrm{stab}}(\T)$ its stabilized finite element approximation given as the solution to \eqref{eq:StokesDiscrete_STAB}. If $\alpha \in (d-2,d)$, then
\begin{equation}
\label{eq:global_upper_bound_stab}
\| \GRAD \be_{\bu}  \|^2_{\bL^2(\dist^{\alpha},\Omega)} + \| e_p  \|^2_{L^2(\dist^{\alpha},\Omega)} \lesssim \mathcal{E}^2_{\alpha,\mathrm{stab}}(\bu_{\T},p_{\T}; \T),
\end{equation}
and
\begin{equation}
\label{eq:local_lower_bound_stab}
\mathcal{E}^2_{\alpha,\mathrm{stab}}(\bu_{\T},p_{\T}; T) \lesssim \| \GRAD \be_{\bu} \|^2_{\bL^2(\dist^{\alpha},\mathcal{N}_{T})} + \|  e_{p} \|^2_{L^2(\dist^{\alpha},\mathcal{N}_{T})},
\end{equation}
where the hidden constants in both inequalities are independent of the continuous and discrete solutions, the size of the elements in the mesh $\T$ and $\#\T$.
\end{theorem}
\begin{proof}
Let $(\bv,q) \in \mathcal{Y}$. We invoke the restriction operator $\mathcal{Q}_{\T}$, defined in \eqref{eq:operator_Q}, and set $(\bv_{\T},0) = \mathcal{Q}_{\T} (\bv,q)$ in \eqref{eq:GO_new}, to conclude that
\begin{multline*}
\langle  \mathcal{R}, (\bv, q)  \rangle_{\mathcal{Y}',\mathcal{Y}} = \langle \bF \delta_z, \bv - \bv_\T \rangle 
+ \sum_{T \in \T} \int_{T} (\Delta \bu_{\T} - \nabla p_{\T}) \cdot (\bv -  \bv_\T)
\\
- \sum_{S \in \Sides} \int_{S}  \llbracket (\nabla \bu_{\T}  - p_{\T} \mathbf{I}) \cdot \boldsymbol{\nu}\rrbracket \cdot (\bv - \bv_{\T}) - \sum_{T\in\T} q \DIV \bu_{\T} 
+ s( \bu_{\T},\bv_{\T}).
\label{eq:reliability_3} 
\end{multline*}
%

Notice that the first four terms on the right--hand side of the previous expression have been previously controlled; see the estimates \eqref{eq:II}--\eqref{eq:I}. It is thus sufficient to control the last term. To bound it we invoke the local stability property \eqref{eq:local_stability} of the quasi--interpolation operator $\Pi_{\T}$ to conclude that 
%
%
\begin{align*}
|s(\bu_{\T},\Pi_{\T}\bv)|
& \leq 
\sum_{T\in\T}\tau_{\mathrm{div}}\int_{T} |\DIV \bu_{\T} \DIV \Pi_{\T}\bv| \\
& \lesssim
\sum_{T\in\T} \tau_{\mathrm{div}}\|\DIV \bu_{\T}\|_{L^{2}(\dist^{\alpha},T)} \|\nabla \bv\|_{L^{2}(\dist^{-\alpha},\mathcal{S}_T)}.
\end{align*}
Consequently, invoking the finite overlapping property of stars, we arrive at
\[
 |s(\bu_{\T},\bv_{\T})| \leq \|\nabla \bv\|_{L^{2}(\dist^{-\alpha},\Omega)} \left( \sum_{T\in\T} \tau_{\mathrm{div}}^{2} \|\DIV \bu_{\T}\|_{L^{2}(\dist^{\alpha},T)}^2 \right)^{\frac{1}{2}}.
\]

Finally, by gathering the estimates \eqref{eq:II}--\eqref{eq:I} with the previous one, and resorting to the finite overlapping property of stars, again, we arrive at the global upper bound \eqref{eq:global_upper_bound_stab}.
%

The local efficiency \eqref{eq:local_lower_bound_stab} follows as a direct consequence of the estimate in Theorem \ref{Th:efficiency} since, as it is usual in residual error estimation, the lower bound does not contain any consistency terms, even when stabilized schemes are considered; see \cite{MR993474}.
\end{proof}


%
%
%
%
 
 \section{Numerical experiments}
\label{sec:numerics}
In this section we present a series of numerical examples that illustrate the performance of the devised error estimators $\E_{\alpha}$ and $\mathcal{E}_{\alpha,\mathrm{stab}}$. To explore the performance of $\E_{\alpha}$, defined in \eqref{eq:global_estimator}, we consider the discrete problem \eqref{eq:StokesDiscrete} with the discrete spaces \eqref{eq:th_V}--\eqref{eq:th_P}. This setting will be referred to as \emph{Taylor--Hood approximation}. The performance of the estimator $\mathcal{E}_{\alpha,\mathrm{stab}}$, defined in \eqref{eq:est_stab}, will be explored with the following finite element setting: the discrete spaces are \eqref{eq:th_V1} and \eqref{eq:th_Pl}, with $\ell = 0$, and the stabilization parameters are $\tau_{\mathrm{div}}=0$, $\tau_T = 0$, and $\tau_{S}=1/12$. This setting will be referred to as \emph{low--order stabilized approximation}.

The numerical experiments that will be presented have been carried out with the help of a code that we implemented using \texttt{C++}. All matrices have been assembled exactly and the global linear systems were solved using the multifrontal massively parallel sparse direct solver (MUMPS) \cite{MR1856597,MR2202663}. After obtaining the approximate solution of \eqref{eq:StokesDiscrete} or \eqref{eq:StokesDiscrete_STAB}, the a posteriori error indicator $\E_{\alpha}$ or $\mathcal{E}_{\alpha,\mathrm{stab}}$ is computed. Every mesh $\T$ was adaptively refined by marking for refinement the element $T \in \T$ that were such that the step 3 in \textbf{Algorithm}~\ref{Algorithm} holds. In this way a sequence of adaptively refined meshes was generated from the initial meshes shown in Figure \ref{fig:meshes0}.

 \begin{figure}[ht]
 \centering
 \includegraphics[trim={0 0 0 0},clip,width=1.5cm,height=1.5cm,scale=0.7]{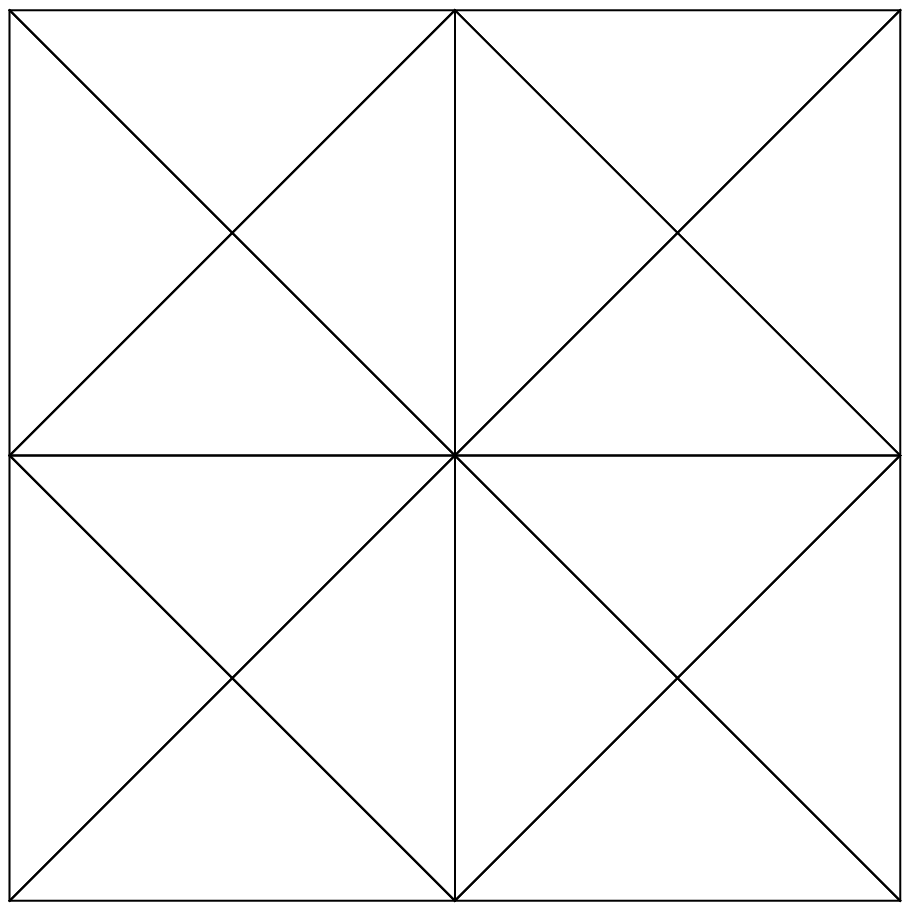}
 \includegraphics[trim={0 0 0 0},clip,width=1.5cm,height=1.5cm,scale=0.7]{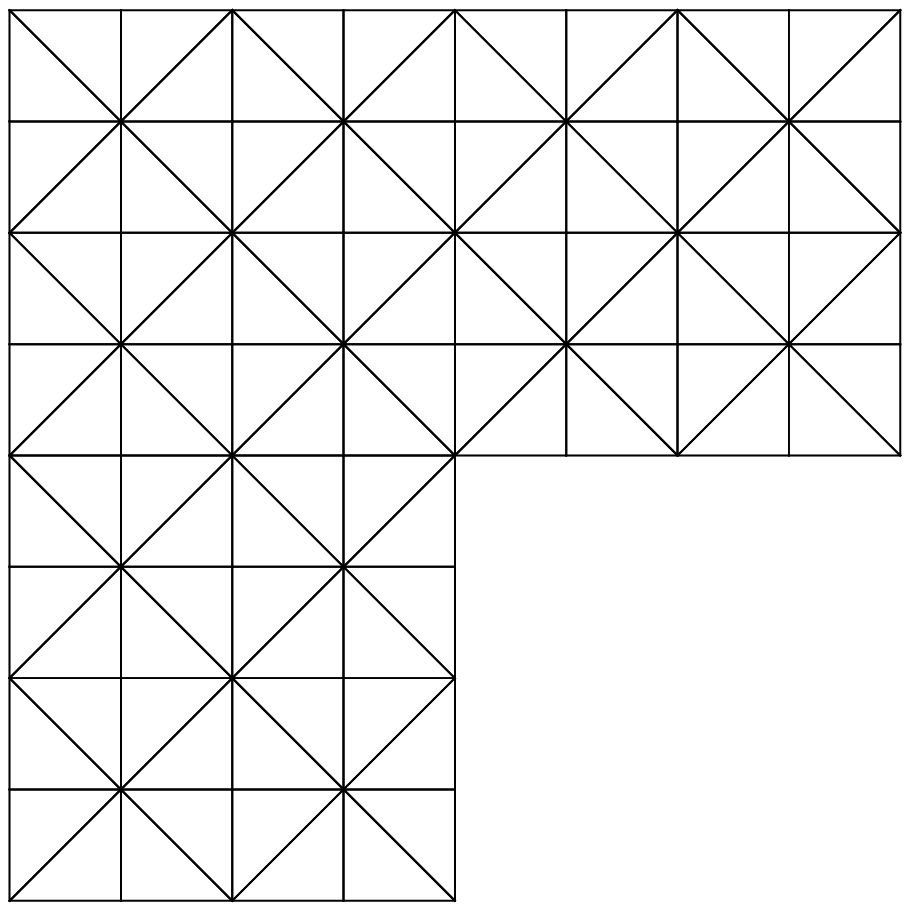}
 \caption{The initial meshes $\T_{0}$ used in the the adaptive \textbf{Algorithm}~\ref{Algorithm} when $\Omega$ is a square (left) and a two--dimensional L--shaped domain (right).}
 \label{fig:meshes0}
 \end{figure}
 
We define the total number of degrees of freedom as $\textsf{Ndof}:=\dim \mathfrak{W} + \dim \mathfrak{P}$, where $(\mathfrak{W},\mathfrak{P}) = (\mathbf{V}(\T),\mathcal{P}(\T))$ for the Taylor-Hood approximation, and $(\mathfrak{W},\mathfrak{P}) = (\mathbf{V}_{\mathrm{stab}}(\T), \mathcal{P}_{\ell,\mathrm{stab}}(\T))$ in the low--order stabilized setting. We measure the error in the $\mathcal{X}$-norm, that is
\begin{equation}\label{eq:error_total}
 \|(\boldsymbol{e}_{\bu},e_{p})\|_{\mathcal{X}}
 :=
 \left( \| \GRAD( \bu-\bu_{\T} )\|_{\bL^{2}(\dist^{\alpha},\Omega)}^{2}+\|p-p_{\T}\|_{L^2(\dist^{\alpha},\Omega)}^{2} \right)^{\frac{1}{2}}.
 \end{equation}

\footnotesize{
\begin{algorithm}[ht]
\caption{\textbf{ Adaptive Algorithm.}}
\label{Algorithm}
Input: Initial mesh $\T_{0}$, interior point $z\in\Omega$, $\alpha$, and stabilization parameters;\\
\textbf{1:} Solve the discrete problem \eqref{eq:StokesDiscrete} (\eqref{eq:StokesDiscrete_STAB}); \\  
\textbf{2:} For each $T \in \mathscr{T}$ compute the local error indicator $\E_\alpha(\bu_{\T},p_{\T};T)$ ($\mathcal{E}_{\alpha,\mathrm{stab}}(\bu_{\T},p_{\T};T)$) given as in \eqref{eq:global_estimator} (\eqref{eq:est_stab});
\\
\textbf{3:} Mark an element $T \in \T$ for refinement if
\[
  \E_\alpha(\bu_{\T},p_{\T};T) > \frac12 \max_{T' \in \T} \E_\alpha(\bu_{\T},p_{\T};T'),
\]
with a similar condition for $\mathcal{E}_{\alpha,\mathrm{stab}}(\bu_{\T},p_{\T};T)$;
\\
\textbf{4:} From step $\boldsymbol{3}$, construct a new mesh, using a longest edge bisection algorithm. Set $i \leftarrow i + 1$, and go to step $\boldsymbol{1}$.
\end{algorithm}}
\normalsize
\subsection{Convex and non--convex domains with homogeneous boundary conditions} 
First, we explore the performance of our devised a posteriori error estimators in problems where no analytical solution is available: convex and non--convex domains $\Omega$ are considered. 

\subsubsection{Example 1: Convex domain}
We consider the square domain $\Omega=(0,1)^{2}$, $\bF=(1 ,1)^{\intercal}$ and $z=(0.5,0.5)^\intercal$. We fix the exponent of the Muckenhoupt weight $\dist^{\alpha}$, defined in \eqref{distance_A2}, as $\alpha=1.5$. 
 \begin{figure}[h]
 \begin{flushleft}
 \begin{minipage}{0.25\textwidth}\centering
 \includegraphics[trim={0 0 0 0},clip,width=3cm,height=3.5cm,scale=0.7]{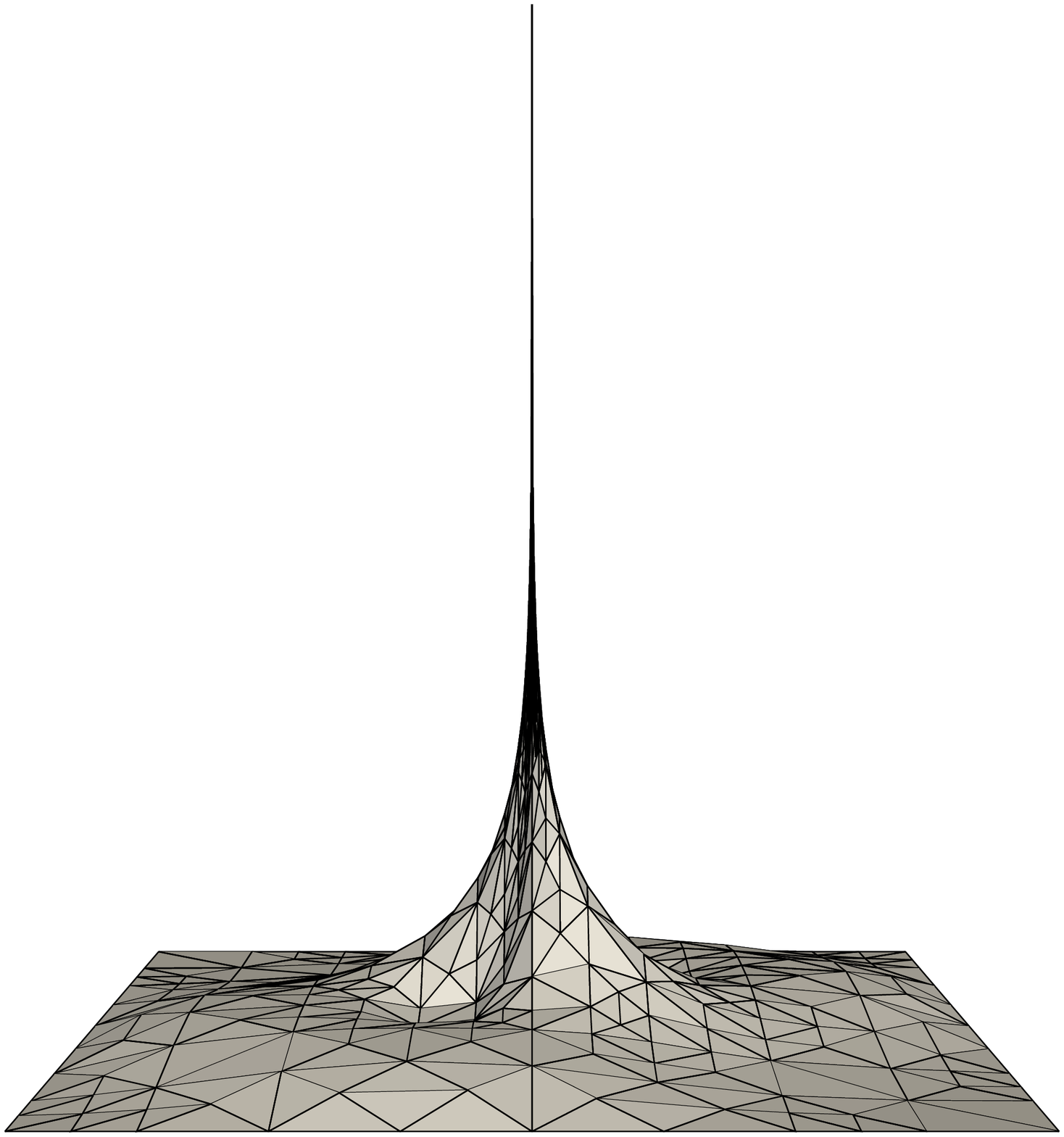}\\
 \tiny{$|\bu_{\T}|$}\\
 \includegraphics[trim={0 0 0 0},clip,width=3cm,height=3.5cm,scale=0.7]{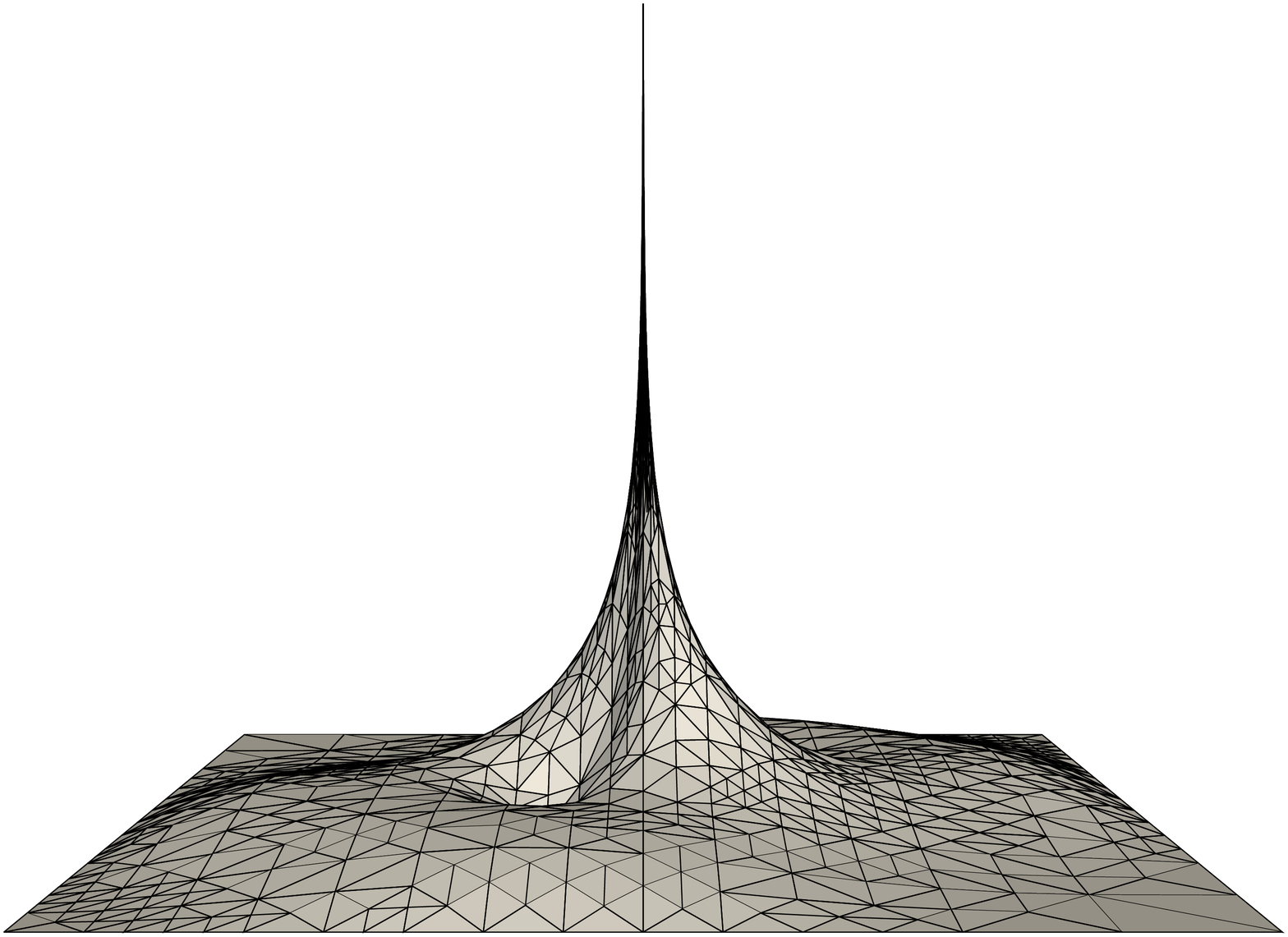}\\
 \tiny{$|\bu_{\T}|$}
 \end{minipage}
 \begin{minipage}{0.25\textwidth}\centering
 \includegraphics[trim={0 0 0 0},clip,width=3cm,height=3.5cm,scale=0.7]{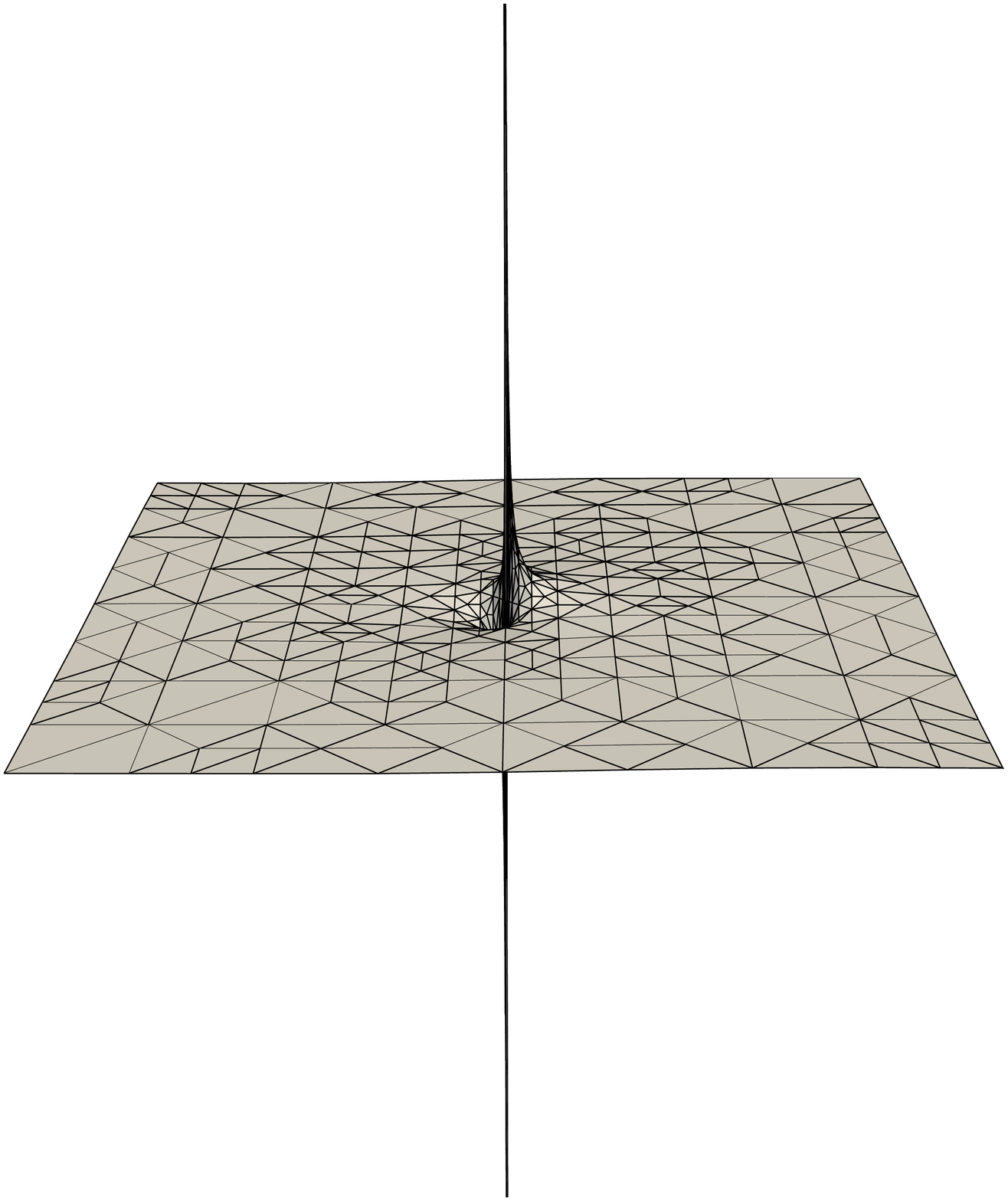}\\
 \tiny{$|p_{\T}|$}\\
 \includegraphics[trim={0 0 0 0},clip,width=3cm,height=3.5cm,scale=0.7]{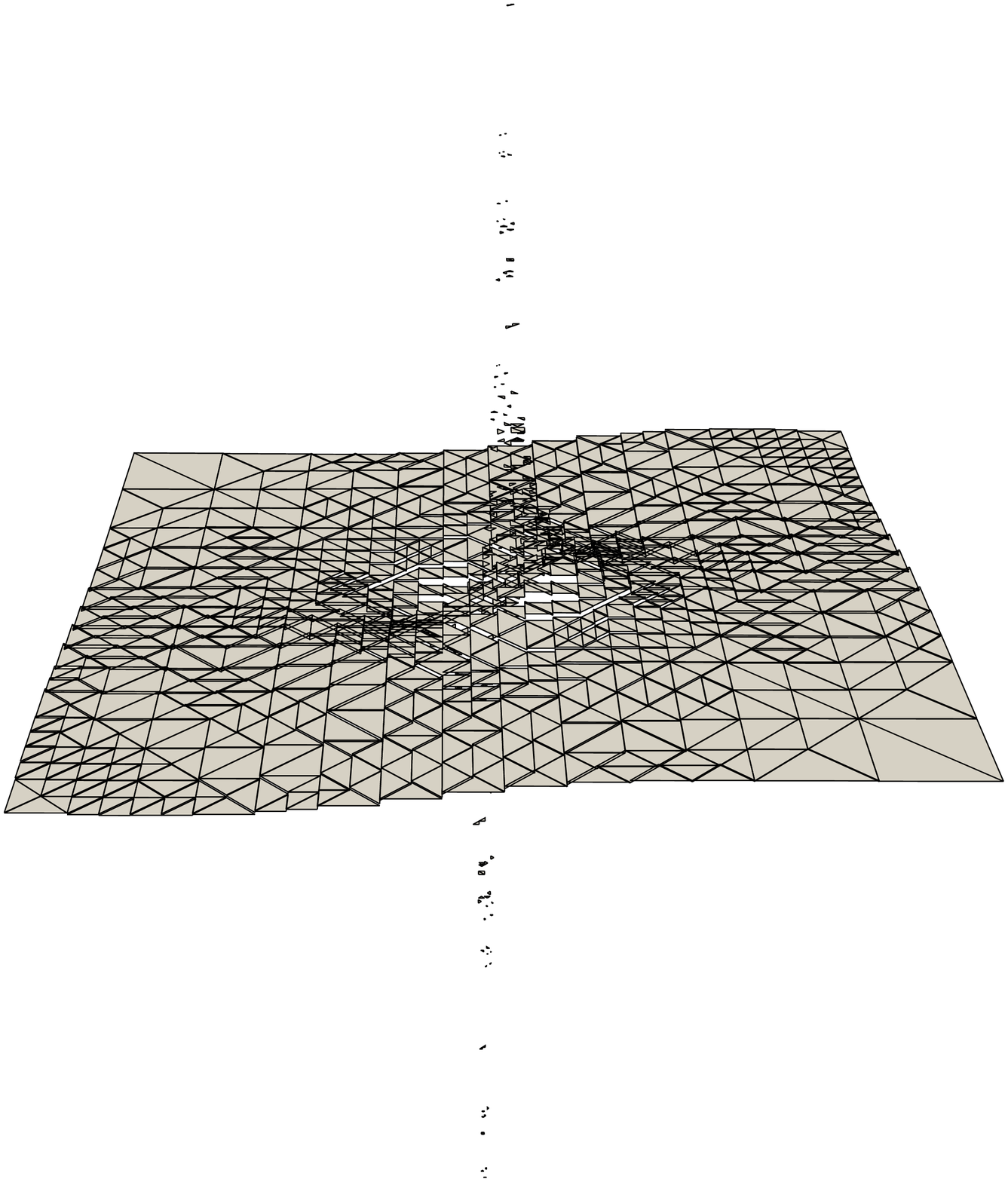}\\
 \tiny{$|p_{\T}|$}
 \end{minipage}
 \begin{minipage}{0.25\textwidth}\centering
 \includegraphics[trim={0 0 0 0},clip,width=3cm,height=3.5cm,scale=0.7]{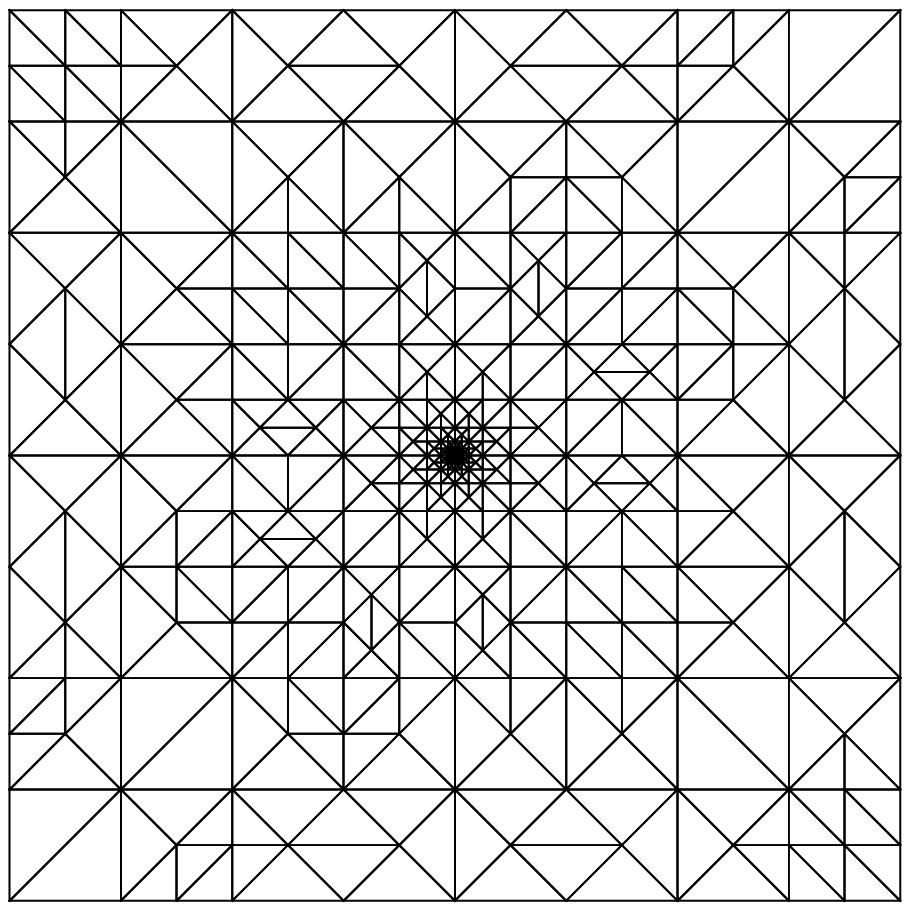}\\
 \tiny{$10^{th}$ adaptive refinement}
 \includegraphics[trim={0 0 0 0},clip,width=3cm,height=3.5cm,scale=0.7]{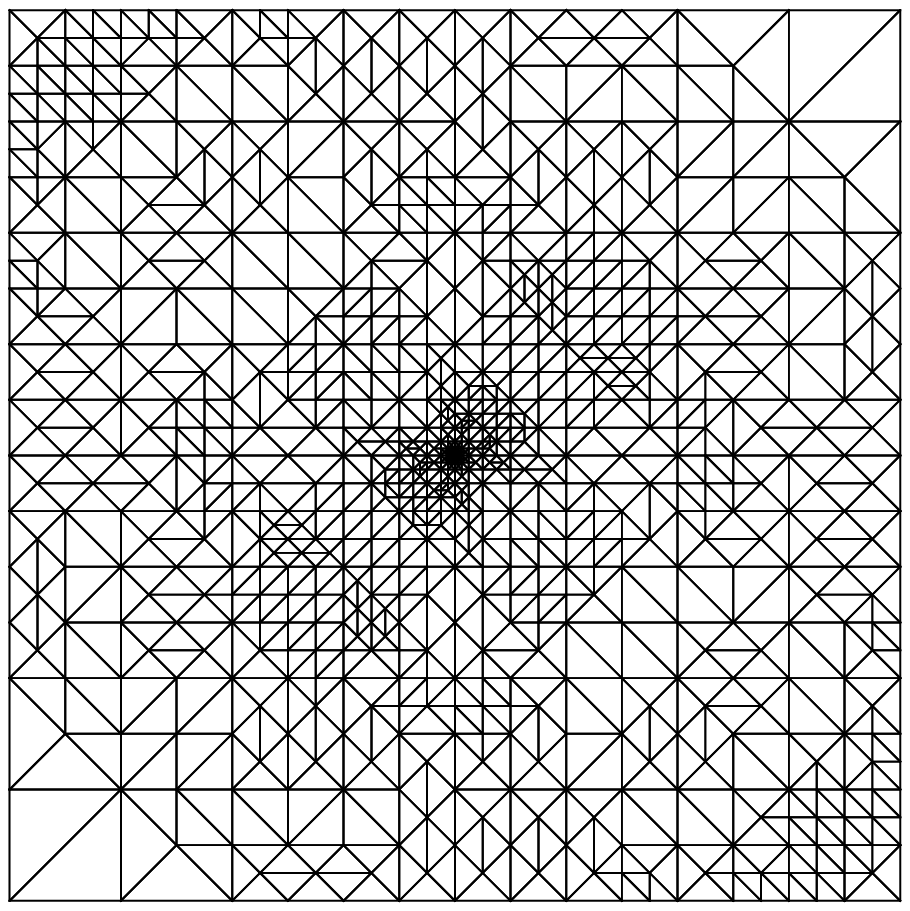}\\
 \tiny{$10^{th}$ adaptive refinement}
 \end{minipage}
 \begin{minipage}{0.25\textwidth}\centering
 \psfrag{error - total - vel - pre}{\large $\|(\boldsymbol{e}_{\bu},e_{p})\|_{\mathcal{X}}$}
 \psfrag{estimator}{\large $\E_{1.5}$}
 \psfrag{rate(h2)}{$\textsf{Ndof}^{-1}$}
 \includegraphics[trim={0 0 0 0},clip,width=3cm,height=3.5cm,scale=0.55]{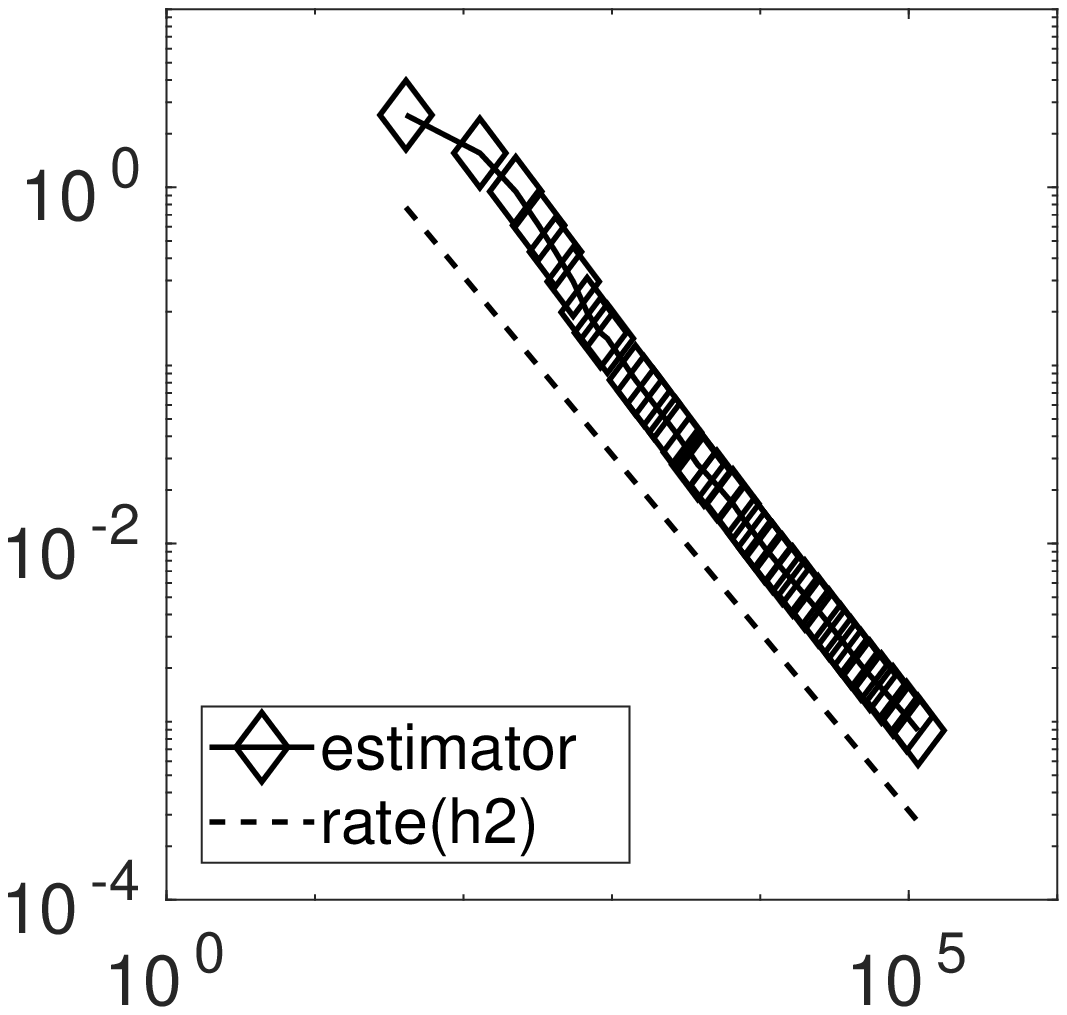}\\
 \psfrag{estimator}{\large $\mathcal{E}_{1.5,\mathrm{stab}}$}
 \psfrag{rate(h2)}{$\textsf{Ndof}^{-1/2}$}
 \includegraphics[trim={0 0 0 0},clip,width=3cm,height=3.5cm,scale=0.55]{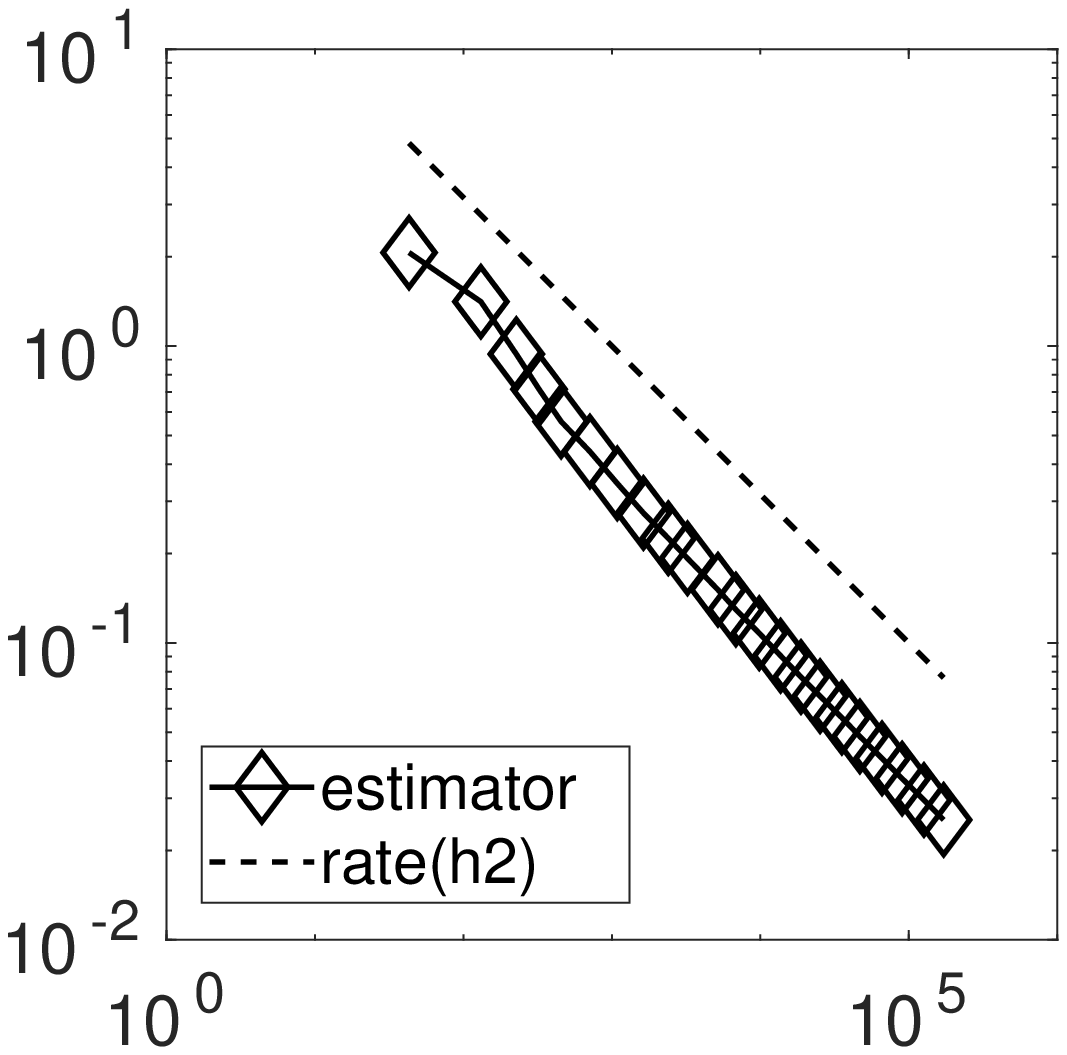}
 \end{minipage}
 \end{flushleft}
 \caption{Example 1: Finite element approximations of $|\bu_{\T}|$ and $p_{\T}$, the mesh obtained after $10$ adaptive refinements and the experimental rate of convergence for the error estimator when Taylor--Hood approximation is used (top) and when the low--order stabilized approximation is considered (bottom).}
 \label{fig:ex1th-st}
 \end{figure}
 \begin{figure}[h]
 \centering
 \psfrag{estima-1}{\large $\E_{0.5}$}
 \psfrag{estima-2}{\large $\E_{0.75}$}
 \psfrag{estima-3}{\large $\E_{1}$}
 \psfrag{estima-4}{\large $\E_{1.25}$}
 \psfrag{estima-5}{\large $\E_{1.5}$}
 \psfrag{rate(h2)}{$\textsf{Ndof}^{-1}$}
 \includegraphics[trim={0 0 0 0},clip,width=5cm,height=5cm,scale=0.55]{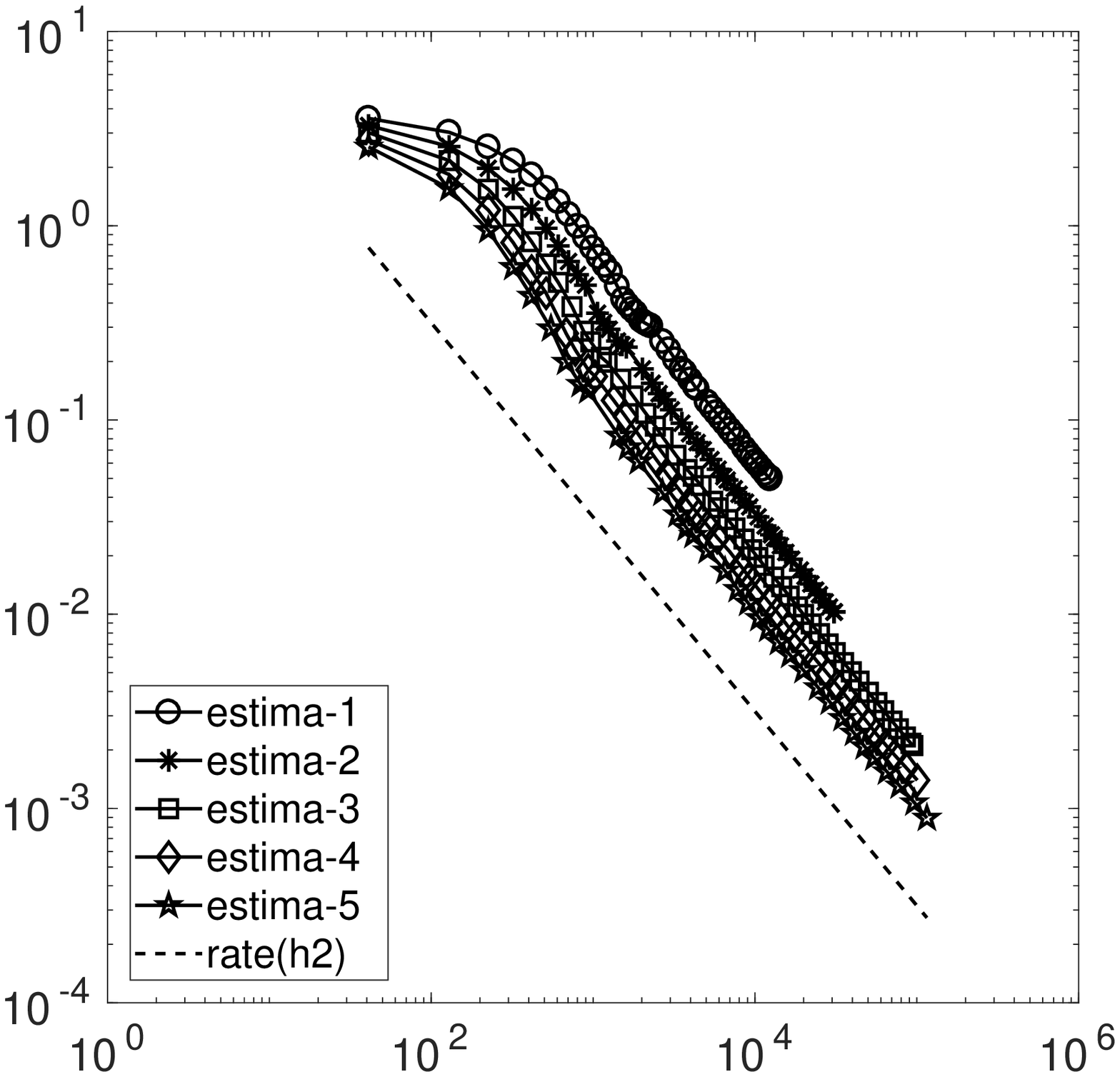}
 \psfrag{estima-1}{\large $\mathcal{E}_{0.5,\mathrm{stab}}$}
 \psfrag{estima-2}{\large $\mathcal{E}_{0.75,\mathrm{stab}}$}
 \psfrag{estima-3}{\large $\mathcal{E}_{1,\mathrm{stab}}$}
 \psfrag{estima-4}{\large $\mathcal{E}_{1.25,\mathrm{stab}}$}
 \psfrag{estima-5}{\large $\mathcal{E}_{1.5,\mathrm{stab}}$}
 \psfrag{rate(h2)}{$\textsf{Ndof}^{-1/2}$}
 \includegraphics[trim={0 0 0 0},clip,width=5cm,height=5cm,scale=0.55]{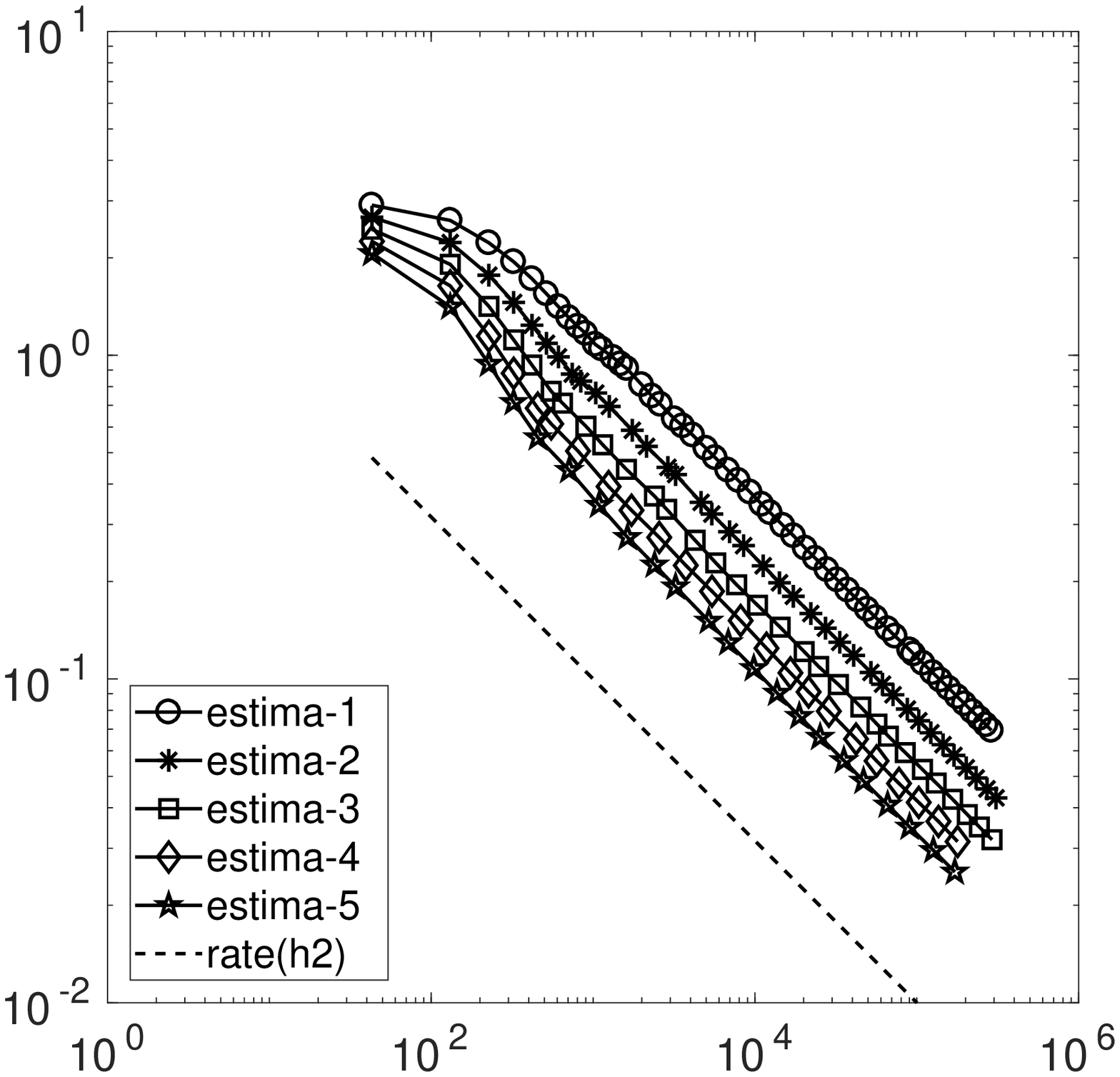}
 \caption{Example 1: For $\alpha=\{0.5,0.75,1,1.25,1.5\}$, we present the experimental rates of convergence for the error estimators $\E_{\alpha}$ (left) and $\mathcal{E}_{\alpha,\mathrm{stab}}$ (right), which are based on Taylor--Hood approximation and low--order stabilized approximation, respectively.}
 \label{fig:ex1th-st-alphas}
 \end{figure}

\subsubsection{Example 2: L-shaped domain} 
We let $\Omega=(-1,1)^{2} \setminus[0,1)\times[-1,0)$, an L--shaped domain, set in  \eqref{eq:StokesDiscrete} the data to be $\bF=(1, 1)^\intercal$ and $z=(0.5,0.5)^\intercal$, and fix the exponent of the Muckenhoupt weight $\dist^{\alpha}$ in \eqref{distance_A2} as $\alpha=1.5$.
 \begin{figure}[h]
 \begin{flushleft}
 \begin{minipage}{0.25\textwidth}\centering
 \includegraphics[trim={0 0 0 0},clip,width=3cm,height=3.5cm,scale=0.7]{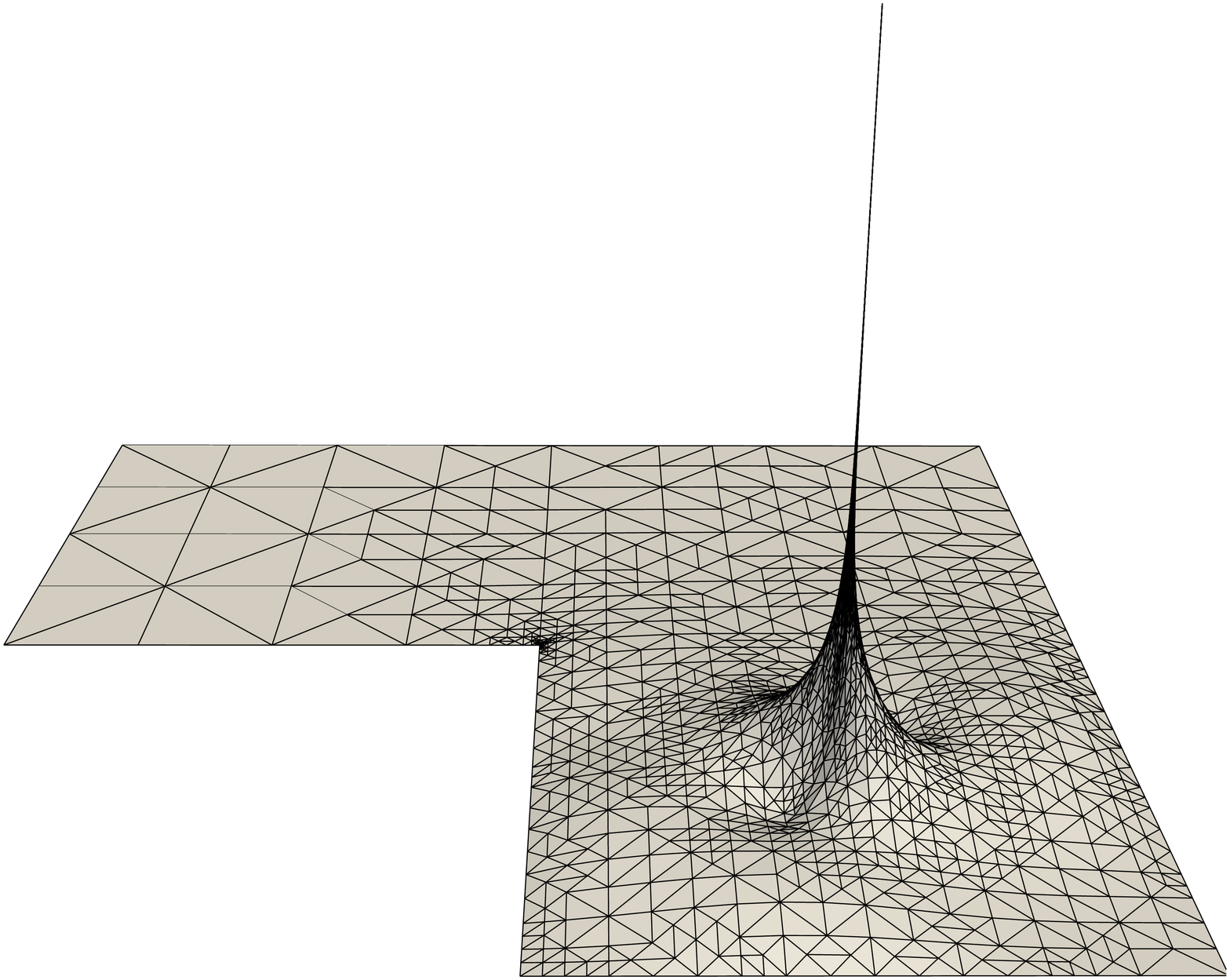}\\
 \tiny{$|\bu_{\T}|$}\\
 \includegraphics[trim={0 0 0 0},clip,width=3cm,height=3.5cm,scale=0.7]{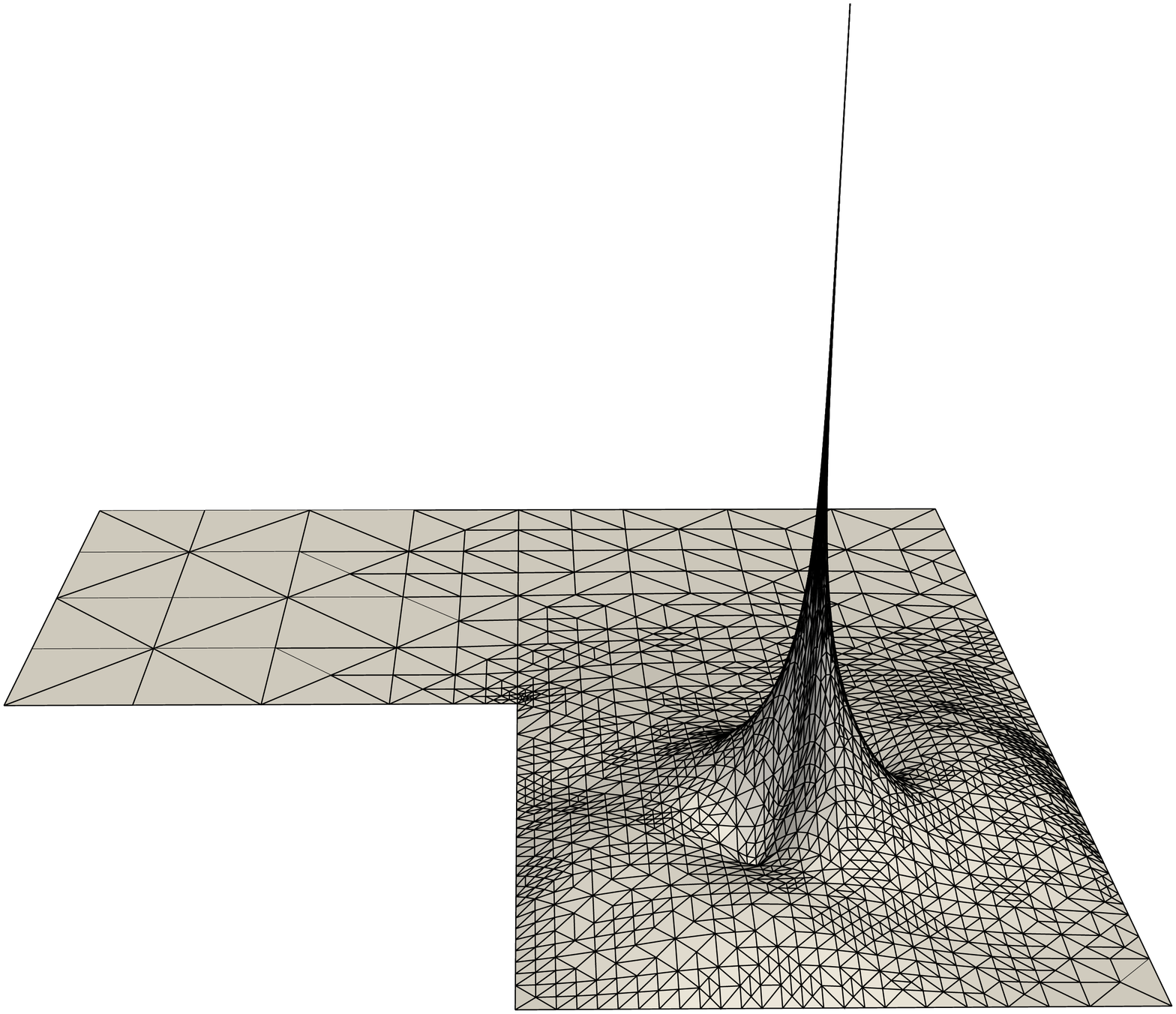}\\
 \tiny{$|\bu_{\T}|$}
 \end{minipage}
 \begin{minipage}{0.25\textwidth}\centering
 \includegraphics[trim={0 0 0 0},clip,width=3cm,height=3.5cm,scale=0.7]{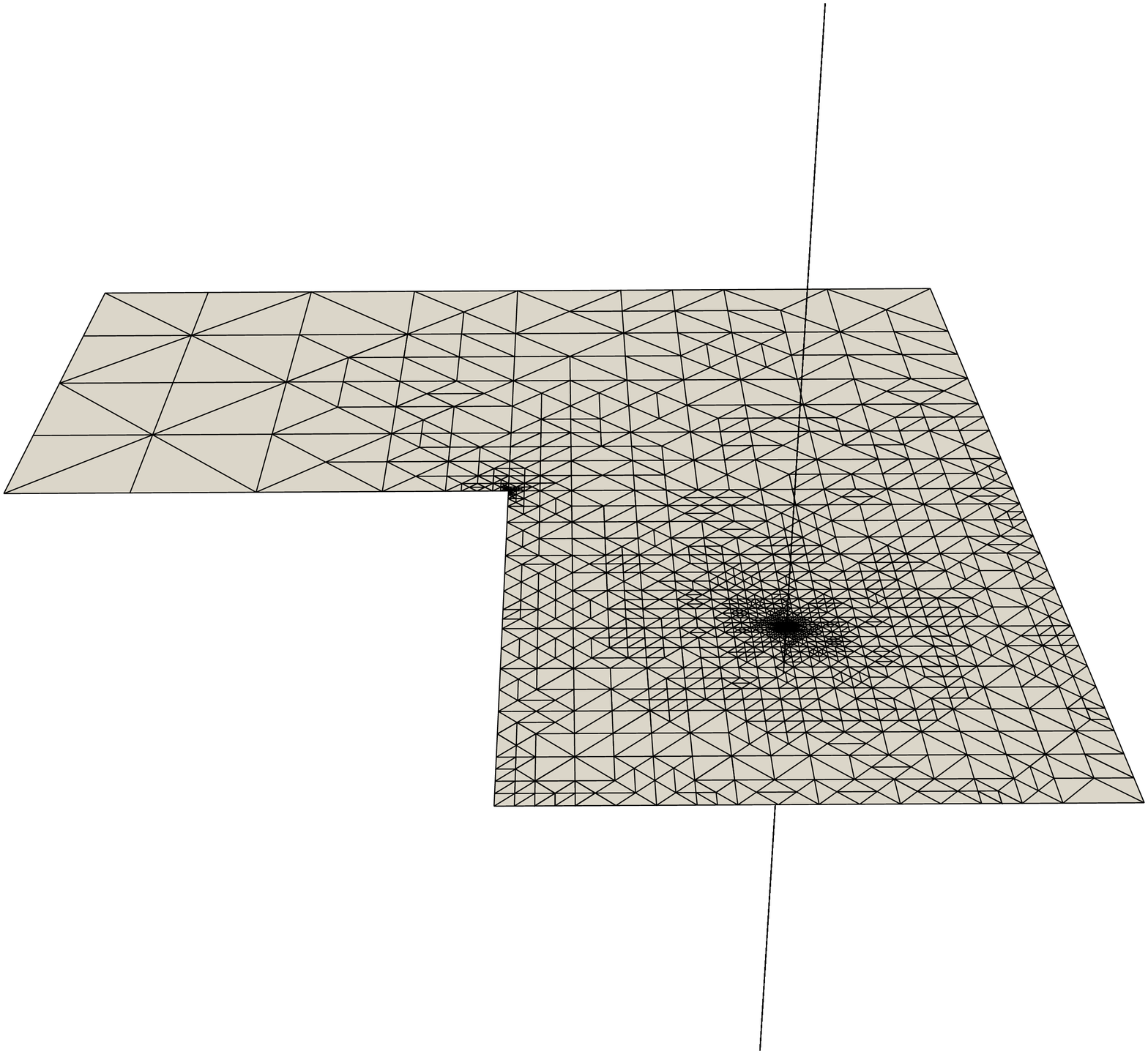}\\
 \tiny{$|p_{\T}|$}\\
 \includegraphics[trim={0 0 0 0},clip,width=3cm,height=3.5cm,scale=0.7]{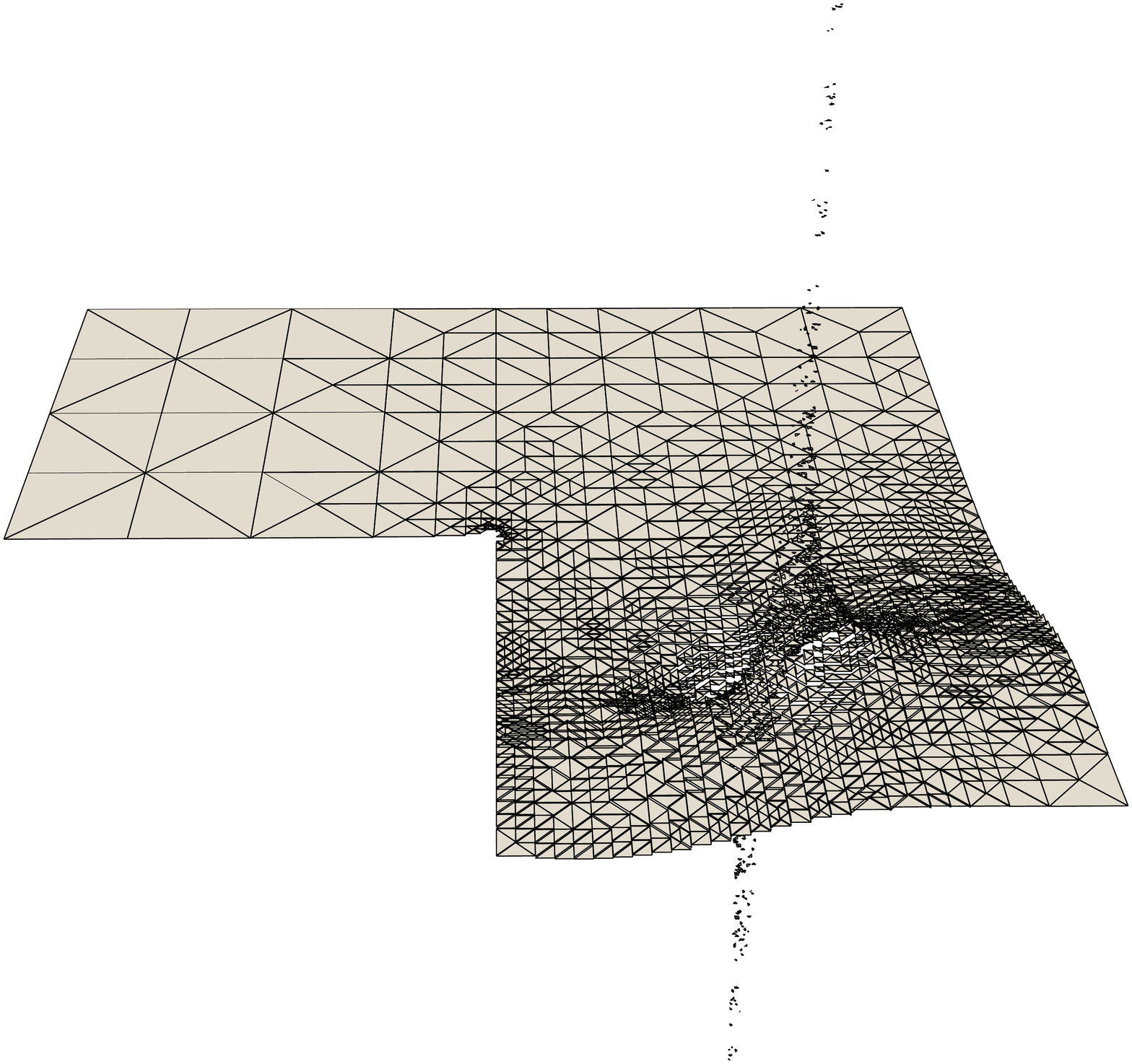}\\
 \tiny{$|p_{\T}|$}
 \end{minipage}
 \begin{minipage}{0.25\textwidth}\centering
 \includegraphics[trim={0 0 0 0},clip,width=3cm,height=3.5cm,scale=0.7]{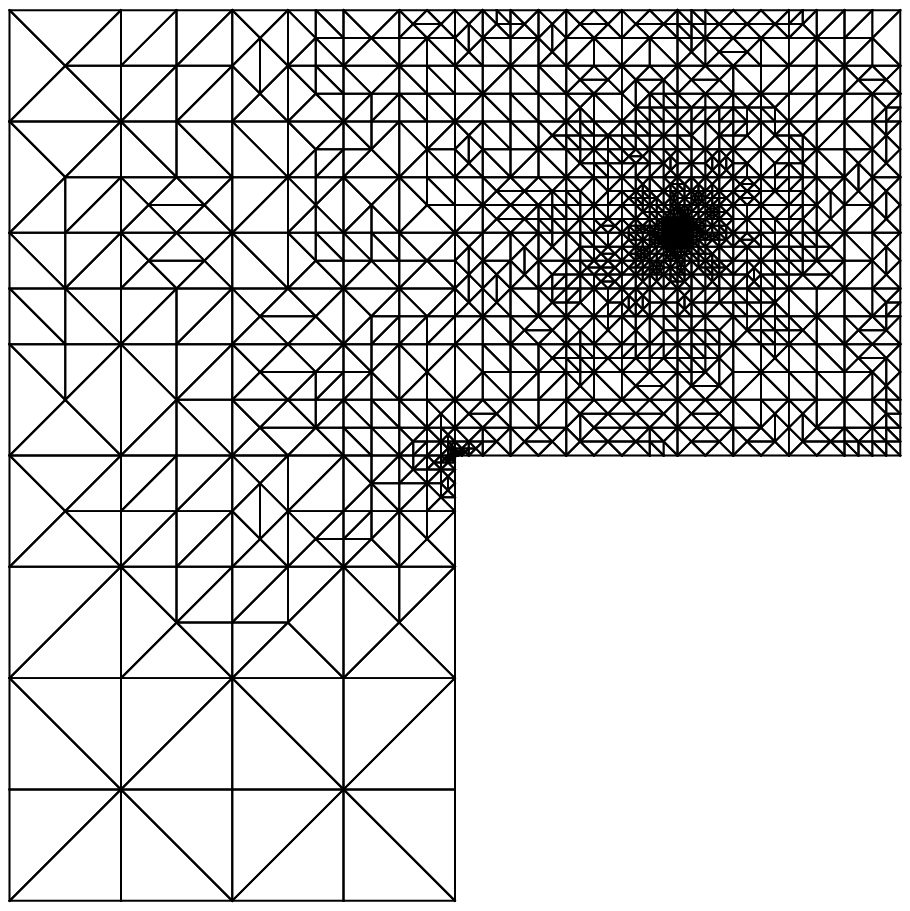}\\
 \tiny{$20^{th}$ adaptive refinement}\\
 \includegraphics[trim={0 0 0 0},clip,width=3cm,height=3.5cm,scale=0.7]{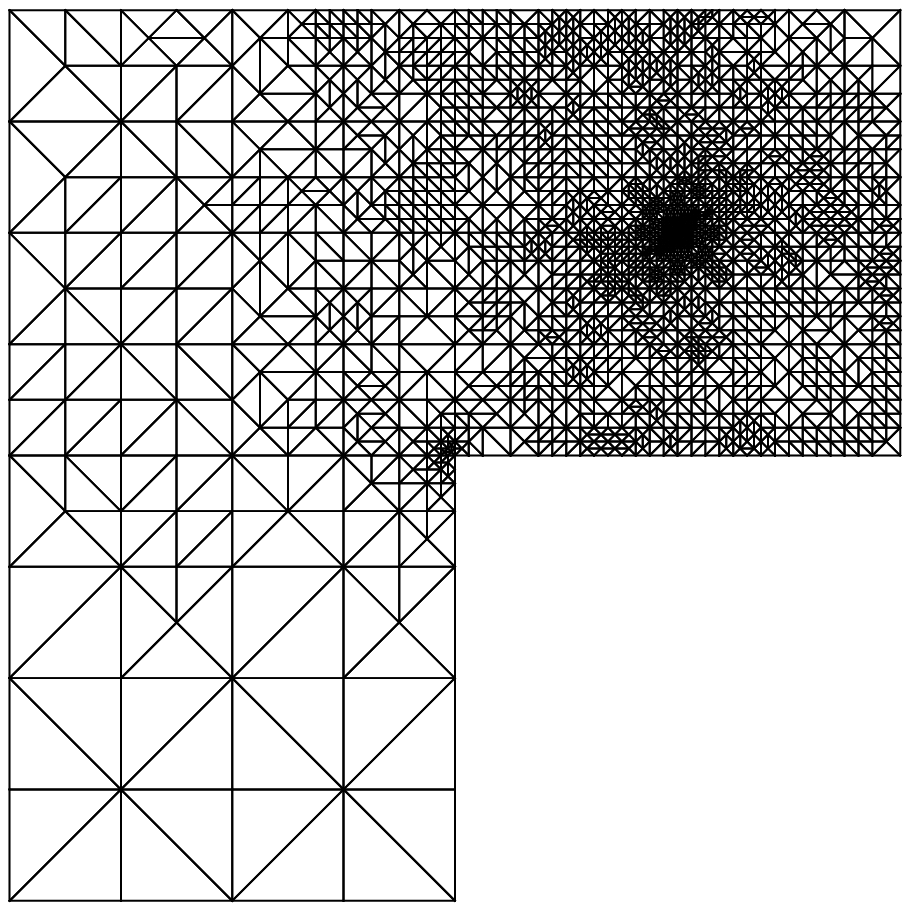}\\
 \tiny{$13^{th}$ adaptive refinement}
 \end{minipage}
 \begin{minipage}{0.25\textwidth}\centering
 \psfrag{error - total - vel - pre}{\large $\|(\boldsymbol{e}_{\bu},e_{p})\|_{\mathcal{X}}$}
 \psfrag{estimator}{\large $\E_{1.5}$}
 \psfrag{rate(h2)}{$\textsf{Ndof}^{-1}$}
 \includegraphics[trim={0 0 0 0},clip,width=3cm,height=3.5cm,scale=0.55]{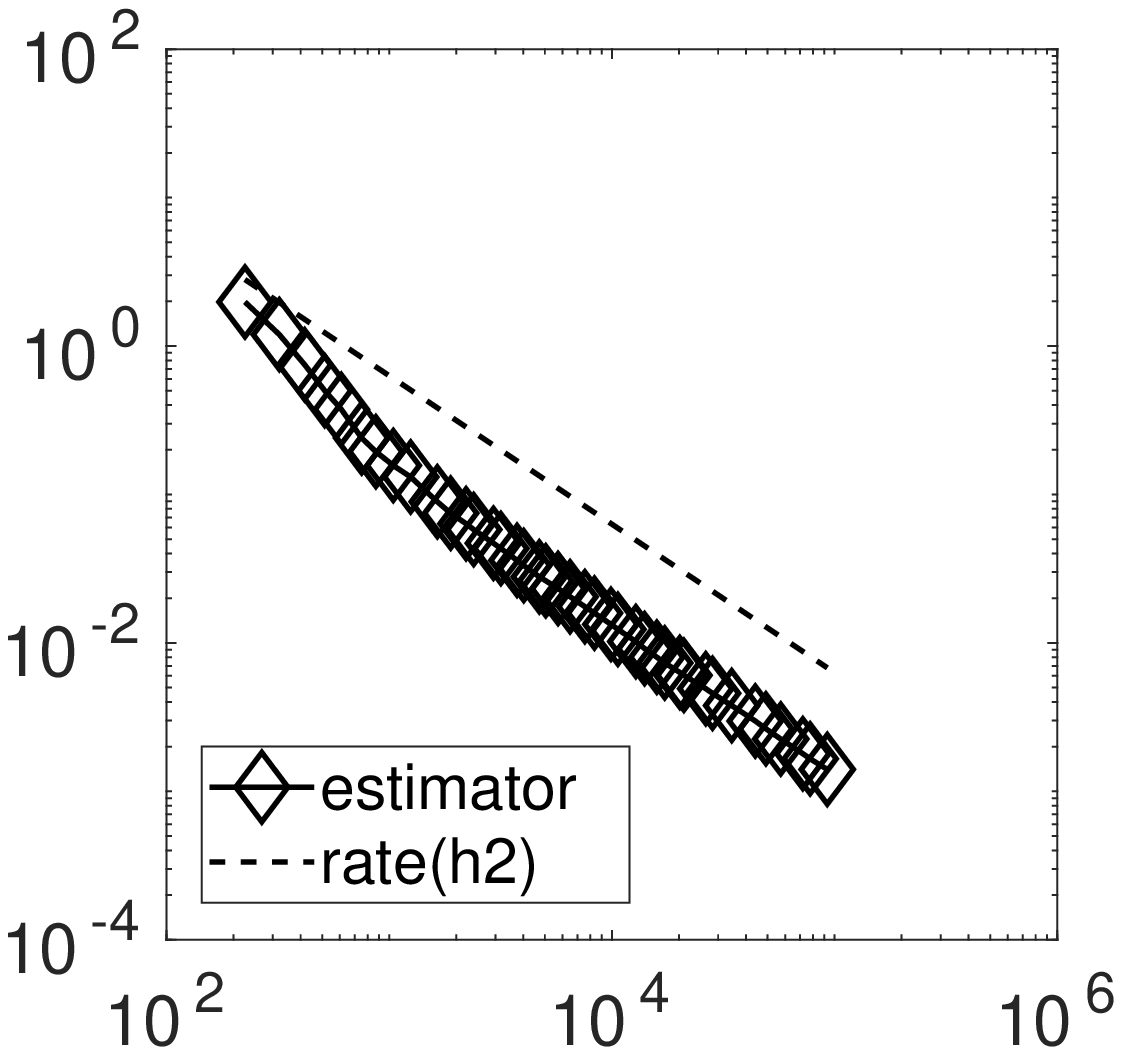}\\
  \psfrag{estimator}{\large $\mathcal{E}_{1.5,\mathrm{stab}}$}
 \psfrag{rate(h2)}{$\textsf{Ndof}^{-1/2}$}
 \includegraphics[trim={0 0 0 0},clip,width=3cm,height=3.5cm,scale=0.55]{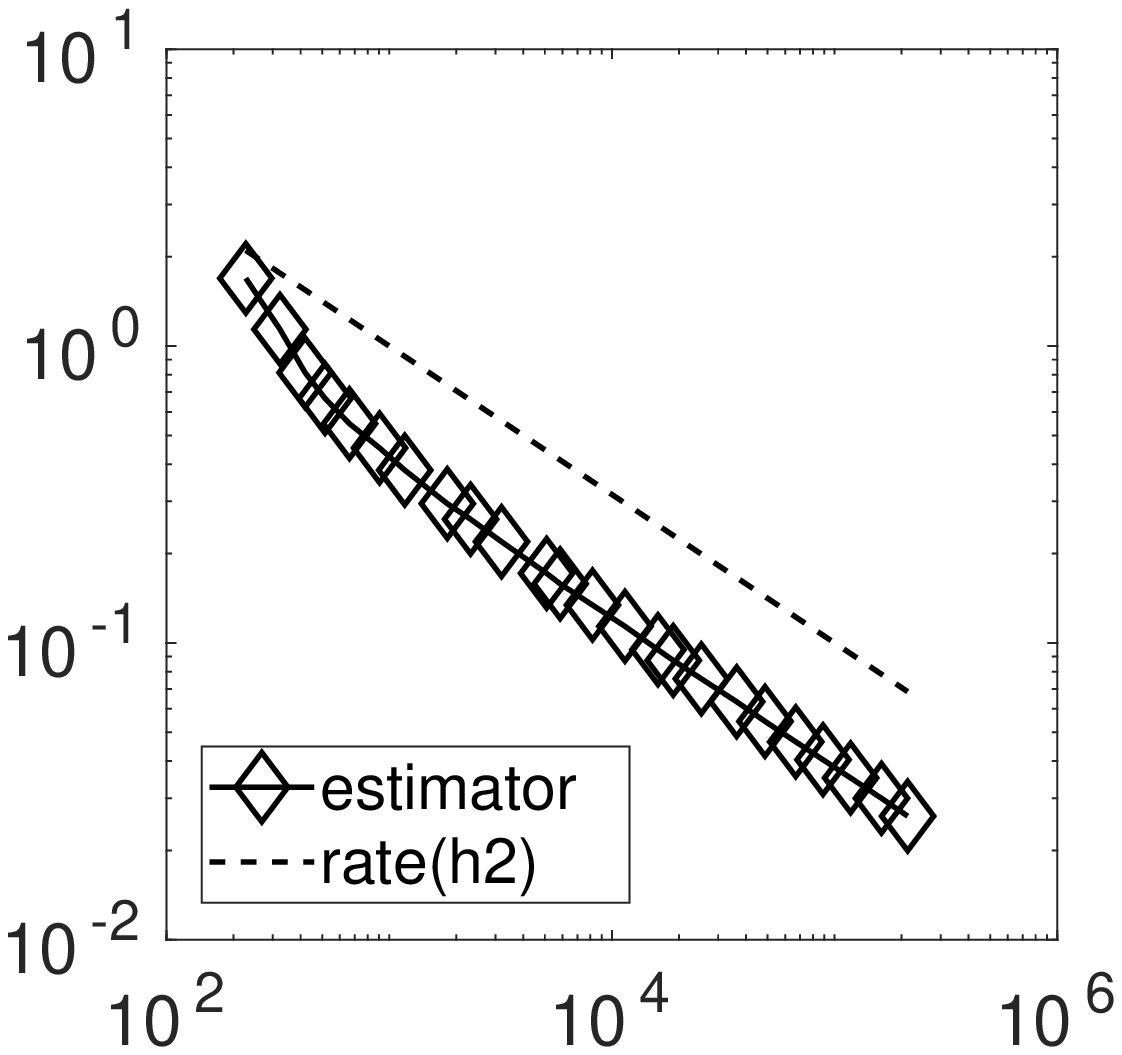}
 \end{minipage}
 \end{flushleft}
 \caption{Example 2: Finite element approximations of $|\bu_{\T}|$ and $p_{\T}$, the mesh obtained after $M$ adaptive refinements, and the experimental rate of convergence for the error estimator when Taylor--Hood approximation is used (top) and when the low--order stabilized approximation is considered; $M=20$ (top) and $M=13$ (bottom).} 
 \label{fig:ex2th-st}
 \end{figure}

In Figures \ref{fig:ex1th-st} and \ref{fig:ex2th-st}, we present the results obtained by the \textbf{Algorithm}~\ref{Algorithm} when is driven by the local indicators $\E_{1.5}$ (Taylor--Hood approximation) and  $\mathcal{E}_{1.5,\mathrm{stab}}$ (low--order stabilized approximation). We show the finite element approximations of $|\bu_{\T}|$ and $p_{\T}$ and the final meshes obtained by the aforementioned schemes. We also present the experimental rates of convergence rate for the estimators $\E_{1.5}$ and $\mathcal{E}_{1.5,\mathrm{stab}}$. We observe that optimal experimental rates of convergence are attained, and that most of the adaptive refinement is concentrated around the delta source. In Figure \ref{fig:ex1th-st-alphas}, we present the experimental rate of convergence for the total error estimators $\E_{\alpha}$ and $\mathcal{E}_{\alpha,\mathrm{stab}}$ when $\alpha\in\{0.5,0.75,1,1.25,1.5\}$. It can be observed that, for all the cases that we have considered, optimal rates of convergence are attained. 
\subsection{Example 3: A series of Dirac sources}

We now go beyond the presented theory and include a series of Dirac delta sources on the right--hand side of the momentum equation. To be precise, we will replace the momentum equation in \eqref{eq:StokesStrong} by
 \begin{equation}\label{deltass_stokes_flow}
 -\Delta\bu + \nabla p = \sum_{z\in\mathcal{Z}}\bF_{z}\delta_{z} \textrm{ in } \Omega,
 \end{equation}
 where $\mathcal{Z}\subset\Omega$ denotes a finite set with cardinality $\#\mathcal{Z}$ which is such that $1< \#\mathcal{Z}<\infty$ and $\{\bF_z\}_{z \in \mathcal{Z}} \subset \R^d$. Based on the results of \cite[Section 5]{MR3679932}, we introduce the weight
 \begin{equation}\label{new_A2}
 \rho(x)=\left\{
 \begin{array}{lc}
 \dist^{\alpha}, & \exists~z\in\mathcal{Z}:|x-z|<\frac{d_{\mathcal{Z}}}{2} ,\\
 1, & |x-z|\geq\frac{d_{\mathcal{Z}}}{2},~\forall~z\in\mathcal{Z},
 \end{array}
 \right.
 \end{equation}
 where
 $ 
 d_{\mathcal{Z}}=
 \min
 \left\{
 \textsf{dist}(\mathcal{Z},\partial\Omega),\min\left\{|z-z'|:z,z'\in\mathcal{Z},z\neq z'\right\}
 \right\}
 $
 and modify the definition \eqref{XandY}, of the spaces $\mathcal{X}$ and $\mathcal{Y}$, as follows:
 \begin{equation}
 \label{new_XandY}
 \mathcal{X} = \bH^{1}_0(\rho,\Omega) \times L^2(\rho,\Omega)/ \mathbb{R}, \quad
 \mathcal{Y} = \bH^{1}_0(\rho^{-1},\Omega) \times L^2(\rho^{-1},\Omega)/ \mathbb{R},
 \end{equation}
It can be proved that $\rho$ belongs to the Muckenhoupt class $A_2$ \cite{MR3215609} and to the restricted class $A_2(\Omega)$.
 
Define 
 \begin{equation}
\label{eq:DT_new}
  D_{T,\mathcal{Z}} := \min_{z \in \mathcal{Z}} \left\{  \max_{x \in T} |x-z| \right\}.
\end{equation}
 We thus propose the following error estimator when the Taylor--Hood scheme is considered:
$$
\mathscr{D}_{\alpha}(\bu_{\T},p_{\T};\T):=\left(\sum_{T\in\T}\mathscr{D}_{\alpha}^{2}(\bu_{\T},p_{\T};T)\right)^{\frac{1}{2}},
$$
where the local indicators are such that
 \begin{multline}
\mathscr{D}_{\alpha}(\bu_{\T},p_{\T};T):= \bigg( h_T^2D_{T,\mathcal{Z}}^{\alpha}   \|  \Delta \bu_{\T} - \nabla p_{\T} \|_{\bL^2(T)}^2 +  \|  \DIV \bu_{\T} \|_{L^2(\rho,T)}^2 
\\
+ h_T D_{T,\mathcal{Z}}^{\alpha}\| \llbracket (\nabla \bu_{\T}  - p_{\T} \mathbf{I})\rrbracket \cdot \boldsymbol{\nu} \|_{\bL^2(\partial T \setminus \partial \Omega)}^2 + \sum_{z \in \mathcal{Z}\cap T} h_{T}^{\alpha + 2 -d} | \bF_{z} |^2 \bigg)^{\frac{1}{2}}.
\label{eq:local_indicator_s}
\end{multline}

Similarly, when the low--order stabilized approximation scheme is considered, we consider the error estimator
\[
\mathcal{D}_{\alpha,\mathrm{stab}}(\bu_{\T},p_{\T};\T):=\left(\sum_{T\in\T}\mathcal{D}_{\alpha,\mathrm{stab}}^{2}(\bu_{\T},p_{\T};T)\right)^{\frac{1}{2}},\]
and the local error indicators
\begin{multline*}
\mathcal{D}_{\alpha,\mathrm{stab}}(\bu_{\T},p_{\T};T):= \bigg( h_T^2D_{T,\mathcal{Z}}^{\alpha}   \|  \Delta \bu_{\T} - \nabla p_{\T} \|_{\bL^2(T)}^2 +  (1+\tau_{\mathrm{div}}^2)\|  \DIV \bu_{\T} \|_{L^2(\rho,T)}^2 
\\
+ h_T D_{T,\mathcal{Z}}^{\alpha} \| \llbracket (\nabla \bu_{\T}  - p_{\T} \mathbf{I})\rrbracket \cdot \boldsymbol{\nu} \|_{\bL^2(\partial T \setminus \partial \Omega)}^2 + \sum_{z \in \mathcal{Z}\cap T} h_{T}^{\alpha + 2 -d} | \bF_{z} |^2 \bigg)^{\frac{1}{2}}.
\end{multline*}

Having defined the problem and estimators we, in particular, set
$\Omega=(0,1)^{2}$ and let 
\[
  \mathcal{Z}=\{(0.25,0.25)^\intercal,(0.25,0.75)^\intercal,(0.75,0.25)^\intercal,(0.75,0.75)^\intercal\}.
\]
We consider $\bF_{z}=(1 , 1)^\intercal$ for all $z \in\mathcal{Z}$ and fix the exponent of the Muckenhoupt weight $\rho$, which is defined in \eqref{new_A2}, as $\alpha=1.5$.
 \begin{figure}[h]
 \begin{flushleft}
 \begin{minipage}{0.25\textwidth}\centering
 \includegraphics[trim={0 0 0 0},clip,width=3cm,height=3cm,scale=0.7]{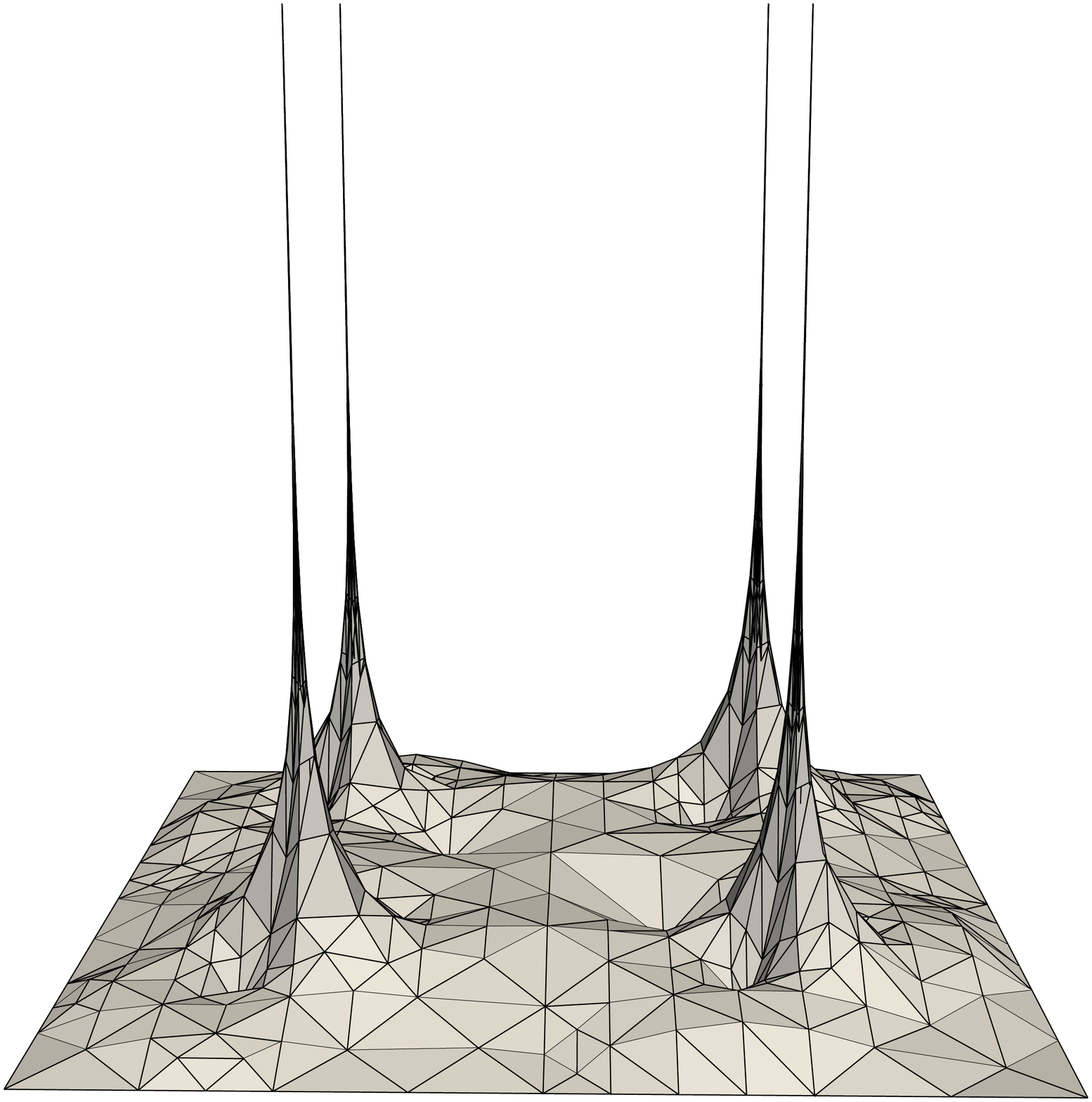}\\
 \tiny{$|\bu_{\T}|$}\\
 \includegraphics[trim={0 0 0 0},clip,width=3cm,height=3cm,scale=0.7]{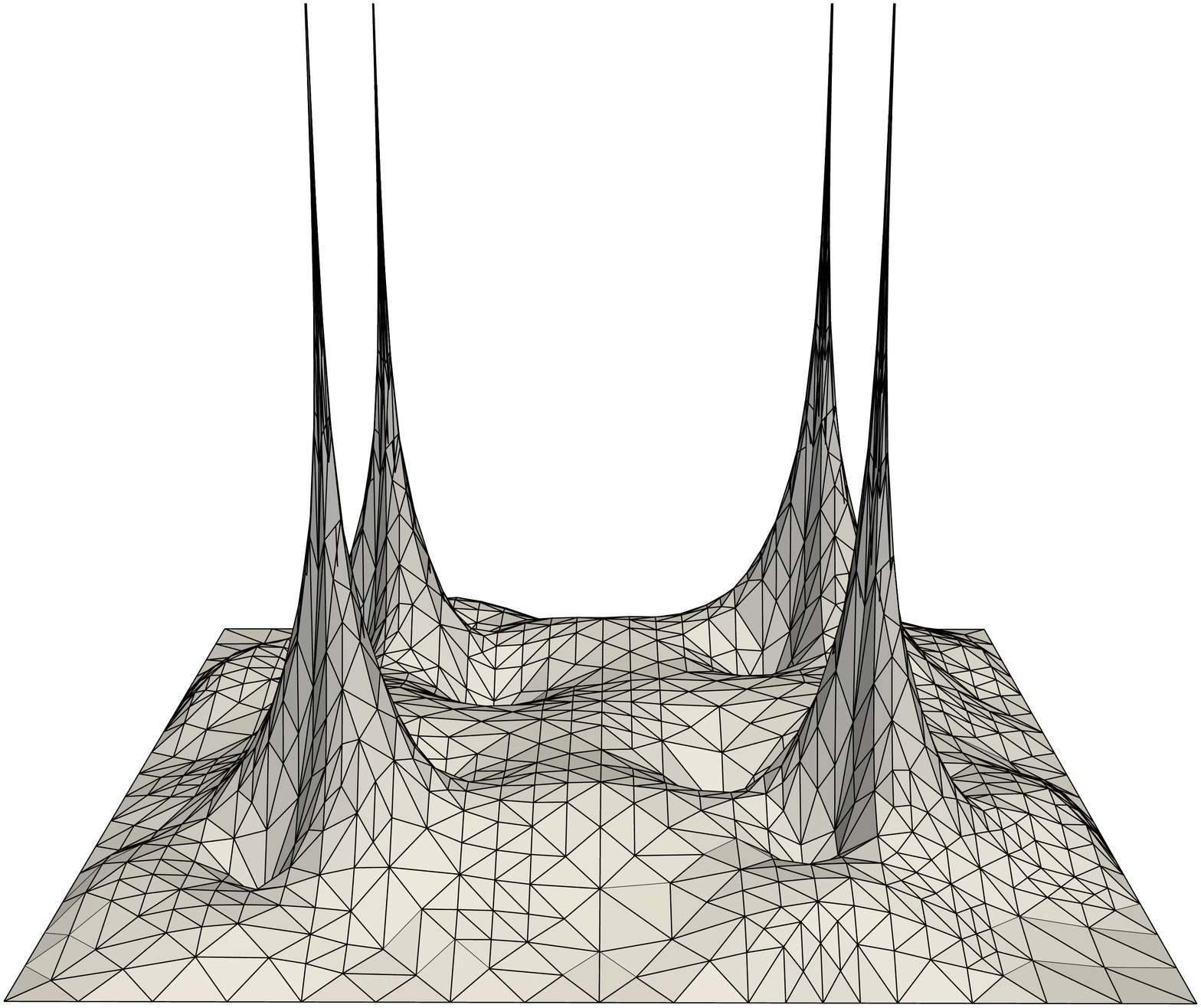}\\
 \tiny{$|\bu_{\T}|$}
 \end{minipage}
 \begin{minipage}{0.25\textwidth}\centering
 \includegraphics[trim={0 0 0 0},clip,width=3cm,height=3cm,scale=0.7]{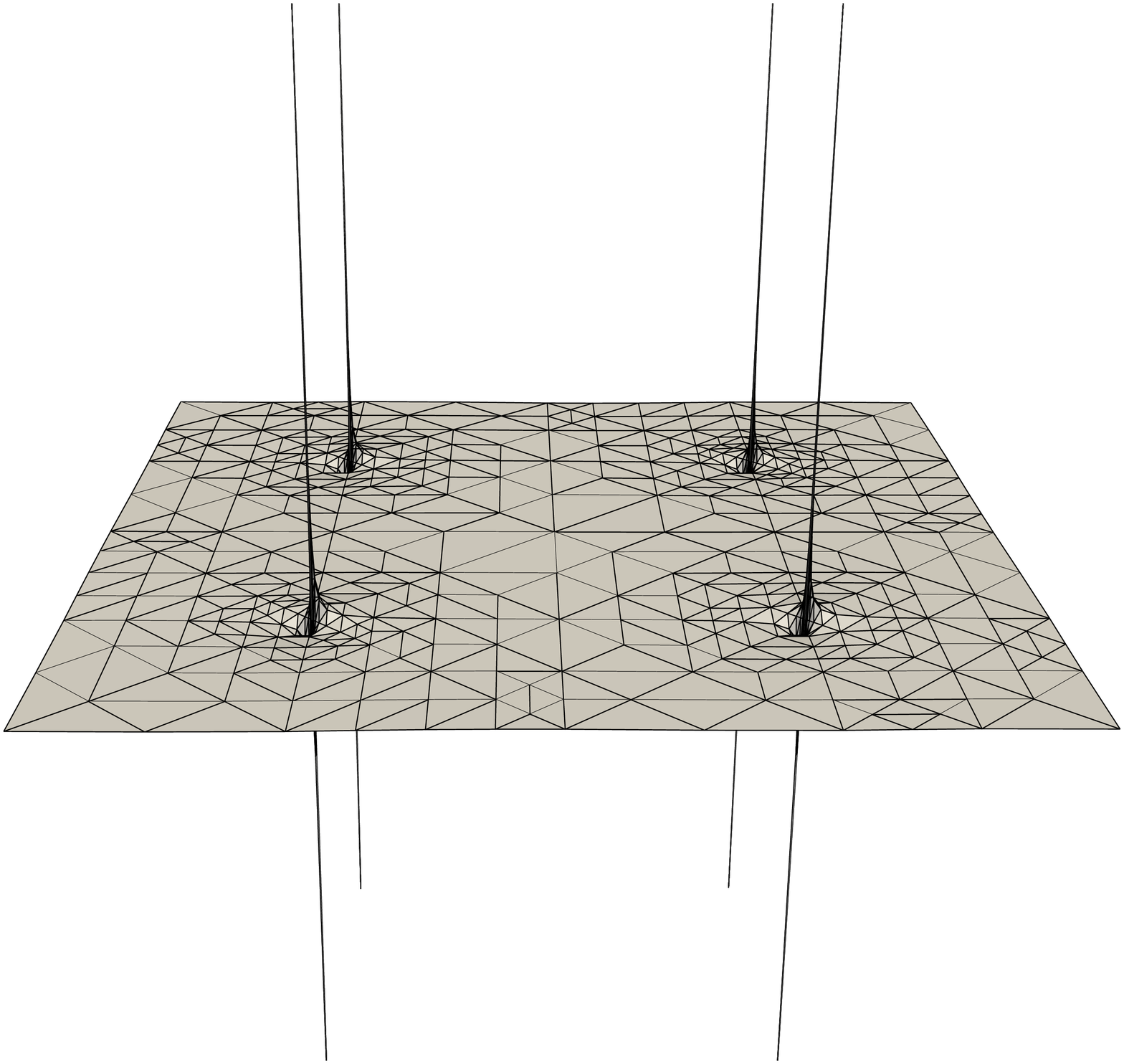}\\
 \tiny{$|p_{\T}|$}\\
 \includegraphics[trim={0 0 0 0},clip,width=3cm,height=3cm,scale=0.7]{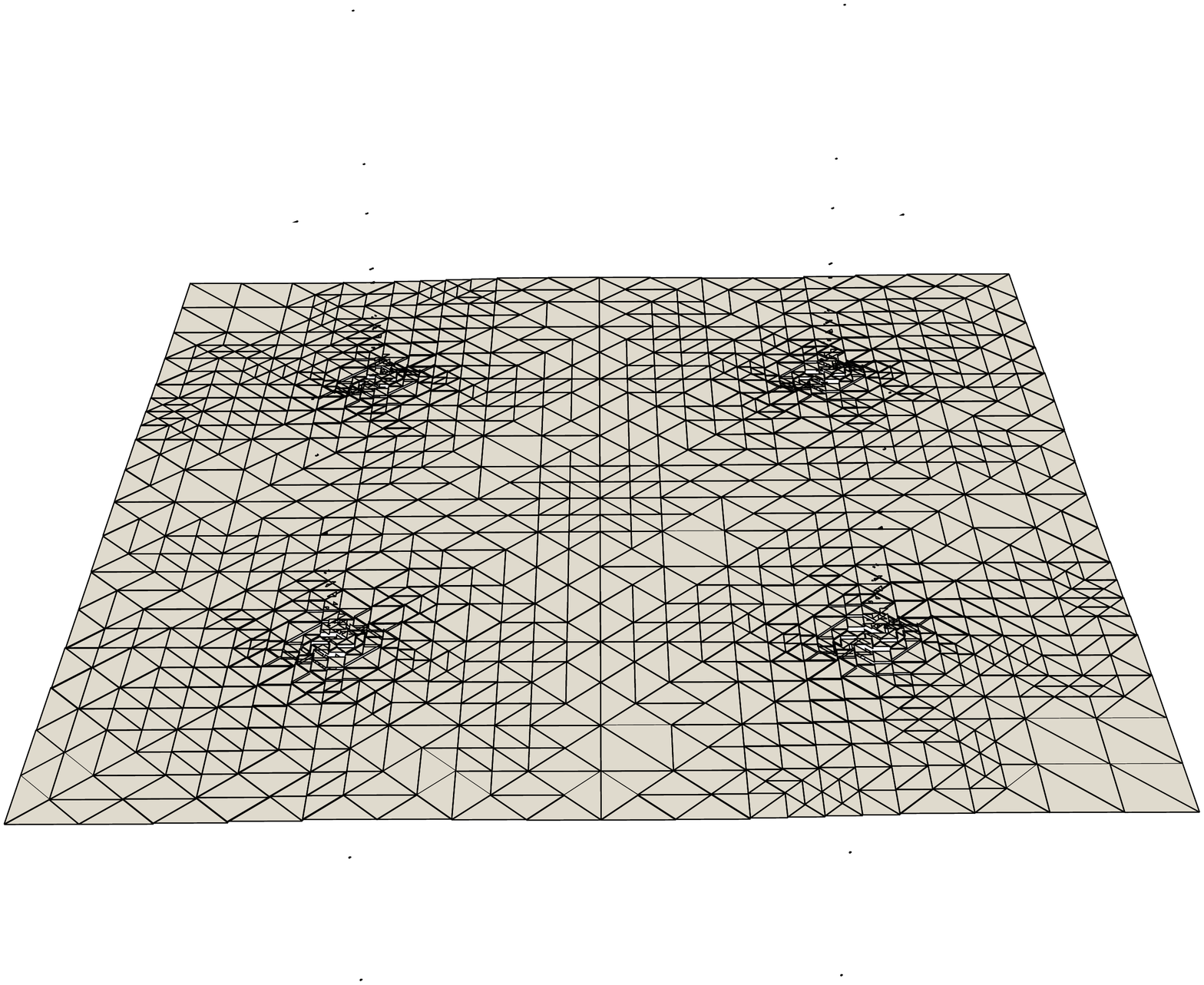}\\
 \tiny{$|p_{\T}|$}
 \end{minipage}
 \begin{minipage}{0.25\textwidth}\centering
 \includegraphics[trim={0 0 0 0},clip,width=3cm,height=3cm,scale=0.7]{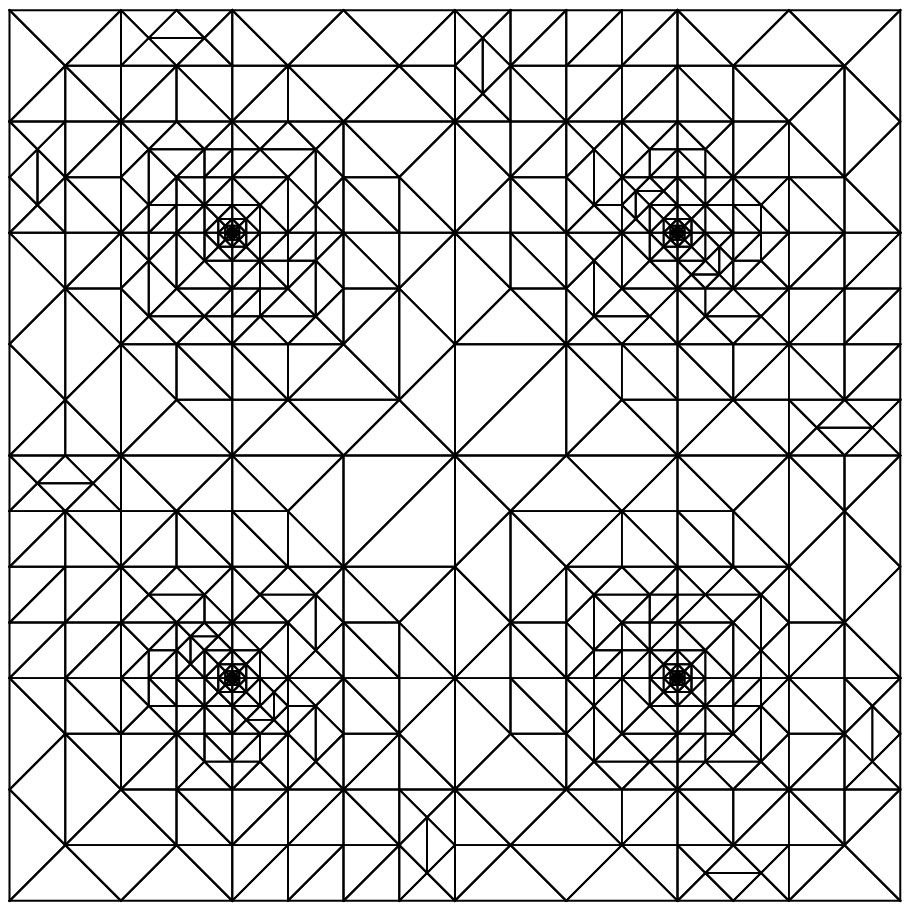}\\
 \tiny{$10^{th}$ adaptive refinement}\\
 \includegraphics[trim={0 0 0 0},clip,width=3cm,height=3cm,scale=0.7]{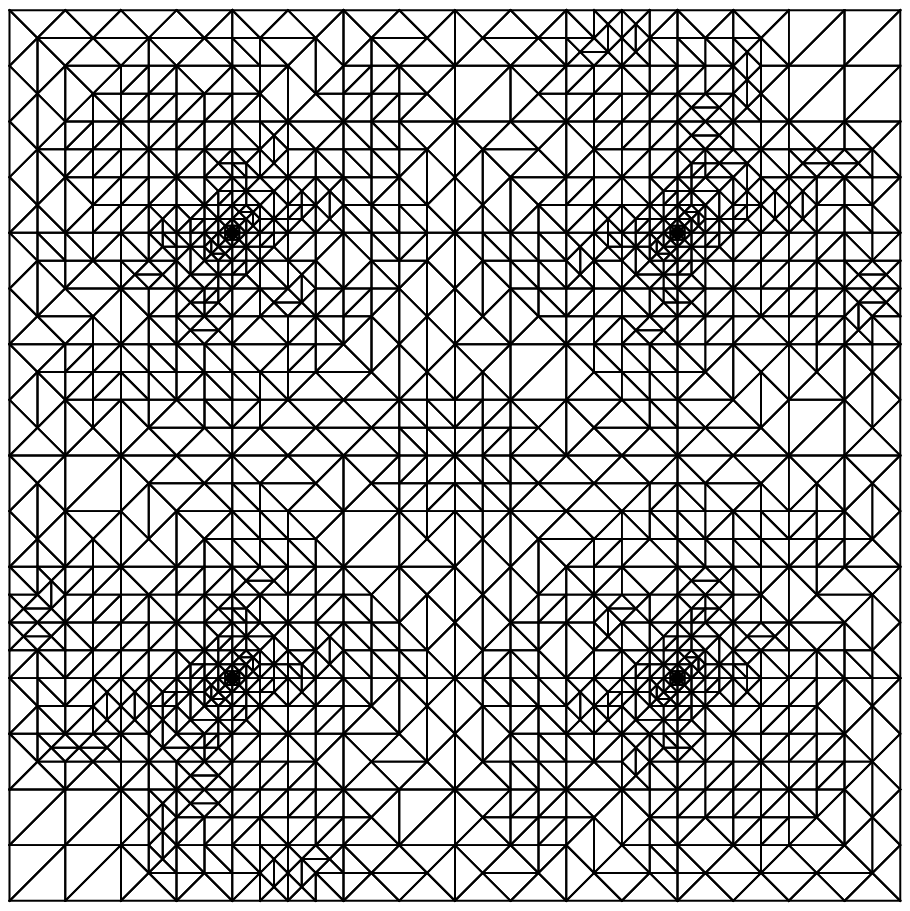}\\
 \tiny{$10^{th}$ adaptive refinement}
 \end{minipage}
 \begin{minipage}{0.25\textwidth}\centering
 \psfrag{error - total}{\large $\|(\boldsymbol{e}_{\bu},e_{p})\|_{\mathcal{X}}$}
 \psfrag{estimator}{\large $\mathscr{D}_{1.5}$}
 \psfrag{rate(h2)}{$\textsf{Ndof}^{-1}$}
 \includegraphics[trim={0 0 0 0},clip,width=3cm,height=3cm,scale=0.55]{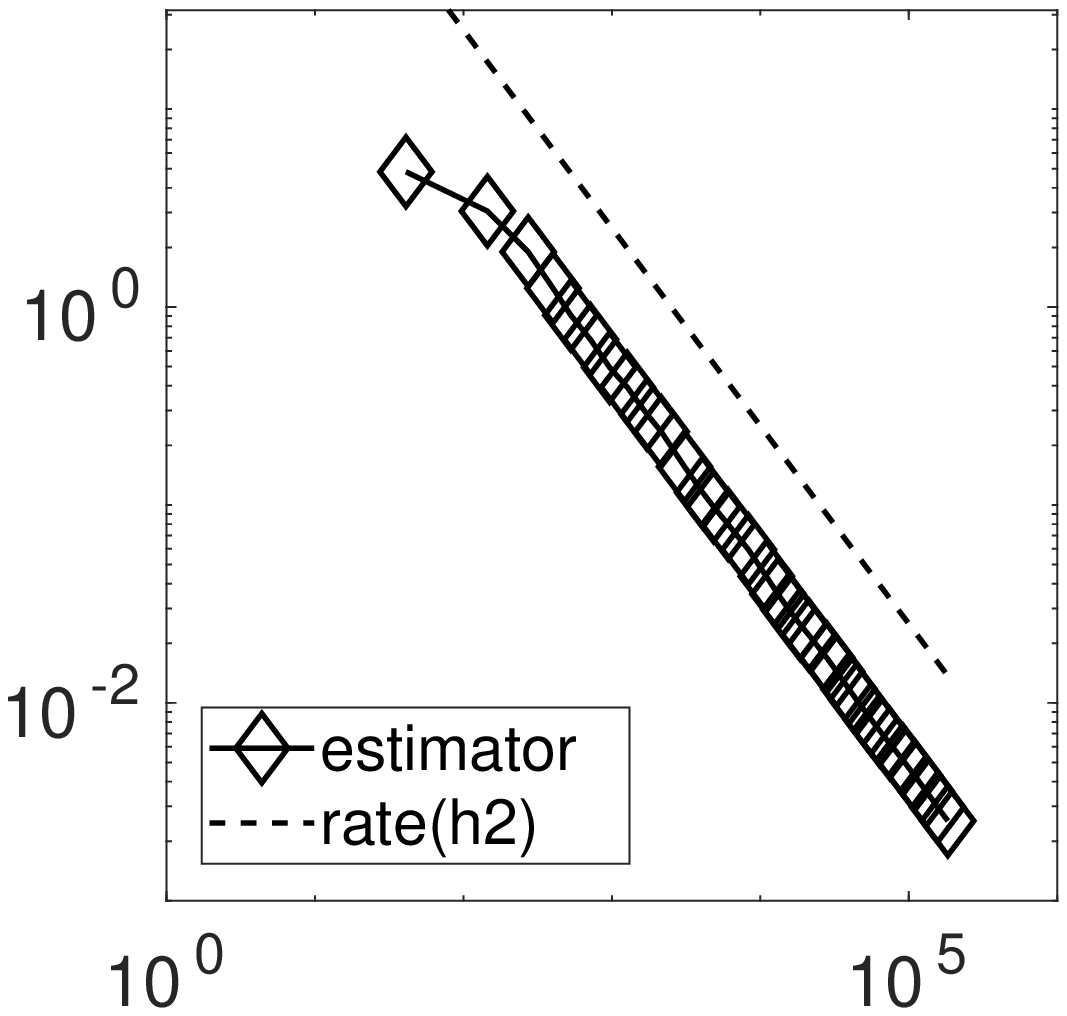}\\
  \psfrag{estimator}{\large $\mathcal{D}_{1.5,\mathrm{stab}}$}
 \psfrag{rate(h2)}{$\textsf{Ndof}^{-1/2}$}
 \includegraphics[trim={0 0 0 0},clip,width=3cm,height=3cm,scale=0.55]{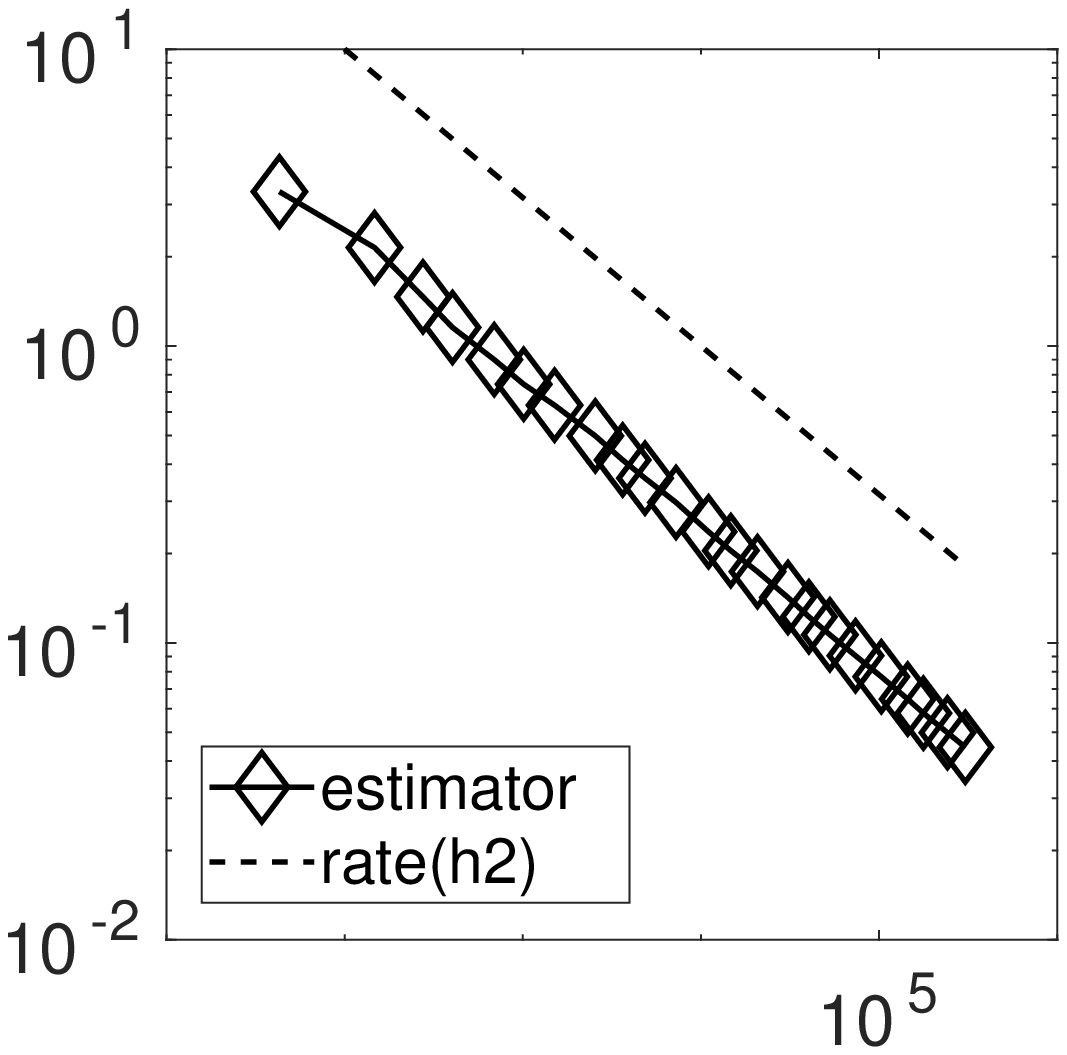}
 \end{minipage}
 \end{flushleft}
 \caption{Example 3: Finite element approximations of $|\bu_{\T}|$ and $p_{\T}$, the meshes obtained after $10$ adaptive refinements, and the experimental rates of convergence when Taylor--Hood approximation is used (top) and when the low--order stabilized approximation is considered (bottom).}
 \label{fig:ex3th-st}
 \end{figure}
 
In Figure \ref{fig:ex3th-st}, we present the results obtained by the \textbf{Algorithm} \ref{Algorithm} when is driven by $\mathscr{D}_{1.5}$ (Taylor--Hood approximation) and $\mathcal{D}_{1.5,\mathrm{stab}}$ (low--order stabilized approximation). We present the finite element approximations of $|\bu_{\T}|$ and $p_{\T}$ and the final meshes obtained by the aforementioned schemes. It can be observed that the proposed error estimators attained optimal experimental rates of convergence, under the natural proposed modification of the involved Muckenhoupt weight. 
\subsection{Example 4: The fundamental solution of a Stokes flow}

In order to measure the experimental rates of convergence for the total error \eqref{eq:error_total}, we invoke the fundamental solution of the Stokes problem; even when we violate the assumption of imposing homogeneous boundary conditions. For a delta source $\delta_{z}$, located at $z=(x_{0} \, , \, y_{0})^\intercal \in \Omega$, and a given constant vector $\bF \in\mathbb{R}^{2}$, we recall the fundamental solution for the Stokes problem \eqref{eq:StokesStrong} when $d=2$:
 \begin{equation}
  \label{eq:fundamental_solution}
  \bu(x,y):=\tilde{\bT}\cdot \bF,\quad p(x,y):=\bT\cdot\bF,
 \end{equation}
 where, if $\bx_{0}=(x-x_{0} , y-y_{0})^\intercal$, then
 \begin{gather*}
 \tilde{\bT}=-\frac{1}{4\pi}
 \left(
 \log|\bx_{0}|
 \left[
 \begin{array}{cc}
 1 & 0\\
 0 & 1
 \end{array}
 \right]
 -
 \frac{1}{|\bx_{0}|^{2}}
 \left[
 \begin{array}{cc}
 (x-x_{0})^{2} & (x-x_{0})(y-y_{0})\\
 (x-x_{0})(y-y_{0}) & (y-y_{0})^{2}
 \end{array}
 \right]
 \right),
 \\
 \bT =
 \frac{\bx_{0}}{2\pi|\bx_{0}|^{2}}.
 \end{gather*}

We consider $\Omega=(0,1)^{2}$, $z=(0.5,0.5)^\intercal$ and $\bF=(1 ,1)^\intercal$ in problem \eqref{eq:StokesStrong}. We fix the exponent of the Muckenhoupt weight in \eqref{distance_A2} as $\alpha=1.5$. The solution of this problem is thus given by \eqref{eq:fundamental_solution}.

In Figure \ref{fig:ex4th-st}, we present the finite element approximations of $|\bu_{\T}|$ and $p_{\T}$ which were obtained after $10$ adaptive refinements, together with the final mesh. We also present the experimental rates of convergence for the total error $\|(\boldsymbol{e}_{\bu},e_{p})\|_{\mathcal{X}}$ and the error estimators $\E_{1.5}$ and $\mathcal{E}_{\alpha,\mathrm{stab}}$. It can be observed that optimal experimental rates of convergence are attained and that most of the adaptive refinement is concentrated around the delta source. Finally, in Figure \ref{fig:ex4th-st-alphas}, we present the experimental rate of convergence for the total error estimators $\E_{\alpha}$ and $\mathcal{E}_{\alpha,\mathrm{stab}}$ when $\alpha\in\{0.5,0.75,1,1.25,1.5\}$. It can be observed that, for all the cases that we have considered, optimal rates of convergence are attained.
 \begin{figure}[ht]
 \begin{flushleft}
 \begin{minipage}{0.25\textwidth}\centering
 \includegraphics[trim={0 0 0 0},clip,width=3cm,height=3cm,scale=0.7]{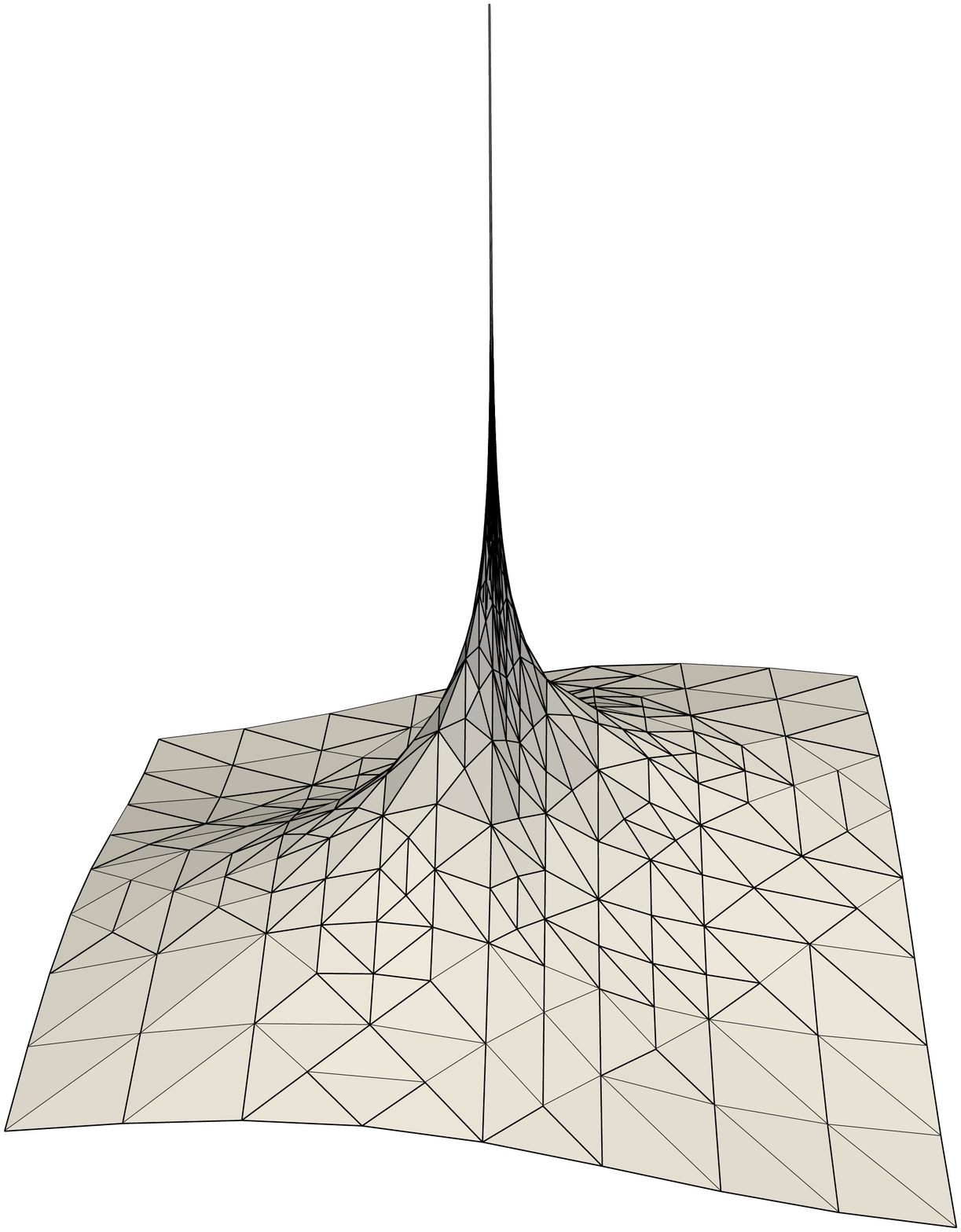}\\
 \tiny{$|\bu_{\T}|$}\\
 \includegraphics[trim={0 0 0 0},clip,width=3cm,height=3cm,scale=0.7]{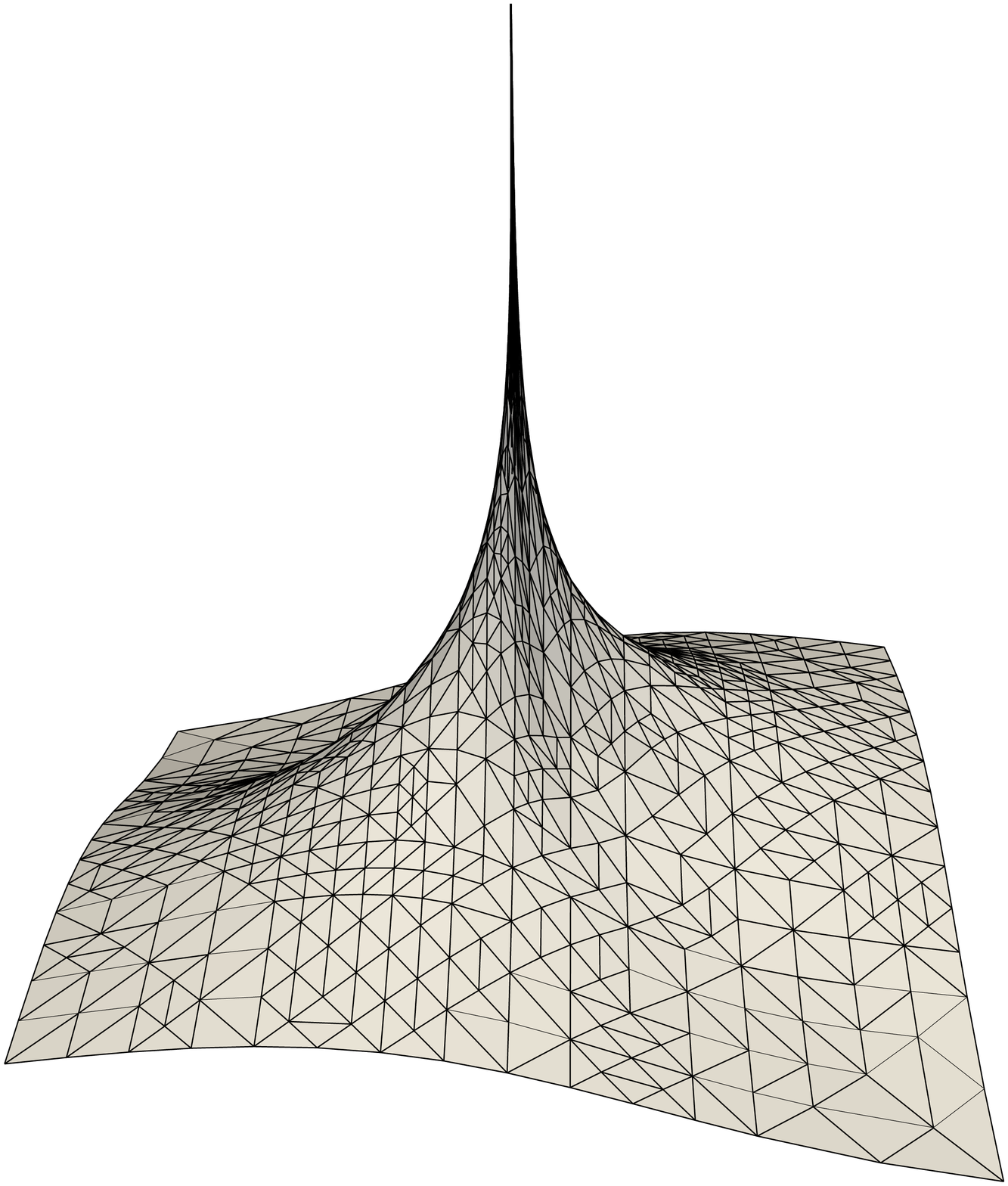}\\
 \tiny{$|\bu_{\T}|$}
 \end{minipage}
 \begin{minipage}{0.25\textwidth}\centering
 \includegraphics[trim={0 0 0 0},clip,width=3cm,height=3cm,scale=0.7]{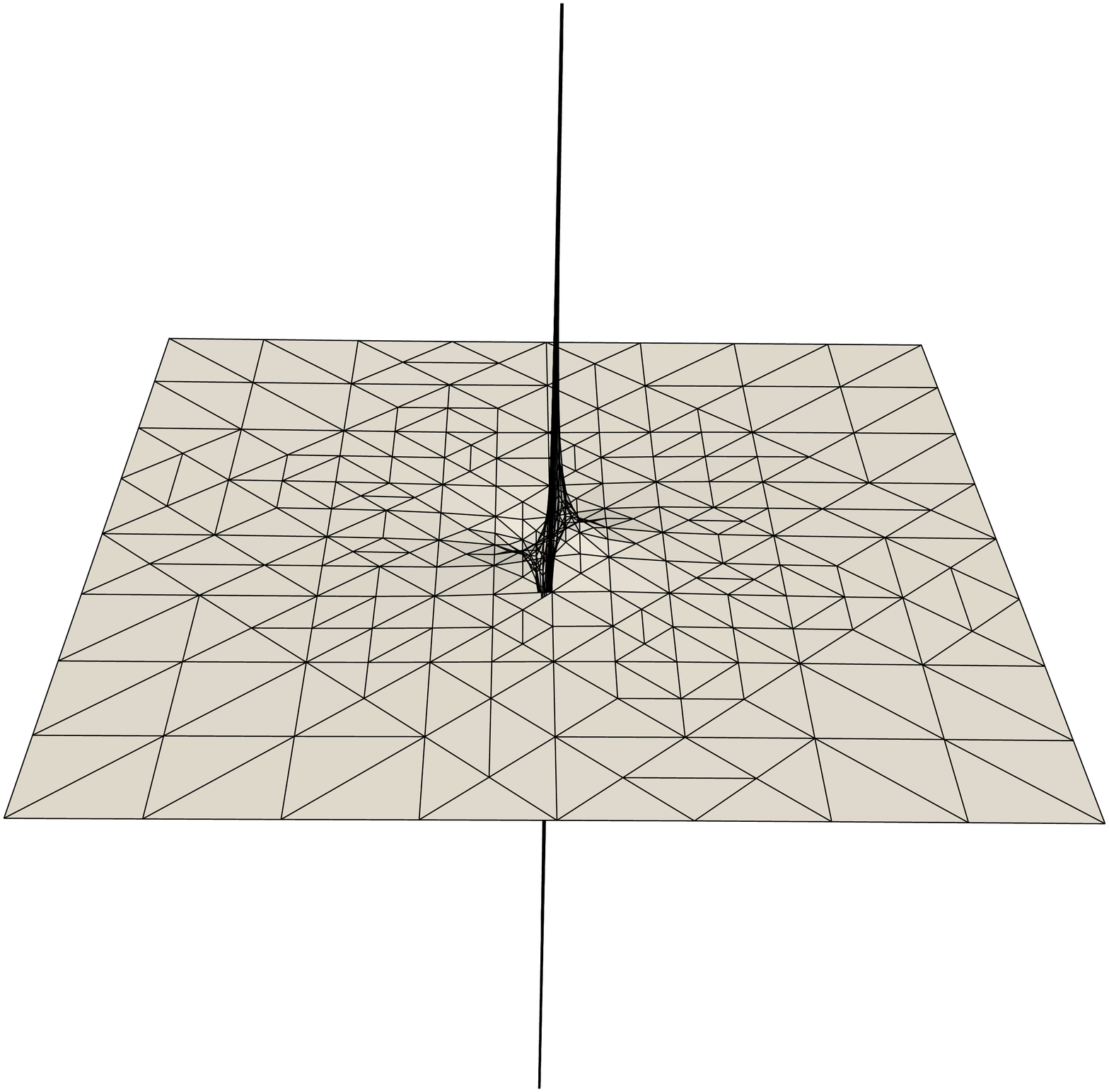}\\
 \tiny{$|p_{\T}|$}\\
 \includegraphics[trim={0 0 0 0},clip,width=3cm,height=3cm,scale=0.7]{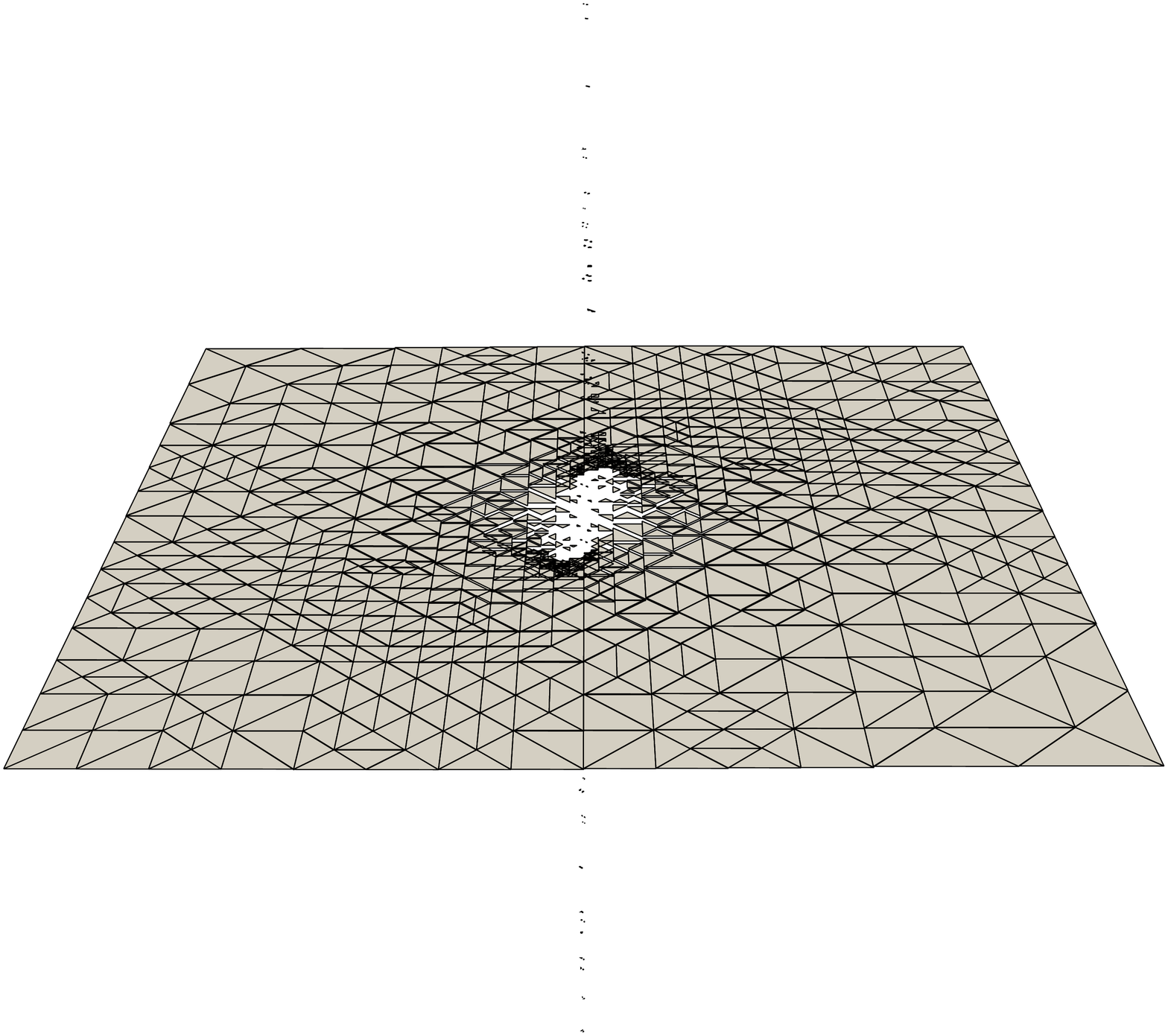}\\
 \tiny{$|p_{\T}|$}
 \end{minipage}
 \begin{minipage}{0.25\textwidth}\centering
 \includegraphics[trim={0 0 0 0},clip,width=3cm,height=3cm,scale=0.7]{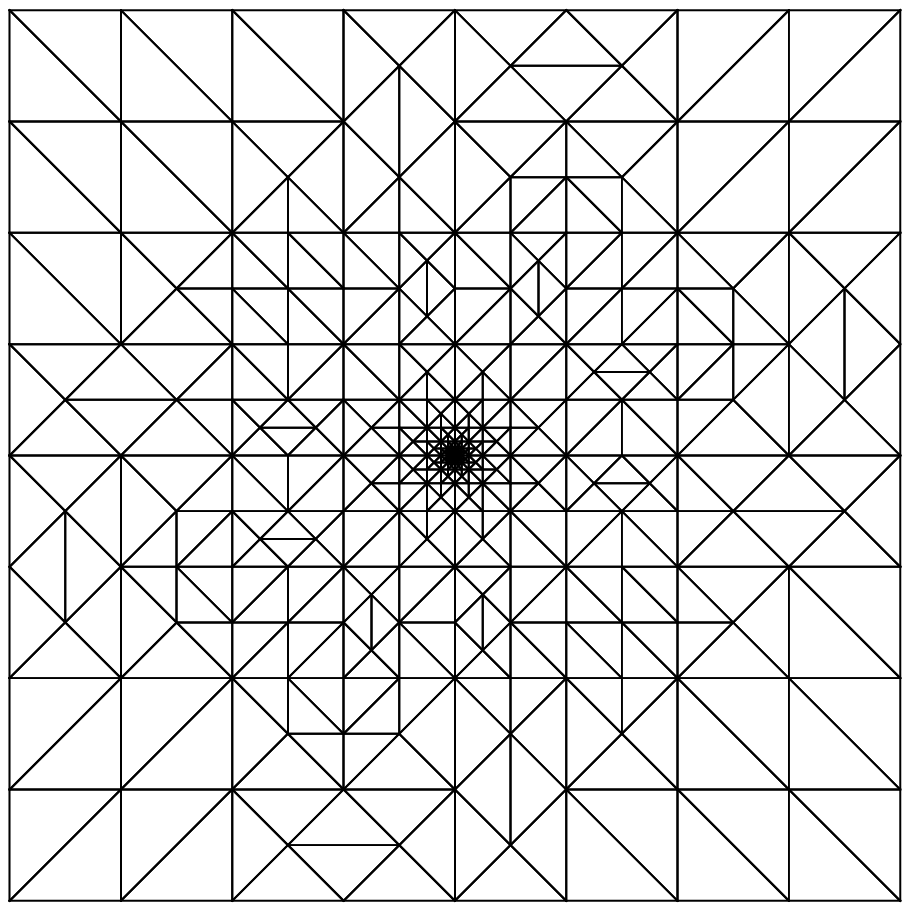}\\
 \tiny{$10^{th}$ adaptive refinement}\\
 \includegraphics[trim={0 0 0 0},clip,width=3cm,height=3cm,scale=0.7]{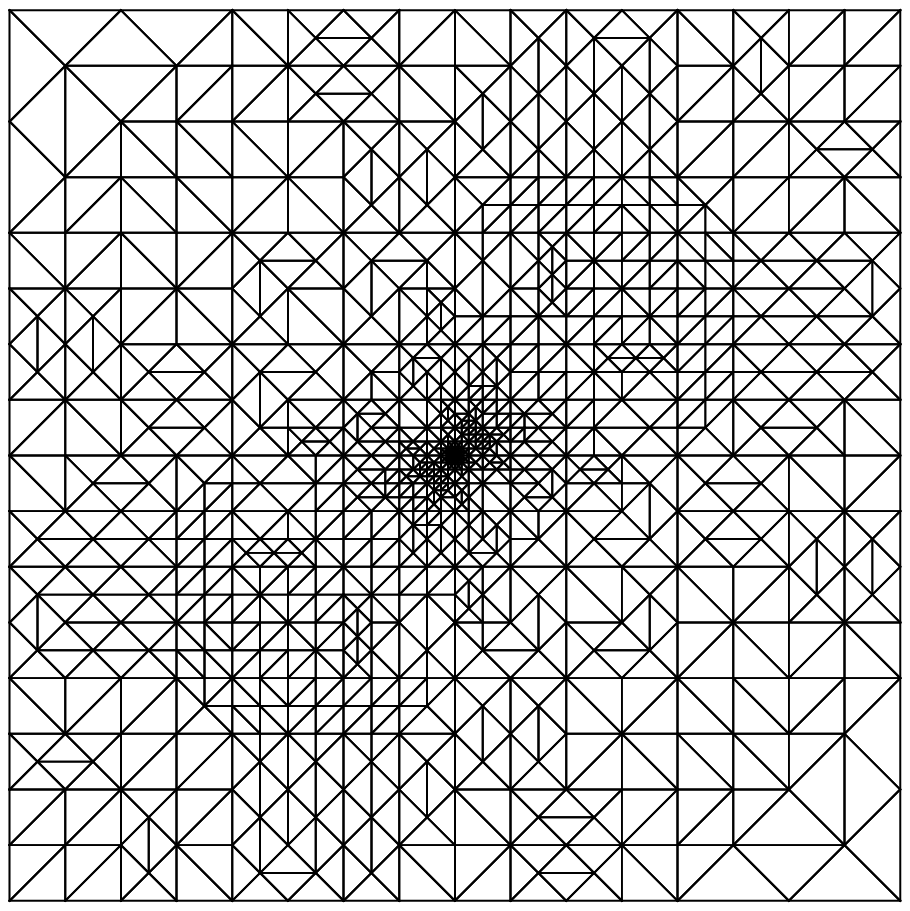}\\
 \tiny{$10^{th}$ adaptive refinement}
 \end{minipage}
 \begin{minipage}{0.25\textwidth}\centering
 \psfrag{error tot}{\large $\|(\boldsymbol{e}_{\bu},e_{p})\|_{\mathcal{X}}$}
 \psfrag{estimator}{\large $\E_{1.5}$}
 \psfrag{rate(h2)}{$\textsf{Ndof}^{-1}$}
 \includegraphics[trim={0 0 0 0},clip,width=3cm,height=3cm,scale=0.7]{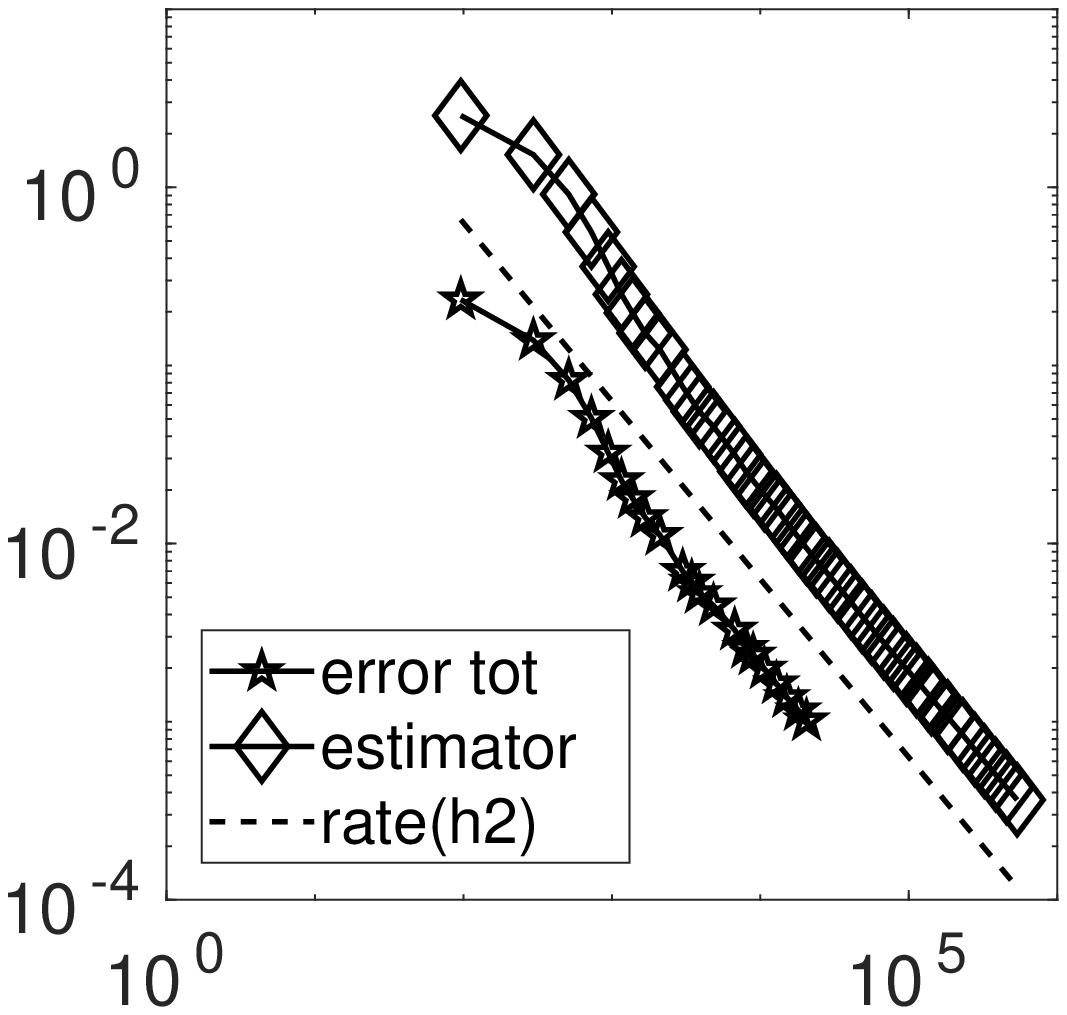}\\
 \psfrag{estimator}{\large $\mathcal{E}_{1.5,\mathrm{stab}}$}
 \psfrag{rate(h2)}{$\textsf{Ndof}^{-1/2}$}
 \includegraphics[trim={0 0 0 0},clip,width=3cm,height=3cm,scale=0.7]{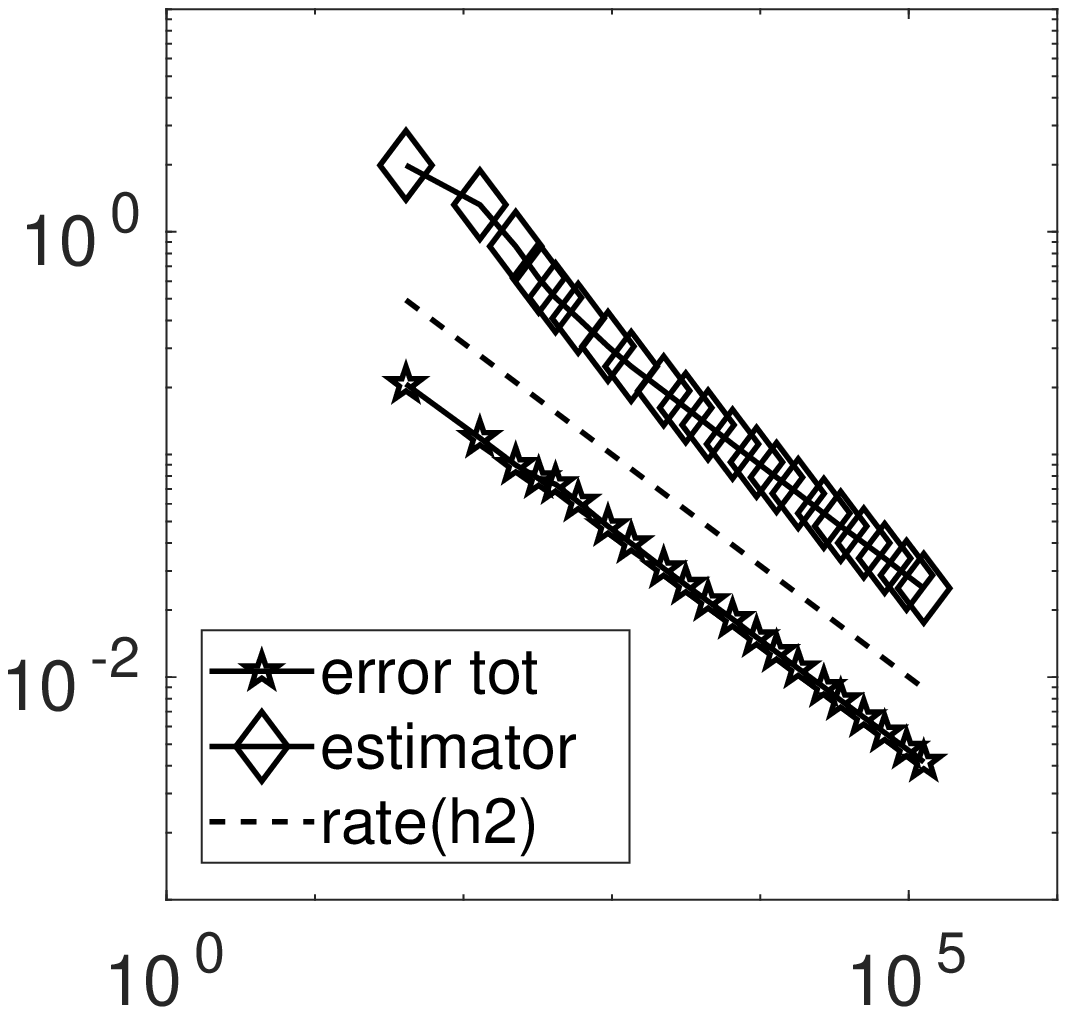}
 \end{minipage}
 \end{flushleft}
 \caption{Example 4: Finite element approximation of $|\bu_{\T}|$ and $p_{\T}$, the mesh obtained after $10$ adaptive refinements, and the experimental rate of convergence for the error estimator when Taylor--Hood approximation is used (top) and when the low--order stabilized approximation is considered (bottom) where $\alpha \in \{0.5,0.75,1,1.25,1.5 \}$.} 
 \label{fig:ex4th-st}
 \end{figure}
 
 \begin{figure}[h]
 \centering
 \psfrag{estima-1}{\large $\alpha=0.5$}
 \psfrag{estima-2}{\large $\alpha=0.75$}
 \psfrag{estima-3}{\large $\alpha=1$}
 \psfrag{estima-4}{\large $\alpha=1.25$}
 \psfrag{estima-5}{\large $\alpha=1.5$}
 \psfrag{rate(h2)}{$\textsf{Ndof}^{-1}$}
 \psfrag{error-1}{\large $\alpha=0.5$}
 \psfrag{error-2}{\large $\alpha=0.75$}
 \psfrag{error-3}{\large $\alpha=1$}
 \psfrag{error-4}{\large $\alpha=1.25$}
 \psfrag{error-5}{\large $\alpha=1.5$}
 \psfrag{errorrrr}{\large $\|(\boldsymbol{e}_{\bu},e_{p})\|_{\mathcal{X}}(\alpha)$}
  \psfrag{error varing alpha}{\Large $\|(\boldsymbol{e}_{\bu},e_{p})\|_{\mathcal{X}}$ varying $\alpha$}
   \psfrag{est varing alpha}{\Large $\E_{\alpha}$ varying $\alpha$} 
 \includegraphics[trim={0 0 0 0},clip,width=5cm,height=5cm,scale=0.6]{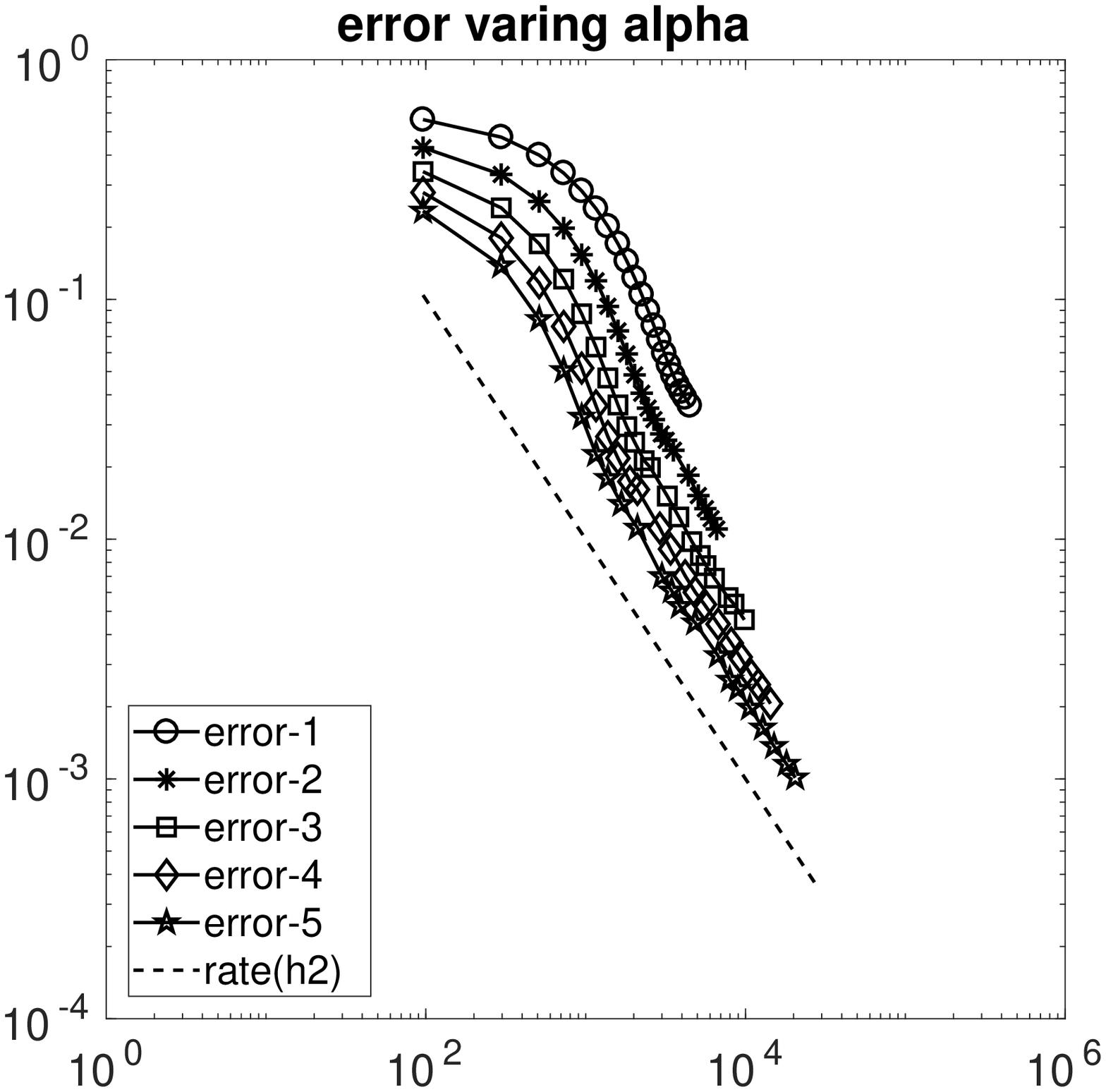}
 \psfrag{estimatorrrr}{\large $\E(\alpha)$}
 \includegraphics[trim={0 0 0 0},clip,width=5cm,height=5cm,scale=0.6]{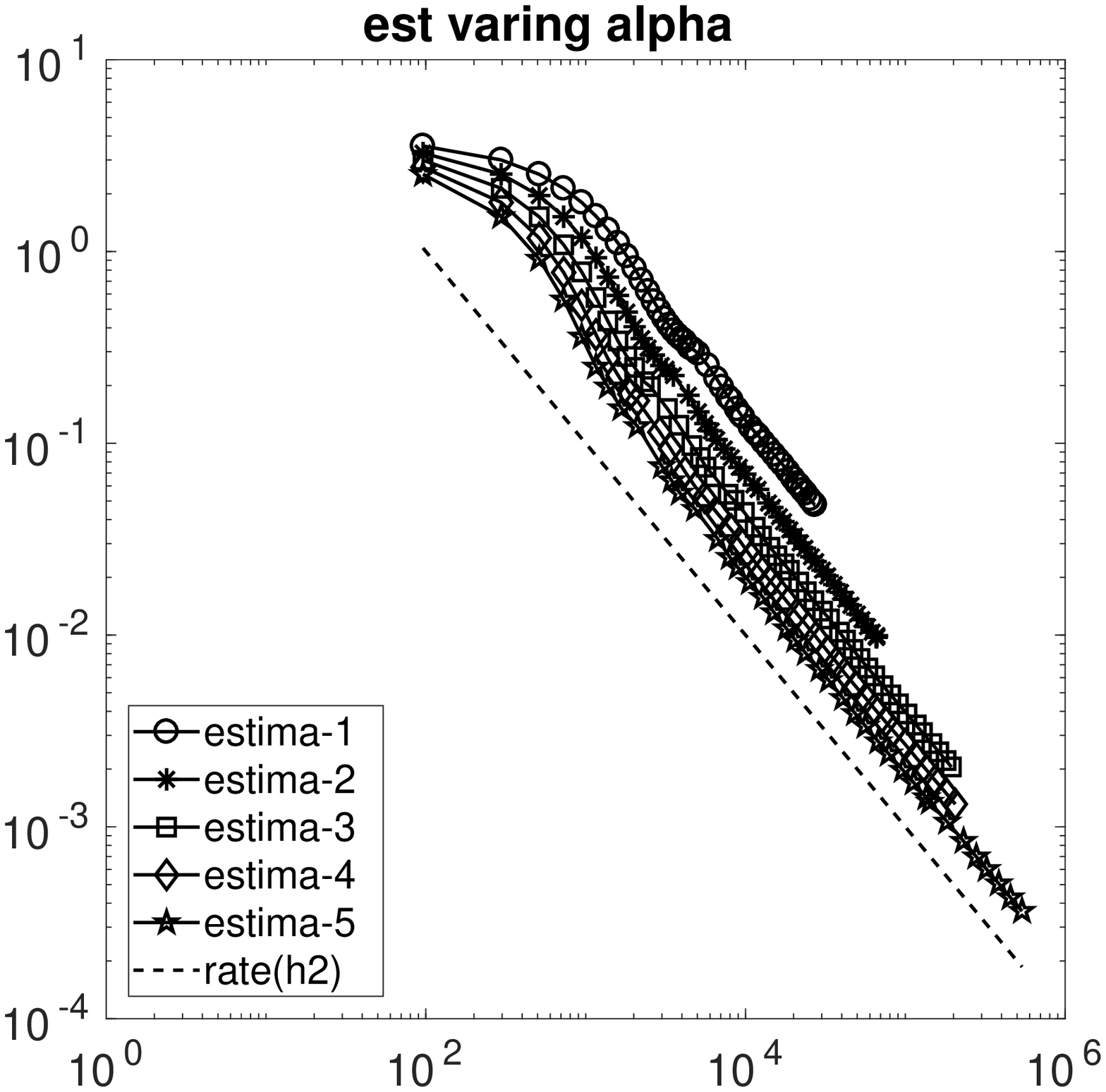}\\
  \psfrag{est varing alpha}{\Large $\mathcal{E}_{\alpha,\mathrm{stab}}$ varying $\alpha$} 
 \psfrag{rate(h2)}{$\textsf{Ndof}^{-1/2}$}
 \includegraphics[trim={0 0 0 0},clip,width=5cm,height=5cm,scale=0.6]{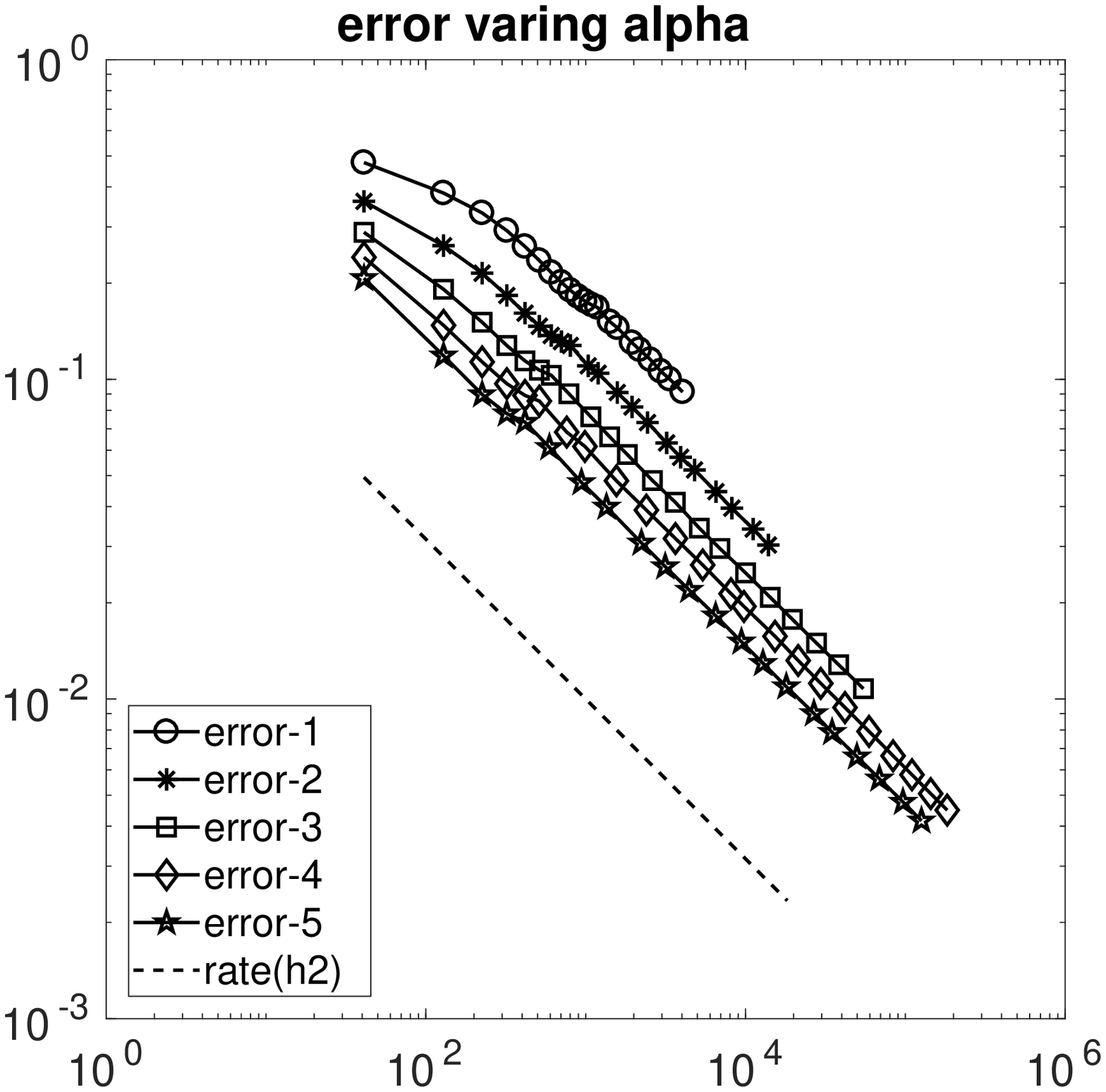}
 \psfrag{estimatorrrr}{\large $\E(\alpha,\mathrm{stab})$}
 \includegraphics[trim={0 0 0 0},clip,width=5cm,height=5cm,scale=0.6]{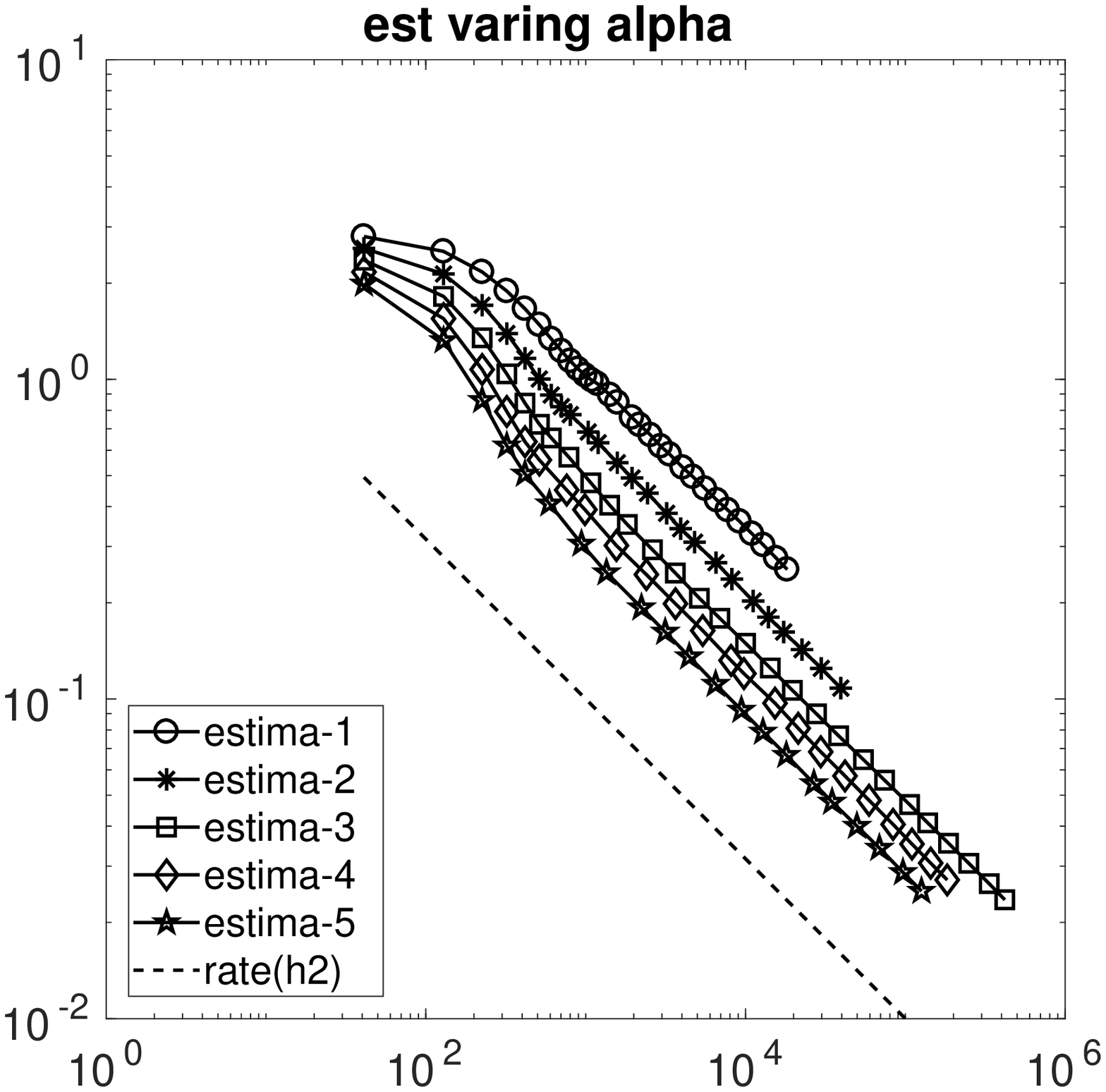}
 \caption{Example 4: Experimental rates of convergence for the total error $\|(\boldsymbol{e}_{\bu},e_{p})\|_{\mathcal{X}}$ and error estimators $\E_{\alpha}$ (Taylor--Hood approximation) and $\mathcal{E}_{\alpha,\mathrm{stab}}$ (low--order stabilized approximation).}
 \label{fig:ex4th-st-alphas}
 \end{figure}

\bibliographystyle{siamplain}
\bibliography{biblio}

\begin{thebibliography}{10}

\bibitem{AGM}
{\sc J.~P. Agnelli, E.~M. Garau, and P.~Morin}, {\em {\it {A} posteriori} error
  estimates for elliptic problems with {D}irac measure terms in weighted
  spaces}, ESAIM Math. Model. Numer. Anal., 48 (2014), pp.~1557--1581,
  \url{http://dx.doi.org/10.1051/m2an/2014010}.

\bibitem{MR3215609}
{\sc H.~Aimar, M.~Carena, R.~Dur\'an, and M.~Toschi}, {\em Powers of distances
  to lower dimensional sets as {M}uckenhoupt weights}, Acta Math. Hungar., 143
  (2014), pp.~119--137, \url{http://dx.doi.org/10.1007/s10474-014-0389-1}.

\bibitem{MR3679932}
{\sc A.~Allendes, E.~Ot\'arola, R.~Rankin, and A.~J. Salgado}, {\em Adaptive
  finite element methods for an optimal control problem involving {D}irac
  measures}, Numer. Math., 137 (2017), pp.~159--197,
  \url{http://dx.doi.org/10.1007/s00211-017-0867-9}.

\bibitem{MR1856597}
{\sc P.~R. Amestoy, I.~S. Duff, J.-Y. L'Excellent, and J.~Koster}, {\em A fully
  asynchronous multifrontal solver using distributed dynamic scheduling}, SIAM
  J. Matrix Anal. Appl., 23 (2001), pp.~15--41 (electronic),
  \url{http://dx.doi.org/10.1137/S0895479899358194}.

\bibitem{MR2202663}
{\sc P.~R. Amestoy, A.~Guermouche, J.-Y. L'Excellent, and S.~Pralet}, {\em
  Hybrid scheduling for the parallel solution of linear systems}, Parallel
  Comput., 32 (2006), pp.~136--156,
  \url{http://dx.doi.org/10.1016/j.parco.2005.07.004}.

\bibitem{MR799997}
{\sc D.~N. Arnold, F.~Brezzi, and M.~Fortin}, {\em A stable finite element for
  the {S}tokes equations}, Calcolo, 21 (1984), pp.~337--344 (1985),
  \url{https://doi.org/10.1007/BF02576171}.

\bibitem{MR972452}
{\sc C.~Bernardi, C.~Canuto, and Y.~Maday}, {\em Generalized inf-sup conditions
  for {C}hebyshev spectral approximation of the {S}tokes problem}, SIAM J.
  Numer. Anal., 25 (1988), pp.~1237--1271,
  \url{http://dx.doi.org/10.1137/0725070}.

\bibitem{bochev2009least}
{\sc P.~B. Bochev and M.~D. Gunzburger}, {\em Least-squares finite element
  methods}, vol.~166 of Applied Mathematical Sciences, Springer, New York,
  2009, \url{https://doi.org/10.1007/b13382}.

\bibitem{BrettElliott}
{\sc C.~Brett, A.~Dedner, and C.~Elliott}, {\em Optimal control of elliptic
  {PDE}s at points}, IMA J. Numer. Anal., 36 (2016), pp.~1015--1050,
  \url{http://dx.doi.org/10.1093/imanum/drv040}.

\bibitem{MR3582412}
{\sc M.~Bul{\'\i}{\v c}ek, J.~Burczak, and S.~Schwarzacher}, {\em A unified
  theory for some non-{N}ewtonian fluids under singular forcing}, SIAM J. Math.
  Anal., 48 (2016), pp.~4241--4267, \url{https://doi.org/10.1137/16M1073881}.

\bibitem{MR3449612}
{\sc L.~Chang, W.~Gong, and N.~Yan}, {\em Numerical analysis for the
  approximation of optimal control problems with pointwise observations}, Math.
  Methods Appl. Sci., 38 (2015), pp.~4502--4520,
  \url{http://dx.doi.org/10.1002/mma.2861}.

\bibitem{CiarletBook}
{\sc P.~G. Ciarlet}, {\em The finite element method for elliptic problems},
  SIAM, Philadelphia, PA, 2002,
  \url{http://dx.doi.org/10.1137/1.9780898719208}.

\bibitem{MR0400739}
{\sc P.~Cl\'ement}, {\em Approximation by finite element functions using local
  regularization}, Rev. Fran\c{c}aise Automat. Informat. Recherche
  Op\'erationnelle S\'er., 9 (1975), pp.~77--84.

\bibitem{MR1800316}
{\sc J.~Duoandikoetxea}, {\em Fourier analysis}, vol.~29 of Graduate Studies in
  Mathematics, American Mathematical Society, Providence, RI, 2001.
\newblock Translated and revised from the 1995 Spanish original by David
  Cruz-Uribe.

\bibitem{MR2164092}
{\sc R.~G. Dur\'an and A.~L. Lombardi}, {\em Error estimates on anisotropic
  {$Q_1$} elements for functions in weighted {S}obolev spaces}, Math. Comp., 74
  (2005), pp.~1679--1706, \url{https://doi.org/10.1090/S0025-5718-05-01732-1}.

\bibitem{Guermond-Ern}
{\sc A.~Ern and J.-L. Guermond}, {\em Theory and practice of finite elements},
  vol.~159 of Applied Mathematical Sciences, Springer-Verlag, New York, 2004.

\bibitem{MR643158}
{\sc E.~B. Fabes, C.~E. Kenig, and R.~P. Serapioni}, {\em The local regularity
  of solutions of degenerate elliptic equations}, Comm. Partial Differential
  Equations, 7 (1982), pp.~77--116,
  \url{https://doi.org/10.1080/03605308208820218}.

\bibitem{MR1601373}
{\sc R.~Farwig and H.~Sohr}, {\em Weighted {$L^q$}-theory for the {S}tokes
  resolvent in exterior domains}, J. Math. Soc. Japan, 49 (1997), pp.~251--288,
  \url{https://doi.org/10.2969/jmsj/04920251}.

\bibitem{MR2808162}
{\sc G.~P. Galdi}, {\em An introduction to the mathematical theory of the
  {N}avier-{S}tokes equations}, Springer Monographs in Mathematics, Springer,
  New York, second~ed., 2011, \url{https://doi.org/10.1007/978-0-387-09620-9}.
\newblock Steady-state problems.

\bibitem{MR851383}
{\sc V.~Girault and P.-A. Raviart}, {\em Finite element methods for
  {N}avier-{S}tokes equations}, vol.~5 of Springer Series in Computational
  Mathematics, Springer-Verlag, Berlin, 1986,
  \url{https://doi.org/10.1007/978-3-642-61623-5}.
\newblock Theory and algorithms.

\bibitem{MR2491902}
{\sc V.~Gol'dshtein and A.~Ukhlov}, {\em Weighted {S}obolev spaces and
  embedding theorems}, Trans. Amer. Math. Soc., 361 (2009), pp.~3829--3850,
  \url{http://dx.doi.org/10.1090/S0002-9947-09-04615-7}.

\bibitem{hood1974navier}
{\sc P.~Hood and C.~Taylor}, {\em Navier-{S}tokes equations using mixed
  interpolation}, Finite element methods in flow problems,  (1974),
  pp.~121--132.

\bibitem{hughes1986new}
{\sc T.~J.~R. Hughes, L.~P. Franca, and M.~Balestra}, {\em A new finite element
  formulation for computational fluid dynamics. {V}. {C}ircumventing the
  {B}abu\v ska-{B}rezzi condition: a stable {P}etrov-{G}alerkin formulation of
  the {S}tokes problem accommodating equal-order interpolations}, Comput.
  Methods Appl. Mech. Engrg., 59 (1986), pp.~85--99,
  \url{https://doi.org/10.1016/0045-7825(86)90025-3}.

\bibitem{john2016finite}
{\sc V.~John}, {\em Finite element methods for incompressible flow problems},
  vol.~51 of Springer Series in Computational Mathematics, Springer, Cham,
  2016, \url{https://doi.org/10.1007/978-3-319-45750-5}.

\bibitem{MR1740398}
{\sc D.~Kay and D.~Silvester}, {\em A posteriori error estimation for
  stabilized mixed approximations of the {S}tokes equations}, SIAM J. Sci.
  Comput., 21 (1999/00), pp.~1321--1336.

\bibitem{KMR}
{\sc V.~Kozlov, V.~Maz'ya, and J.~Rossmann}, {\em {Elliptic boundary value
  problems in domains with point singularities}}, American Mathematical
  Society, Providence, Rhode Island, USA, 1997.

\bibitem{Lacouture2015187}
{\sc L.~Lacouture}, {\em A numerical method to solve the stokes problem with a
  punctual force in source term}, Comptes Rendus Mécanique, 343 (2015),
  pp.~187 -- 191,
  \url{http://dx.doi.org/http://dx.doi.org/10.1016/j.crme.2014.09.008}.

\bibitem{MR0293384}
{\sc B.~Muckenhoupt}, {\em Weighted norm inequalities for the {H}ardy maximal
  function}, Trans. Amer. Math. Soc., 165 (1972), pp.~207--226,
  \url{https://doi.org/10.2307/1995882}.

\bibitem{MR0163054}
{\sc J.~Ne{\v c}as}, {\em Sur une m\'ethode pour r\'esoudre les \'equations aux
  d\'eriv\'ees partielles du type elliptique, voisine de la variationnelle},
  Ann. Scuola Norm. Sup. Pisa (3), 16 (1962), pp.~305--326.

\bibitem{MR0227584}
{\sc J.~r. Ne{\v c}as}, {\em Les m\'ethodes directes en th\'eorie des
  \'equations elliptiques}, Masson et Cie, \'Editeurs, Paris; Academia,
  \'Editeurs, Prague, 1967.

\bibitem{NOS3}
{\sc R.~H. Nochetto, E.~Ot\'arola, and A.~J. Salgado}, {\em Piecewise
  polynomial interpolation in {M}uckenhoupt weighted {S}obolev spaces and
  applications}, Numer. Math., 132 (2016), pp.~85--130,
  \url{http://dx.doi.org/10.1007/s00211-015-0709-6}.

\bibitem{NSV:09}
{\sc R.~H. Nochetto, K.~G. Siebert, and A.~Veeser}, {\em Theory of adaptive
  finite element methods: an introduction}, in Multiscale, nonlinear and
  adaptive approximation, Springer, 2009,
  \url{http://dx.doi.org/10.1007/978-3-642-03413-8_12}.

\bibitem{OS:17infsup}
{\sc E.~Ot\'arola and A.~J. Salgado}, {\em The {P}oisson problem in noncovex,
  {L}ipschitz polytopes.}
\newblock arXiv:1711.08542.

\bibitem{MR2454024}
{\sc H.-G. Roos, M.~Stynes, and L.~Tobiska}, {\em Robust numerical methods for
  singularly perturbed differential equations}, vol.~24 of Springer Series in
  Computational Mathematics, Springer-Verlag, Berlin, second~ed., 2008.
\newblock Convection-diffusion-reaction and flow problems.

\bibitem{MR1011446}
{\sc L.~R. Scott and S.~Zhang}, {\em Finite element interpolation of nonsmooth
  functions satisfying boundary conditions}, Math. Comp., 54 (1990),
  pp.~483--493, \url{https://doi.org/10.2307/2008497}.

\bibitem{MR1774162}
{\sc B.~O. Turesson}, {\em Nonlinear potential theory and weighted {S}obolev
  spaces}, vol.~1736 of Lecture Notes in Mathematics, Springer-Verlag, Berlin,
  2000, \url{http://dx.doi.org/10.1007/BFb0103908}.

\bibitem{MR993474}
{\sc R.~Verf\"urth}, {\em A posteriori error estimators for the {S}tokes
  equations}, Numer. Math., 55 (1989), pp.~309--325,
  \url{https://doi.org/10.1007/BF01390056}.

\bibitem{Verfurth}
{\sc R.~Verf\"urth}, {\em A posteriori error estimation techniques for finite
  element methods}, Numerical Mathematics and Scientific Computation, Oxford
  University Press, Oxford, 2013,
  \url{https://doi.org/10.1093/acprof:oso/9780199679423.001.0001}.

\end{thebibliography}
\end{document}